\documentclass[11pt,longbibliography]{article}

\usepackage[a4paper, left=2.7cm, right=2.7cm, top=2.7cm, bottom=3.3cm]{geometry}

\usepackage{amsmath, amsthm, amsfonts, amssymb}
\usepackage{mathrsfs} 
\usepackage{graphics}
\usepackage{graphicx}

\usepackage{bbm} 

\usepackage{microtype}
\numberwithin{figure}{section}

\newtheorem{theorem}{Theorem}[section]

\newtheorem{lemma}[theorem]{Lemma}

\newtheorem*{decomp-thm-GUT}{Theorem~\ref{decomp-thm-GUT}} 
\newtheorem*{decomp-thm-GU}{Theorem~\ref{decomp-thm-GU}} 
\newtheorem*{decomp-thm-GT}{Theorem~\ref{decomp-thm-GT}} 
\newtheorem*{decomp-thm-GUTcap}{Theorem~\ref{decomp-thm-GUTcap}}

\title{Clique-cutsets beyond chordal graphs}

\author{Valerio Boncompagni\thanks{School of Computing, University of Leeds, Leeds LS2 9JT, UK. Email: \texttt{scvb@leeds.ac.uk}.} \and Irena Penev\thanks{School of Computing, University of Leeds, Leeds LS2 9JT, UK. Email: \texttt{i.penev@leeds.ac.uk}. Partially supported by EPSRC grant EP/N0196660/1.} \and Kristina Vu\v{s}kovi\'{c}\thanks{School of Computing, University of Leeds, Leeds LS2 9JT, UK. Email: \texttt{k.vuskovic@leeds.ac.uk}. Partially supported by EPSRC grants EP/K016423/1 and EP/N0196660/1. Partially supported by Serbian Ministry of Education and Science projects 174033 and III44006.}}

\begin{document}

\maketitle

\begin{abstract}
 
Truemper configurations (thetas, pyramids, prisms, and wheels) have played an important role in the study of complex hereditary graph classes (e.g.\ the class of perfect graphs and the class of even-hole-free graphs), appearing both as excluded configurations, and as configurations around which graphs can be decomposed. In this paper, we study the structure of graphs that contain (as induced subgraphs) no Truemper configurations other than (possibly) universal wheels and twin wheels. We also study several subclasses of this class. We use our structural results to analyze the complexity of the recognition, maximum weight clique, maximum weight stable set, and optimal vertex coloring problems for these classes. Furthermore, we obtain polynomial $\chi$-bounding functions for these classes. 
\end{abstract}

\section{Introduction}\label{sec1}

All graphs in this paper are finite, simple, and nonnull. We say that a graph $G$ {\em contains} a graph $H$ if $H$ is isomorphic to an induced subgraph of $G$; $G$ is {\em $H$-free} if $G$ does not contain $H$. For a family of graphs ${\cal H}$, we say that $G$ is {\em ${\cal H}$-free} if $G$ is $H$-free for every $H \in {\cal H}$. A class of graphs is \emph{hereditary} if for every graph $G$ in the class, all (isomorphic copies of) induced subgraphs of $G$ belong to the class. Note that a class $\mathcal{G}$ is hereditary if and only if there exists a family $\mathcal{H}$ such that $\mathcal{G}$ is precisely the class of $\mathcal{H}$-free graphs (the ``if'' part is obvious; for the ``only if'' part, we can take $\mathcal{H}$ to be the collection of all graphs that do not belong to $\mathcal{G}$, but all of whose proper induced subgraphs do belong to $\mathcal{G}$). 

Configurations known as thetas, pyramids, prisms, and wheels (defined below) have played an important role in the study of such diverse (and important) classes as the classes of regular matroids, balanceable matrices, perfect graphs, and even-hole-free graphs (for a survey, see~\cite{Truemper-survey}). These configurations are also called {\em Truemper configurations}, as they appear in a theorem due to Truemper \cite{truemper} that characterizes graphs whose edges can be labeled so that all induced cycles have prescribed parities. In this paper, we study various classes of graphs that are defined by excluding certain Truemper configurations. 

A {\em hole} is an induced cycle on at least four vertices, and an {\em antihole} is the complement of a hole. The {\em length} of a hole or antihole is the number of vertices that it contains. A hole or antihole is {\em long} if it is of length at least five. A hole or antihole is {\em odd} (resp. {\em even}) if its length is odd (resp. even). For an integer $k \geq 4$, a {\em $k$-hole} (resp. {\em $k$-antihole}) is a hole (resp. antihole) of length $k$. 

A {\em theta} is any subdivision of the complete bipartite graph $K_{2,3}$; in particular, $K_{2,3}$ is a theta. A {\em pyramid} is any subdivision of the complete graph $K_4$ in which one triangle remains unsubdivided, and of the remaining three edges, at least two edges are subdivided at least once. A {\em prism} is any subdivision of $\overline{C_6}$ (where $\overline{C_6}$ is the complement of $C_6$) in which the two triangles remain unsubdivided; in particular, $\overline{C_6}$ is a prism. A {\em three-path-configuration} (or {\em 3PC} for short) is any theta, pyramid, or prism; the three types of 3PC are represented in Figure~\ref{fig:Truemper}. 

\begin{figure} 
\begin{center}
\includegraphics[scale=0.6]{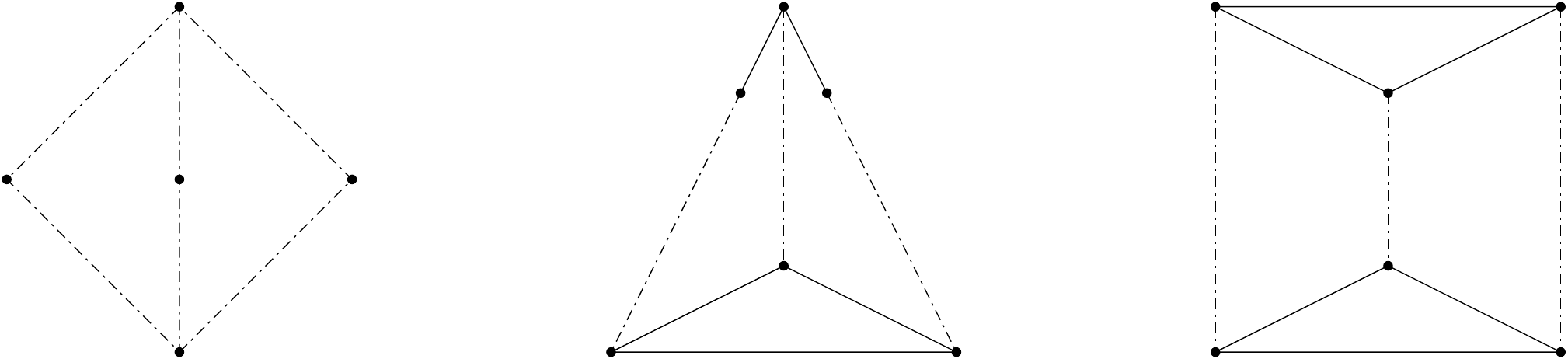}
\end{center} 
\caption{Theta, pyramid, and prism.
(A full line represents an edge, and a dashed line represents a path that has at least one edge.)} \label{fig:Truemper} 
\end{figure}

A {\em wheel} is a graph that consists of a hole and an additional vertex that has at least three neighbors in the hole. If this additional vertex is adjacent to all vertices of the hole, then the wheel is said to be a {\em universal wheel}; if the additional vertex is adjacent to three consecutive vertices of the hole, and to no other vertices of the hole, then the wheel is said to be a {\em twin wheel}. For $k \geq 4$, the universal wheel on $k+1$ vertices is denoted by $W_k$, and the twin wheel on $k+1$ vertices is denoted by $W_k^{\text{t}}$. A {\em proper wheel} is a wheel that is neither a universal wheel nor a twin wheel. Note that every proper wheel has at least six vertices. 

A {\em Truemper configuration} is any 3PC or wheel. Note that every Truemper configuration contains a hole. Note, furthermore, that every prism or theta contains an even hole, and every pyramid contains an odd hole. Thus, even-hole-free graphs contain no prisms and no thetas, and odd-hole-free graphs contain no pyramids. 

As usual, given a graph $G$, we denote by $\chi(G)$ the chromatic number of $G$, by $\omega(G)$ the clique number (i.e.\ the maximum size of a clique) of $G$, and by $\alpha(G)$ the stability number (i.e.\ the maximum size of a stable set) of $G$. A graph $G$ is {\em perfect} if all its induced subgraphs $H$ satisfy $\chi(H) = \omega(H)$. A graph is {\em Berge} if it contains no odd holes and no odd antiholes. The famous Strong Perfect Graph Theorem~\cite{SPGT} states that a graph is perfect if and only if it is Berge. The main ingredient of the proof of the Strong Perfect Graph Theorem is a decomposition theorem for Berge graphs; wheels play a particularly important role (as configurations around which graphs can be decomposed) in the proof of this decomposition theorem. Since perfect graphs are odd-hole-free, we see that perfect graphs contain no pyramids; in fact, detection of pyramids plays an important role in the polynomial time recognition algorithm for Berge (equivalently: perfect) graphs~\cite{BergeRec}. 

A graph is {\em chordal} if it contains no holes. Clearly, every Truemper configuration contains a hole, and consequently, chordal graphs contain no Truemper configurations. A {\em clique-cutset} of a graph $G$ is a (possibly empty) clique $C$ such that $G \setminus C$ is disconnected. 

\begin{theorem}\cite{Dirac61} \label{thm-Dirac61} If $G$ is a chordal graph, then either $G$ is a complete graph, or $G$ admits a clique-cutset. Furthermore, chordal graphs are perfect. 
\end{theorem} 

A graph $G$ is {\em universally signable} if for every prescription of parities to the holes of $G$, there exists an assignment of zero or one weights to the edges of $G$ such that for each hole, the sum of weights of its edges has prescribed parity, and for every triangle, the sum of weights of its edges is odd. Clearly, every chordal graph is universally signable: we simply assign weight one to each edge. Note, however, that holes are universally signable, and so not all universally signable graphs are chordal, and moreover, not all universally signable graphs are perfect. 

\begin{theorem} \cite{univsign} \label{thm-univ-sign} A graph is universally signable if and only if it contains no Truemper configurations. Furthermore, if $G$ is a universally signable graph, then either $G$ is a complete graph or a hole, or $G$ admits a clique-cutset. 
\end{theorem} 

In this paper, we are interested in a superclass of universally signable graphs. In particular, we study the class of (3PC, proper wheel)-free graphs; we call this class $\mathcal{G}_{\text{UT}}$. Clearly, the only Truemper configurations that graphs in $\mathcal{G}_{\text{UT}}$ may contain are universal wheels and twin wheels. In view of Theorem~\ref{thm-univ-sign}, we see that the class of universally signable graphs is a proper subclass of the class $\mathcal{G}_{\text{UT}}$. 

We also study three subclasses of the class $\mathcal{G}_{\text{UT}}$. $\mathcal{G}_{\text{U}}$ is the class of all (3PC, proper wheel, twin wheel)-free graphs, and $\mathcal{G}_{\text{T}}$ is the class of all (3PC, proper wheel, universal wheel)-free graphs. Clearly, the only Truemper configurations that graphs in $\mathcal{G}_{\text{U}}$ may contain are universal wheels, and the only Truemper configurations that graphs in $\mathcal{G}_{\text{T}}$ may contain are twin wheels. A {\em cap} is a graph that consists of a hole and an additional vertex that is adjacent to two consecutive vertices of the hole and to no other vertices of the hole. $\mathcal{G}_{\text{UT}}^{\text{cap-free}}$ is the class of all (3PC, proper wheel, cap)-free graphs. Clearly, $\mathcal{G}_{\text{U}},\mathcal{G}_{\text{T}},\mathcal{G}_{\text{UT}}^{\text{cap-free}}$ are all proper subclasses of $\mathcal{G}_{\text{UT}}$. Furthermore, classes  $\mathcal{G}_{\text{U}},\mathcal{G}_{\text{T}},\mathcal{G}_{\text{UT}}^{\text{cap-free}}$ are pairwise incomparable, that is, none of the three classes is included in either of the remaining two. Since every Truemper configuration and every cap contains a hole, we see that the class of chordal graphs is a (proper) subclass of each of our four classes (i.e.\ classes $\mathcal{G}_{\text{UT}},\mathcal{G}_{\text{U}},\mathcal{G}_{\text{T}},\mathcal{G}_{\text{UT}}^{\text{cap-free}}$). Furthermore, by Theorem~\ref{thm-univ-sign}, the class of universally signable graphs is a proper subclass of each of $\mathcal{G}_{\text{UT}},\mathcal{G}_{\text{U}},\mathcal{G}_{\text{T}}$. However, the class of universally signable graphs and the class $\mathcal{G}_{\text{UT}}^{\text{cap-free}}$ are incomparable, that is, neither is a subclass of the other (indeed, caps are universally signable, but do not belong to $\mathcal{G}_{\text{UT}}^{\text{cap-free}}$; on the other hand, universal wheels and twin wheels belong to $\mathcal{G}_{\text{UT}}^{\text{cap-free}}$, but they are not universally signable). 

In subsection~\ref{sec1.1decomp}, we describe our structural results, and in subsection~\ref{sec1.2chialg}, we describe our results that involve $\chi$-boundedness and algorithms. In section~\ref{sec:prelim}, we introduce some terminology and notation (mostly standard) that we use throughout the paper, and we prove a few simple lemmas. In sections~\ref{sec:decompGUT}-\ref{sec:alg}, we prove the results outlined in subsections~\ref{sec1.1decomp} and~\ref{sec1.2chialg}.

\subsection{Results: Decomposition theorems for classes $\mathcal{G}_{\text{UT}},\mathcal{G}_{\text{U}},\mathcal{G}_{\text{T}},\mathcal{G}_{\text{UT}}^{\text{cap-free}}$} \label{sec1.1decomp} 

In this subsection, we state our decomposition theorems for the classes $\mathcal{G}_{\text{UT}},\mathcal{G}_{\text{U}},\mathcal{G}_{\text{T}},\mathcal{G}_{\text{UT}}^{\text{cap-free}}$. We first define classes $\mathcal{B}_{\text{UT}},\mathcal{B}_{\text{U}},\mathcal{B}_{\text{T}},\mathcal{B}_{\text{UT}}^{\text{cap-free}}$, which we think of as ``basic'' classes corresponding to the classes $\mathcal{G}_{\text{UT}},\mathcal{G}_{\text{U}},\mathcal{G}_{\text{T}},\mathcal{G}_{\text{UT}}^{\text{cap-free}}$, respectively. For each of the classes $\mathcal{G}_{\text{UT}},\mathcal{G}_{\text{U}},\mathcal{G}_{\text{T}},\mathcal{G}_{\text{UT}}^{\text{cap-free}}$, we show that every graph in the class either belongs to the corresponding basic class or admits a clique-cutset. We state these theorems in the present subsection, and we prove them in sections~\ref{sec:decompGUT}-\ref{sec:decompGUTcap}. 

The complement of a graph $G$ is denoted by $\overline{G}$. As usual, a {\em component} of $G$ is a maximal connected induced subgraph of $G$. A graph is {\em anticonnected} if its complement is connected. An {\em anticomponent} of a graph $G$ is a maximal anticonnected induced subgraph of $G$. (Thus, $H$ is an anticomponent of $G$ if and only if $\overline{H}$ is a component of $\overline{G}$.) Note that anticomponents of a graph $G$ are pairwise ``complete'' to each other in $G$, that is, all possible edges between each pair of distinct anticomponents of $G$ are present in $G$. A component or anticomponent is {\em trivial} if it has just one vertex, and it is {\em nontrivial} if it has at least two vertices. 

\begin{lemma} \label{lemma-anticomp} Let $G$ and $H$ be graphs, and assume that $H$ is anticonnected. Then $G$ is $H$-free if and only if all anticomponents of $G$ are $H$-free. 
\end{lemma} 
\begin{proof} 
This follows immediately from the appropriate definitions. 
\end{proof} 

For an integer $k \geq 4$, a {\em $k$-hyperhole} (or a {\em hyperhole of length $k$}) is any graph obtained from a $k$-hole by blowing up each vertex to a nonempty clique of arbitrary size. Similarly, a {\em $k$-hyperantihole} (or a {\em hyperantihole of length $k$}) is any graph obtained from a $k$-antihole by blowing up each vertex to a nonempty clique of arbitrary size. A hyperhole or hyperantihole is {\em long} if it is of length at least five. 

A {\em ring} is a graph $R$ whose vertex set can be partitioned into $k \geq 4$ nonempty sets, say $X_1,\dots,X_k$ (with subscripts understood to be in $\mathbb{Z}_k$), such that for all $i \in \mathbb{Z}_k$, $X_i$ can be ordered as $X_i = \{u_1^i,\dots,u_{|X_i|}^i\}$ so that $X_i \subseteq N_R[u_{|X_i|}^i] \subseteq \dots \subseteq N_R[u_1^i] = X_{i-1} \cup X_i \cup X_{i+1}$. Under these circumstances, we say that the ring $R$ is of {\em length} $k$, as well as that $R$ is a {\em $k$-ring}. A ring is {\em long} if it is of length at least five. Furthermore, we say that $(X_1,\dots,X_k)$ is a {\em good partition} of the ring $R$. We observe that every $k$-hyperhole is a $k$-ring. 

Given a graph $G$ and distinct vertices $u,v \in V(G)$, we say that $u$ {\em dominates} $v$ in $G$, or that $v$ is {\em dominated} by $u$ in $G$, provided that $N_G[v] \subseteq N_G[u]$. 

\begin{lemma} \label{lemma-ring-char} Let $G$ be a graph, and let $(X_1,\dots,X_k)$, with $k \geq 4$ and subscripts understood to be in $\mathbb{Z}_k$, be a partition of $V(G)$. Then $G$ is a $k$-ring with good partition $(X_1,\dots,X_k)$ if and only if all the following hold: 
\begin{itemize} 
\item[(a)] $X_1,\dots,X_k$ are cliques; 
\item[(b)] for all $i \in \mathbb{Z}_k$, $X_i$ is anticomplete to $V(G) \setminus (X_{i-1} \cup X_i \cup X_{i+1})$; 
\item[(c)] for all $i \in \mathbb{Z}_k$, some vertex of $X_i$ is complete to $X_{i-1} \cup X_{i+1}$; 
\item[(d)] for all $i \in \mathbb{Z}_k$, and all distinct $y_i,y_i' \in X_i$, one of $y_i,y_i'$ dominates the other. 
\end{itemize} 
\end{lemma} 
\begin{proof} 
This readily follows from the definition of a ring. 
\end{proof} 

Let $\mathcal{B}_{\text{UT}}$ be the class of all graphs $G$ that satisfy at least one of the following: 
\begin{itemize} 
\item $G$ has exactly one nontrivial anticomponent, and this anticomponent is a long ring; 
\item $G$ is (long hole, $K_{2,3}$, $\overline{C_6}$)-free; 
\item $\alpha(G) = 2$, and every anticomponent of $G$ is either a 5-hyperhole or a $(C_5,\overline{C_6})$-free graph. 
\end{itemize} 
Note that $\alpha(K_{2,3}) = 3$, and that holes of length at least six have stability number at least three. Thus, graphs of stability number at most two contain no $K_{2,3}$ and no holes of length at least six; consequently, $(C_5,\overline{C_6})$-free graphs of stability number at most two are in fact (long hole, $K_{2,3}$, $\overline{C_6})$-free. 

Let $\mathcal{B}_{\text{U}}$ be the class of all graphs $G$ that satisfy one of the following: 
\begin{itemize} 
\item $G$ has exactly one nontrivial anticomponent, and this anticomponent is a long hole; 
\item all nontrivial anticomponents of $G$ are isomorphic to $\overline{K_2}$. 
\end{itemize} 

Let $\mathcal{B}_{\text{T}}$ be the class of all complete graphs, rings, and 7-hyperantiholes. 

As usual, a graph is {\em bipartite} if its vertex set can be partitioned into two (possibly empty) stable sets. A graph is {\em cobipartite} if its complement is bipartite. A {\em chordal cobipartite graph} is a graph that is both chordal and cobipartite. Let $\mathcal{B}_{\text{UT}}^{\text{cap-free}}$ be the class of all graphs $G$ that satisfy one of the following: 
\begin{itemize} 
\item $G$ has exactly one nontrivial anticomponent, and this anticomponent is a hyperhole of length at least six; 
\item each anticomponent of $G$ is either a 5-hyperhole or a chordal cobipartite graph. 
\end{itemize}
Note that every anticomponent of a complete graph is a chordal cobipartite graph. Thus, complete graphs belong to $\mathcal{B}_{\text{UT}}^{\text{cap-free}}$. Furthermore, if a graph $G$ contains exactly one nontrivial anticomponent, and this anticomponent is a long hyperhole (possibly of length five), then $G \in \mathcal{B}_{\text{UT}}^{\text{cap-free}}$. 

By Lemma~\ref{ring-in-GT}(d) (stated and proven in section~\ref{sec:prelim}), rings are (3PC, proper wheel, universal wheel)-free. Consequently, rings belong to $\mathcal{G}_{\text{T}}$ and to $\mathcal{G}_{\text{UT}}$. Using this fact, we easily obtain the following lemma. 

\begin{lemma} \label{BB-in-GG} $\mathcal{B}_{\text{UT}} \subseteq \mathcal{G}_{\text{UT}}$, $\mathcal{B}_{\text{U}} \subseteq \mathcal{G}_{\text{U}}$, $\mathcal{B}_{\text{T}} \subseteq \mathcal{G}_{\text{T}}$, and $\mathcal{B}_{\text{UT}}^{\text{cap-free}} \subseteq \mathcal{G}_{\text{UT}}^{\text{cap-free}}$. 
\end{lemma} 
\begin{proof}[Proof (assuming Lemma~\ref{ring-in-GT})] 
It follows from Lemma~\ref{ring-in-GT}(d) that rings belong to $\mathcal{G}_{\text{T}}$ and to $\mathcal{G}_{\text{UT}}$. Furthermore, note that the only Truemper configurations that are not anticonnected are the theta $K_{2,3}$, the twin wheel $W_4^{\text{t}}$, and universal wheels. The result now follows from Lemma~\ref{lemma-anticomp} and routine checking. 
\end{proof} 

We now state our decomposition theorems for the classes $\mathcal{G}_{\text{UT}},\mathcal{G}_{\text{U}},\mathcal{G}_{\text{T}},\mathcal{G}_{\text{UT}}^{\text{cap-free}}$. We prove these theorems in sections~\ref{sec:decompGUT}-\ref{sec:decompGUTcap}. 

\begin{theorem} \label{decomp-thm-GUT} Every graph in $\mathcal{G}_{\text{UT}}$ either belongs to $\mathcal{B}_{\text{UT}}$ or admits a clique-cutset. 
\end{theorem} 

\begin{theorem} \label{decomp-thm-GU} Every graph in $\mathcal{G}_{\text{U}}$ either belongs to $\mathcal{B}_{\text{U}}$ or admits a clique-cutset. 
\end{theorem} 

\begin{theorem} \label{decomp-thm-GT} Every graph in $\mathcal{G}_{\text{T}}$ either belongs to $\mathcal{B}_{\text{T}}$ or admits a clique-cutset. 
\end{theorem} 

\begin{theorem} \label{decomp-thm-GUTcap} Every graph in $\mathcal{G}_{\text{UT}}^{\text{cap-free}}$ either belongs to $\mathcal{B}_{\text{UT}}^{\text{cap-free}}$ or admits a clique-cutset. 
\end{theorem} 

Note that the the clique-cutset decomposition has a natural reverse operation, namely the operation of ``gluing along a clique.'' Let $G_1$ and $G_2$ be graphs, and assume that $C = V(G_1) \cap V(G_2)$ is a (possibly empty) clique. Let $G$ be the graph with vertex set $V(G) = V(G_1) \cup V(G_2)$ and edge set $E(G) = E(G_1) \cup E(G_2)$. Under these circumstances, we say that $G$ is obtained by {\em gluing $G_1$ and $G_2$ along the clique $C$}, or simply that $G$ is obtained from $G_1$ and $G_2$ by {\em gluing along a clique}. 

\begin{lemma} \label{lemma-decomp-to-comp} Let $\mathcal{H}$ be a family of graphs, none of which admits a clique-cutset, and let $\mathcal{G}$ be the class of $\mathcal{H}$-free graphs. Let $\mathcal{B}$ be a subclass of $\mathcal{G}$. Assume that every graph in $\mathcal{G}$ either belongs to $\mathcal{B}$ or admits a clique-cutset. Then a graph belongs to $\mathcal{G}$ if and only if it can be obtained from graphs in $\mathcal{B}$ by repeatedly gluing along cliques. 
\end{lemma} 
\begin{proof} 
This readily follows from appropriate definitions. 
\end{proof} 

Since no 3PC and no wheel admits a clique-cutset, Lemmas~\ref{BB-in-GG} and~\ref{lemma-decomp-to-comp} imply that Theorems~\ref{decomp-thm-GUT},~\ref{decomp-thm-GU}, and~\ref{decomp-thm-GT} can readily be converted into composition theorems. On the other hand, every cap admits a clique-cutset, and so Theorem~\ref{decomp-thm-GUTcap} cannot be converted to a composition theorem.

\subsection{Results: $\chi$-Boundedness and algorithms} \label{sec1.2chialg}

In section~\ref{sec:chi}, we study $\chi$-boundedness. A class $\mathcal{G}$ is said to be \emph{$\chi$-bounded} provided that there exists a function $f:\mathbb{N}^+ \rightarrow \mathbb{N}^+$ (called a {\em $\chi$-bounding function} for $\mathcal{G}$) such that for all graphs $G\in\mathcal{G}$, all induced subgraphs $H$ of $G$ satisfy $\chi(H) \leq f(\omega(H))$. Note that a hereditary class $\mathcal{G}$ is $\chi$-bounded if and only if there exists a function $f:\mathbb{N}^+ \rightarrow \mathbb{N}^+$ such that every graph $G \in \mathcal{G}$ satisfies $\chi(G) \leq f(\omega(G))$. $\chi$-Boundedness was introduced by Gy\'{a}rf\'{a}s \cite{Gya87} as a natural generalization of perfection: clearly, the class of perfect graphs is hereditary and $\chi$-bounded by the identity function. It follows from~\cite{KuhnOst} that the class of theta-free graphs is $\chi$-bounded; consequently, our four classes (i.e.\ classes $\mathcal{G}_{\text{UT}},\mathcal{G}_{\text{U}},\mathcal{G}_{\text{T}},\mathcal{G}_{\text{UT}}^{\text{cap-free}}$) are all $\chi$-bounded. Unfortunately, the $\chi$-bounding function from~\cite{KuhnOst} is superexponential. Using our structural results, we obtain polynomial $\chi$-bounding functions for our four classes. In fact, we obtain linear $\chi$-bounding functions for the classes $\mathcal{G}_{\text{U}},\mathcal{G}_{\text{T}},\mathcal{G}_{\text{UT}}^{\text{cap-free}}$; our $\chi$-bounding function for the class $\mathcal{G}_{\text{UT}}$ is a fourth-degree polynomial function. 

Finally, in section~\ref{sec:alg}, we turn to the algorithmic consequences of our structural results. We consider four algorithmic problems: 
\begin{itemize} 
\item the recognition problem, i.e.\ the problem of determining whether an input graph belongs to a given class; 
\item the maximum weight stable set problem (MWSSP), i.e.\ the problem of finding a maximum weight stable set in an input weighted graph (with real weights); 
\item the maximum weight clique problem (MWCP), i.e.\ the problem of finding a maximum weight clique in an input weighted graph (with real weights); 
\item the optimal coloring problem (ColP), i.e.\ the problem of finding an optimal coloring of an input graph. 
\end{itemize} 
We remark that all our algorithms are {\em robust}, that is, they either produce a correct solution to the problem in question for the input (weighted) graph, or they correctly determine that the graph does not belong to the class under consideration. (If the input graph does not belong to the class under consideration, a robust algorithm may possibly produce a correct solution to the problem in question, rather than determine that the input graph does not belong to the class.) 

A summary of our results is given in the table below. As usual, $n$ is the number of vertices and $m$ the number of edges of the input graph. For the sake of compactness, we write $O(nm)$ and $O(n^2m)$ instead of $O(nm+n^2)$ and $O(n^2m+n^3)$, respectively. 

\begin{center}
 \begin{tabular}{|l || c | c | c | c | l |} 
 \hline
  & recognition & MWSSP & MWCP & ColP & $\chi$-bound. \\ [0.5ex] 
 \hline\hline
 $\mathcal{G}_{\text{UT}}$ & $O(n^6)$ & ? & NP-hard & ? & $\chi \leq 2\omega^4$ \\ 
 \hline
 $\mathcal{G}_{\text{U}}$ & $O(nm)$ & $O(nm)$ & $O(nm)$~\cite{ACTV15} & $O(nm)$ & $\chi \leq \omega+1$ \\
 \hline
 $\mathcal{G}_{\text{T}}$ & $O(n^3)$ & $O(n^2m)$ & $O(nm)$ & ? & $\chi \leq \lfloor \frac{3}{2}\omega \rfloor$ \\
 \hline
 $\mathcal{G}_{\text{UT}}^{\text{cap-free}}$ & $O(n^5)$ & $O(n^3)$ & $O(n^3)$ & $O(n^3)$ & $\chi \leq \lfloor \frac{3}{2}\omega \rfloor$ \\
 \hline
\end{tabular} 
\end{center}

Most of our algorithms rely on Theorems~\ref{decomp-thm-GUT},~\ref{decomp-thm-GU},~\ref{decomp-thm-GT}, and~\ref{decomp-thm-GUTcap}. Since all four theorems involve clique-cutsets, most of our algorithms also rely on techniques developed in~\cite{Tarjan} for handling clique-cutsets. 

At this time, we do not know whether rings can be optimally colored in polynomial time, and for this reason, we do not know the complexity of the ColP for the class $\mathcal{G}_{\text{T}}$. 

As shown in the table, an $O(nm+n^2)$ time algorithm solving the MWCP for the class $\mathcal{G}_{\text{U}}$ was given in~\cite{ACTV15}; that algorithm relies on LexBFS~\cite{RTL76}. In the present paper, we give a different algorithm solving the MWCP for the class $\mathcal{G}_{\text{U}}$ (our algorithm has the same complexity as the one from~\cite{ACTV15}, but it relies on our structural results for the class $\mathcal{G}_{\text{U}}$). Further, we note that the complexity of the ColP for the class $\mathcal{G}_{\text{U}}$ was left open in~\cite{ACTV15}; here, we give a polynomial time algorithm that solves this problem. Finally, we note that it was shown in~\cite{ACTV15} that every graph $G \in \mathcal{G}_{\text{U}}$ has a {\em bisimplicial} vertex, i.e.\ a vertex whose neighborhood can be partitioned into two (possibly empty) cliques; this result readily implies that every graph $G \in \mathcal{G}_{\text{U}}$ satisfies $\chi(G) \leq 2\omega(G)-1$. Using our structural results, we obtain a better $\chi$-bounding function for the class $\mathcal{G}_{\text{U}}$ (as shown in the table above). 

The {\em join} of graphs $G_1,\dots,G_k$ on pairwise disjoint vertex sets is the graph $G$ with vertex set $V(G) = V(G_1) \cup \dots \cup V(G_k)$ and edge set $E(G) = E(G_1) \cup \dots \cup E(G_k) \cup \{x_ix_j \mid 1 \leq i < j \leq k, x_i \in V(G_i), x_j \in V(G_j)\}$. 

Note that if $H$ is the join of an odd hole and a complete graph, then $H \in \mathcal{G}_{\text{U}}$ and $\chi(H) = \omega(H)+1$. Further, if $K$ is the join of arbitrarily many copies of $C_5$, then $K \in \mathcal{G}_{\text{UT}}^{\text{cap-free}}$ and $\chi(K) = \lfloor \frac{3}{2} \omega(K) \rfloor$. This proves that our $\chi$-bounding functions for the classes $\mathcal{G}_{\text{U}}$ and $\mathcal{G}_{\text{UT}}^{\text{cap-free}}$ are optimal. We do not know whether our $\chi$-bounding function for the class $\mathcal{G}_{\text{T}}$ is optimal. In section~\ref{sec:chi}, we show that the class $\mathcal{G}_{\text{UT}}$ is $\chi$-bounded by a function of order $\frac{\omega^4}{\log^2\omega}$, and so the $\chi$-bounding function for $\mathcal{G}_{\text{UT}}$ given in the table above is not optimal. In fact, we do not know the order of the optimal $\chi$-bounding function for the class $\mathcal{G}_{\text{UT}}$.

\section{Preliminaries} \label{sec:prelim} 

In this section, we introduce some (mostly standard) terminology and notation that we use throughout the paper. We also prove a few preliminary results.  

\subsection{Terminology and notation}\label{sec:def}

The set of nonnegative integers is denoted by $\mathbb{N}$, and the set of positive integers by $\mathbb{N}^+$. A {\em singleton} is a one-element set. 

The vertex set and edge set of a graph $G$ are denoted by $V(G)$ and $E(G)$, respectively. When no confusion is possible, we write $G$ instead of $V(G)$. 

A graph is {\em trivial} if it has just one vertex; a graph is {\em nontrivial} if it has at least two vertices. For a vertex $x$ of a graph $G$, $N_G(x)$ is the set of all neighbors of $x$ in $G$, $d_G(x)=|N_G(x)|$ is the \emph{degree} of $x$ in $G$, and $N_G[x]=N_G(x)\cup \{x\}$. For a set $S \subseteq V(G)$, $N_G(S)$ is the set of all vertices in $V(G)\setminus S$ that have at least one neighbor in $S$, and $N_G[S]=N_G(S) \cup S$. The maximum degree of $G$ is denoted by $\Delta(G)$, that is, $\Delta(G) = \max\{d_G(x) \mid x \in V(G)\}$. 

For a graph $G$ and a nonempty set $S \subseteq V(G)$, $G[S]$ denotes the subgraph of $G$ induced by $S$. Given vertices $x_1,\dots,x_k \in V(G)$, we often write $G[x_1,\dots,x_k]$ instead of $G[\{x_1,\dots,x_k\}]$. 

For a graph $G$ and a set $S \subsetneqq V(G)$, we set $G \setminus S = G[V(G) \setminus S]$. If $G$ is nontrivial and $x \in V(G)$, we often write $G \setminus x$ instead of $G \setminus \{x\}$. (Since we only deal with nonnull graphs, if $G$ is trival and $x$ is the only vertex of $G$, then $G \setminus x$ is undefined.) 

Given a graph $G$, a vertex $x \in V(G)$, and a set $Y \subseteq V(G) \setminus \{x\}$, we say that $x$ is {\em complete} (resp. {\em anticomplete}) to $Y$ in $G$ provided that $x$ is adjacent (resp. nonadjacent) to every vertex in $Y$. Given disjoint sets $X,Y \subseteq V(G)$, we say that $X$ is {\em complete} (resp. {\em anticomplete}) to $Y$ in $G$ provided that every vertex in $X$ is complete (resp. anticomplete) to $Y$. 

As usual, a \emph{clique} (resp. {\em stable set}) in a graph $G$ is a (possibly empty) set of pairwise adjacent (resp. nonadjacent) vertices of $G$. The \emph{clique number} of $G$, denoted by $\omega(G)$, is the size of a largest clique in $G$; the {\em stability number} of $G$, denoted by $\alpha(G)$, is the size of a largest stable set in $G$. A {\em maximum clique} (resp. {\em maximum stable set}) of $G$ is a clique (resp. stable set) of size $\omega(G)$ (resp. $\alpha(G)$). A \emph{complete graph} is a graph whose vertex set is a clique. The complete graph on $n$ vertices is denoted by $K_n$; $K_3$ is also referred to as a \emph{triangle}. 

A \emph{weighted graph} is an ordered pair $(G,w)$, where $G$ is a graph and $w:V(G) \rightarrow \mathbb{R}$ is a \emph{weight function} for $G$. For a set $S\subseteq V(G)$, the {\em weight} of $S$, denoted by $w(S)$, is the sum of weights of all vertices in $S$, that is, $w(S) = \sum_{v \in S} w(v)$. The \emph{clique number} (resp. \emph{stability number}) of a weighted graph $(G,w)$, denoted by $\omega(G,w)$ (resp. $\alpha(G,w)$), is the maximum weight of a clique (resp. stable set) of $G$. A {\em maximum weight clique} (resp. {\em maximum weight stable set}) of $(G,w)$ is a clique (resp. stable set) of $G$ whose weight is precisely $\omega(G,w)$ (resp. $\alpha(G,w)$). Clearly, if $(G,w)$ is a weighted graph and $H$ is an induced subgraph of $G$, then the restriction of $w$ to $V(H)$, denoted by $w \upharpoonright V(H)$, is a weight function for $H$, and $(H,w \upharpoonright V(H))$ is a weighted graph; to simplify notation, we usually write $(H,w)$ instead of $(H,w \upharpoonright V(H))$. 

For a positive integer $k$, a \emph{$k$-coloring} of a graph $G$ is a function $c:V(G)\rightarrow \{1,\dots,k\}$ such that $c(x)\neq c(y)$ whenever $xy\in E(G)$; elements of $\{1,\dots,k\}$ are called \emph{colors}. A graph is \emph{$k$-colorable} if it admits a $k$-coloring. The \emph{chromatic number} of $G$, denoted by $\chi(G)$, is the smallest integer $k$ such that $G$ is $k$-colorable. 

A {\em path} is a graph $P$ with vertex set $V(P) = \{x_1,\dots,x_k\}$ (where $k \geq 1$) and edge set $E(P) = \{x_1x_2,x_2x_3,\dots,x_{k-1}x_k\}$; under these circumstances, we write that ``$P = x_1,\dots,x_k$ is a path,'' and we say that the {\em length} of the path $P$ is $k-1$ (i.e.\ the length of a path is the number of edges that it contains), that the {\em endpoints} of $P$ are $x_1$ and $x_k$ (if $k = 1$, then the endpoints of $P$ coincide), that $x_2,\dots,x_{k-1}$ are the {\em interior vertices} of the path $P$ (note that $P$ has interior vertices if and only if $k \geq 3$), and that $P$ is a path {\em between} $x_1$ and $x_k$. A {\em path} in a graph $G$ is a subgraph of $G$ that is a path. An {\em induced path} in a graph $G$ is an induced subgraph of $G$ that is a path. 

A {\em cycle} is a graph $C$ with vertex set $V(C) = \{x_1,\dots,x_k\}$ (where $k \geq 3$, and subscripts are understood to be in $\mathbb{Z}_k$) and edge set $E(C) = \{x_1x_2,x_2x_3,\dots,x_{k-1}x_k,x_kx_1\}$; under these circumstances, we write that ``$C = x_1,\dots,x_k,x_1$ is a cycle,'' and we say that the {\em length} of $C$ is $k$. A {\em cycle} in a graph $G$ is a subgraph of $G$ that is a cycle. An {\em induced cycle} in a graph $G$ is an induced subgraph of $G$ that is a cycle. 

A path of length $k$ is denoted by $P_{k+1}$ (note that $P_{k+1}$ has $k+1$ vertices and $k$ edges), and a cycle of length $k$ is denoted by $C_k$ (note that $C_k$ has $k$ vertices and $k$ edges). 

A {\em hole} in a graph $G$ is an induced cycle of length at least four. An {\em antihole} in a graph $G$ is an induced subgraph of $G$ whose complement is a hole in $\overline{G}$. The {\em length} of a hole or antihole is the number of vertices that it contains; a {\em $k$-hole} (resp. {\em $k$-antihole}) is a hole (resp. antihole) of length $k$. A hole or antihole is {\em long} if it is of length at least five. A hole or antihole is {\em odd} (resp. {\em even}) if its length is odd (resp. even). Further, consistently with the notation above, we write ``$H = x_1,\dots,x_k,x_1$ is a hole,'' or simply ``$x_1,\dots,x_k,x_1$ is a hole'' (with $k \geq 4$, and subscripts understood to be in $\mathbb{Z}_k$) when $H = x_1,\dots,x_k,x_1$ is an induced cycle. On the other hand, we write that ``$A = x_1,\dots,x_k,x_1$ is an antihole,'' or simply ``$x_1,\dots,x_k,x_1$ is an antihole'' (with $k \geq 4$, and subscripts understood to be in $\mathbb{Z}_k$) when $\overline{A} = x_1,\dots,x_k,x_1$ is a hole. 


Let $H$ be an induced subgraph of a graph $G$. Two distinct vertices $x,y\in V(G)$ are \emph{twins} with respect to $H$ if $N_G[x] \cap V(H) = N_G[y] \cap V(H)$. Given a vertex $x \in V(G)$, we denote by $X_x^G(H)$ the set consisting of $x$ and all twins of $x$ in $G$ with respect to $H$. The set of all vertices in $V(G)\setminus V(H)$ that are complete to $V(H)$ is denoted by $U_H^G$. When no confusion is possible, we omit the superscript $G$ in $X_x^G(H)$ and $U_H^G$, and instead, we write simply $X_x(H)$ and $U_H$, respectively. Further, we set $H_G^* = G[\bigcup_{x \in V(H)} X_x^G(H)]$; when no confusion is possible, we omit the subscript $G$ and write simply $H^*$. 

A {\em hyperhole} is a graph $H$ whose vertex set can be partitioned into $k \geq 4$ nonempty cliques, call them $X_1,\dots,X_k$ (with subscripts understood to be in $\mathbb{Z}_k$), such that for all $i \in \mathbb{Z}_k$, $X_i$ is complete to $X_{i-1} \cup X_{i+1}$ and anticomplete to $V(H) \setminus (X_{i-1} \cup X_i \cup X_{i+1})$; under these circumstances, we say that the hyperhole $H$ is of {\em length} $k$, and we also write that ``$H = X_1,\dots,X_k,X_1$ is a hyperhole''; furthermore, we say that $(X_1,\dots,X_k)$ is a {\em good partition} of the hyperhole $H$. A {\em $k$-hyperhole} is a hyperhole of length $k$, and a {\em long hyperhole} is a hyperhole of length at least five. Note that if $H$ is a $k$-hyperhole with good partition $(X_1,\dots,X_k)$, then $H$ is a $k$-ring with good partition $(X_1,\dots,X_k)$. 

A {\em hyperantihole} is a graph $A$ whose vertex set can be partitioned into $k \geq 4$ nonempty cliques, call them $X_1,\dots,X_k$ (with subscripts understood to be in $\mathbb{Z}_k$), such that for all $i \in \mathbb{Z}_k$, $X_i$ is anticomplete to $X_{i-1} \cup X_{i+1}$ and complete to $V(A) \setminus (X_{i-1} \cup X_i \cup X_{i+1})$; under these circumstances, we say that the hyperantihole $A$ is of {\em length} $k$, and we also write that ``$A = X_1,\dots,X_k,X_1$ is a hyperantihole''; furthermore, we say that $(X_1,\dots,X_k)$ is a {\em good partition} of the hyperantihole $A$. A {\em $k$-hyperantihole} is a hyperantihole of length $k$, and a {\em long hyperantihole} is a hyperantihole of length at least five. Note that the complement of a hyperantihole need not be a hyperhole. 

A graph is \emph{bipartite} if its vertex set can be partitioned into two (possibly empty) stable sets. A graph is \emph{cobipartite} if its complement is bipartite. A \emph{complete bipartite graph} is a graph whose vertex set can be partitioned into two (possibly empty) stable sets that are complete to each other; $K_{n,m}$ is a complete bipartite graph whose vertex set can be partitioned into two stable sets, one of size $n$ and the other one of size $m$, that are complete to each other. 

A {\em cutset} of a graph $G$ is a (possibly empty) set $C \subsetneqq V(G)$ such that $G \setminus C$ is disconnected. A {\em cut-partition} of a graph $G$ is a partition $(A,B,C)$ of $V(G)$ such that $A$ and $B$ are nonempty and anticomplete to each other (the set $C$ may possibly be empty). Clearly, if $(A,B,C)$ is a cut-partition of $G$, then $C$ is a cutset of $G$; conversely, every cutset of $G$ gives rise to at least one cut-partition of $G$. A {\em clique-cutset} of a graph $G$ is a (possibly empty) clique of $G$ that is also a cutset of $G$. A {\em clique-cut-partition} of a graph $G$ is a cut-partition $(A,B,C)$ of $G$ such that $C$ is a clique. Again, if $(A,B,C)$ is a clique-cut-partition of $G$, then $C$ is a clique-cutset of $G$, and conversely, every clique-cutset of $G$ gives rise to at least one clique-cut-partition of $G$. 

Let $G$ be a 3PC. Then $G$ contains three induced paths, say $P_1=x_1,\ldots,y_1$, $P_2=x_2,\ldots,y_2$, and $P_3=x_3,\ldots,y_3$, such that $V(G) = V(P_1) \cup V(P_2) \cup V(P_3)$, and such that $\{x_1,x_2,x_3\}\cap\{y_1,y_2,y_3\}= \emptyset$, $\{ x_1,x_2,x_3\}$ either induces a triangle or is a singleton (i.e.\ $x_1 = x_2 = x_3$), $\{ y_1,y_2,y_3\}$ either induces a triangle or is a singleton (i.e.\ $y_1 = y_2 = y_3$), and $V(P_i) \cup V(P_j)$ induces a hole for all distinct $i,j \in \{1,2,3\}$. If $x_1 = x_2 = x_3$ and $y_1 = y_2 = y_3$, then we say that $G$ is a $3PC(x_1,y_1)$; in this case, $G$ is a theta. If $\{x_1,x_2,x_3\}$ induces a triangle and $y_1 = y_2 = y_3$, then we say that $G$ is a $3PC(x_1x_2x_3,y_1)$; in this case, $G$ is a pyramid. Finally, if $\{x_1,x_2,x_3\}$ and $\{y_1,y_2,y_3\}$ both induce a triangle, then we say that $G$ is a $3PC(x_1x_2x_3,y_1y_2y_3)$; in this case, $G$ is a prism. When we say that ``$K$ is a 3PC in $G$,'' we always assume that $K$ is an induced subgraph of $G$. 

A {\em wheel} $(H,x)$ is a graph that consists of a hole $H$, called the {\em rim}, and an additional vertex $x$, called the {\em center}, such that $x$ has at least three neighbors in $H$. A {\em universal wheel} is a wheel $(H,x)$ in which $x$ is complete to $V(H)$. A {\em twin wheel} is a wheel $(H,x)$ such that $x$ has precisely three neighbors in $H$, and those three neighbors are consecutive vertices of $H$. A wheel that is neither a universal wheel nor a twin wheel is called a \emph{proper wheel}. When we say that ``$(H,x)$ is a wheel in $G$,'' we always assume that the wheel $(H,x)$ is an induced subgraph of $G$.

\subsection{A few preliminary lemmas} \label{sec:lemmas}

Let $W_5^4$ be the six-vertex wheel consisting of a $C_5$ and a vertex that has precisely four neighbors in the $C_5$. We remind the reader that, for $k \geq 4$, the universal wheel on $k+1$ vertices is denoted by $W_k$, and the twin wheel on $k+1$ vertices is denoted by $W_k^{\text{t}}$.

\begin{lemma} \label{rmrk-Truemper} No Truemper configuration admits a clique-cutset. The only Truemper configurations of stability number two are the prism $\overline{C_6}$, the universal wheels $W_4$ and $W_5$, the twin wheels $W_4^{\text{t}}$ and $W_5^{\text{t}}$, and the proper wheel $W_5^4$; all other Truemper configurations have stability number at least three. The theta $K_{2,3}$, the prism $\overline{C_6}$, the universal wheel $W_4$, and the twin wheel $W_4^{\text{t}}$ are the only Truemper configuration that do not contain a long hole. The only Truemper configurations that are not anticonnected are the theta $K_{2,3}$, the twin wheel $W_4^{\text{t}}$, and universal wheels. 
\end{lemma} 
\begin{proof} 
This follows by routine checking. 
\end{proof} 

\begin{lemma} \label{rmrk-K23-anticomp} If a $K_{2,3}$-free graph $G$ has at least two nontrivial anticomponents, then $\alpha(G) = 2$. 
\end{lemma} 
\begin{proof} 
Let $G$ be a graph that has at least two nontrivial anticomponents, and assume that $\alpha(G) \geq 3$. Let $\{a_1,a_2,a_3\}$ be a stable set of size three in $G$; clearly, $a_1,a_2,a_3$ belong to the same anticomponent of $G$. Let $H$ be a nontrivial anticomponent of $G$ that is different from the one containing $a_1,a_2,a_3$, and fix nonadjacent vertices $b_1,b_2 \in V(H)$. Then $G[b_1,b_2,a_1,a_2,a_3]$ is a $K_{2,3}$, and so $G$ is not $K_{2,3}$-free. 
\end{proof} 

The (unique) cap on five vertices is called the {\em house}. Note that the house is isomorphic to $\overline{P_5}$. Clearly, every cap-free graph is house-free. 

\begin{lemma} \label{lemma-clique-cut-alpha-2} Let $G$ be a graph. Assume that $\alpha(G) \leq 2$ and that $G$ admits a clique-cutset. Then the following hold: 
\begin{itemize} 
\item[(a)] $G$ is cobipartite, and consequently, $G$ contains no long holes; 
\item[(b)] if $G$ is house-free, then $G$ is chordal. 
\end{itemize} 
\end{lemma} 
\begin{proof} 
Let $(A,B,C)$ be a clique-cut-partition of $G$. Then $A$ is a clique, for otherwise, we fix nonadjacent vertices $a_1,a_2 \in A$, we fix any $b \in B$, and we observe that $\{a_1,a_2,b\}$ is a stable set of size three, a contradiction. Similarly, $B$ is a clique. Further, every vertex of $C$ is complete to at least one of $A$ and $B$, for otherwise, we fix some $c \in C$ that has a nonneighbor $a \in A$ and a nonneighbor $b \in B$, and we observe that $\{a,b,c\}$ is a stable set of size three, a contradiction. Let $C_A$ be the set of all vertices of $C$ that are complete to $A$, and let $C_B = C \setminus C_A$; then $C_B$ is complete to $B$. Now $A \cup C_A$ and $B \cup C_B$ are (disjoint) cliques whose union is $V(G)$, and it follows that $G$ is cobipartite. Since no cobipartite graph contains a long hole, (a) follows. 

It remains to prove (b). We assume that $G$ is house-free, and we show that $G$ is chordal. In view of (a), we just need to show that $G$ contains no 4-holes. Suppose otherwise, and let $H = x_1,x_2,x_3,x_4,x_1$ be a 4-hole in $G$. Since $H$ contains no clique-cutset, we see that either $V(H) \subseteq A \cup C$ or $V(H) \subseteq B \cup C$; by symmetry, we may assume that $V(H) \subseteq A \cup C$. Since $A$ and $C$ are cliques, and since $H$ contains no triangles, we see that each of $A$ and $C$ contains at most two vertices of $H$, and furthermore, if $A$ or $C$ contains precisely two vertices, then those two vertices are adjacent. By symmetry, we may now assume that $x_1,x_2 \in A$ and $x_3,x_4 \in C$. But then neither $x_3$ nor $x_4$ is complete to $A$, and consequently, $x_3$ and $x_4$ are complete to $B$. Fix $b \in B$. Now $G[x_1,x_2,x_3,x_4,b]$ is a house, a contradiction. This proves (b). 
\end{proof} 

\begin{lemma} \label{ring-in-GT} Let $R$ be a $k$-ring with good partition $(X_1,\dots,X_k)$. Then all the following hold: 
\begin{itemize} 
\item[(a)] every hole in $R$ intersects each of $X_1,\dots,X_k$ in exactly one vertex; 
\item[(b)] every hole in $R$ is of length $k$; 
\item[(c)] for all $i \in \mathbb{Z}_k$, $R \setminus X_i$ is chordal; 
\item[(d)] $R$ is (3PC, proper wheel, universal wheel)-free; 
\item[(e)] $R$ is cap-free if and only if $R$ is a $k$-hyperhole with good partition $(X_1,\dots,X_k)$. 
\end{itemize} 
\end{lemma} 
\begin{proof} 
Since no vertex in a hole dominates any other vertex of that hole, Lemma~\ref{lemma-ring-char}(d) guarantees that a hole in $R$ can intersect each of $X_1,\dots,X_k$ in at most one vertex. Statement (a) now follows from Lemma~\ref{lemma-ring-char}(b). 

Statements (b) and (c) follow immediately from (a). 

Next, we prove (d). Suppose that $K$ is a 3PC in $R$. We know that $K$ contains at least three holes, and by (a), each of those holes contains exactly one vertex from each of $X_1,\dots,X_k$. Thus, some $X_i$ (with $i \in \mathbb{Z}_k$) contains at least two distinct vertices of $K$. But by the definition of a 3PC, we see that every pair of distinct vertices of $K$ belongs to a hole of $K$. Thus, $X_i$ contains at least two vertices of some hole of $K$, contrary to (a). This proves that $R$ is 3PC-free. 

Suppose now that $(H,x)$ is a wheel in $R$; we must show that $(H,x)$ is a twin wheel. Using (a), for each $i \in \mathbb{Z}_k$, we let $x_i$ be the unique vertex in $V(H) \cap X_i$. It readily follows from Lemma~\ref{lemma-ring-char}(b) that the hole $H$ is of the form $H = x_1,\dots,x_k,x_1$. By symmetry, we may assume that $x \in X_2$. Since $x$ has at least three neighbors in $H$ (because $(H,x)$ is a wheel), Lemma~\ref{lemma-ring-char}(b) implies that the neighbors of $x$ in $H$ are precisely $x_1,x_2,x_3$. Thus, $(H,x)$ is a twin wheel, and we deduce that $R$ is (proper wheel, universal wheel)-free. This proves (d). 

It remains to prove (e). The ``if'' part follows from (a) and routine checking. For the ``only if'' part, we assume that $R$ is not a $k$-hyperhole with good partition $(X_1,\dots,X_k)$, and we show that $R$ is not cap-free. Since $R$ is a $k$-ring with good partition $(X_1,\dots,X_k)$, but not a $k$-hyperhole with good partition $(X_1,\dots,X_k)$, we may assume by symmetry that $X_1$ is not complete to $X_2$. Fix nonadjacent vertices $y_1 \in X_1$ and $y_2 \in X_2$. By the definition of a ring, for each $i \in \mathbb{Z}_k$, there exists a vertex $x_i \in X_i$ such that $N_R[x_i] = X_{i-1} \cup X_i \cup X_{i+1}$. Since $y_1y_2 \notin E(R)$, we see that $y_1 \neq x_1$ and $y_2 \neq x_2$. But now $H = y_1,x_2,\dots,x_k,y_1$ is a hole in $R$, and $N_R(y_2) \cap V(H) = \{x_2,x_3\}$; it follows that $R[y_1,y_2,x_2,\dots,x_k]$ is a cap, and so $R$ is not cap-free. This proves (e). 
\end{proof}

\section{A decomposition theorem for the class $\mathcal{G}_{\text{UT}}$}\label{sec:decompGUT}

In this section, we prove Theorem~\ref{decomp-thm-GUT}, which states that every graph in $\mathcal{G}_{\text{UT}}$ either belongs to $\mathcal{B}_{\text{UT}}$ or admits a clique-cutset. We begin with a few preliminary lemmas, which will be of use to us, not only in this section, but also in subsequent ones. 

\begin{lemma} \label{lemma-vertex-attach-GUT} Let $G \in \mathcal{G}_{\text{UT}}$, let $H = x_1,\dots,x_k,x_1$ (with $k \geq 4$) be a hole in $G$, and let $x \in V(G) \setminus V(H)$. Then one of the following holds: 
\begin{itemize} 
\item[(a)] $x$ is complete to $V(H)$; 
\item[(b)] there exists some $i \in \mathbb{Z}_k$ such that $N_G(x) \cap V(H) = \{x_{i-1},x_i,x_{i+1}\}$ (i.e.\ $x$ is a twin of $x_i$ with respect to $H$); 
\item[(c)] there exists some $i \in \mathbb{Z}_k$ such that $N_G(x) \cap V(H) \subseteq \{x_i,x_{i+1}\}$ (i.e.\ $N_G(x) \cap V(H)$ is a clique of size at most two). 
\end{itemize} 
\end{lemma} 
\begin{proof} 
If $|N_G(x) \cap V(H)| \leq 1$, then (c) holds. If $|N_G(x) \cap V(H)| = 2$, then (c) holds, for otherwise, $G[V(H) \cup \{x\}]$ is a theta, a contradiction. If $3 \leq |N_G(x) \cap V(H)| \leq k-1$, then (b) holds, for otherwise, $(H,x)$ is a proper wheel in $G$, a contradiction. Finally, if $|N_G(x) \cap V(H)| = k$, then (a) holds, and we are done. 
\end{proof} 

\begin{lemma} \label{lemma-hole-ring-GUT} Let $G \in \mathcal{G}_{\text{UT}}$, and let $H = x_1,\dots,x_k,x_1$ (with $k \geq 4$) be a hole in $G$. For all $i \in \{1,\dots,k\}$, set $X_i = X_{x_i}(H)$. Then the following hold: 
\begin{itemize} 
\item $X_1,\dots,X_k$ are pairwise disjoint cliques; 
\item if $k \geq 5$, then $H^*$ is a $k$-ring with good partition $(X_1,\dots,X_k)$. 
\end{itemize} 
\end{lemma} 
\begin{proof} 
It is clear that $X_1,\dots,X_k$ are pairwise disjoint, so that $(X_1,\dots,X_k)$ is a partition of $V(H^*)$. Let us show that $X_1,\dots,X_k$ are cliques. By symmetry, it suffices to show that $X_1$ is a clique. Suppose otherwise, and fix nonadjacent vertices $y_1,y_1' \in X_1$. But then $G[y_1,y_1',x_2,\dots,x_k]$ is a $3PC(x_2,x_k)$, a contradiction. This proves that $X_1,\dots,X_k$ are cliques. 

From now on, we assume that $k \geq 5$. Our goal is to show that $H^*$ and $(X_1,\dots,X_k)$ satisfy (a)-(d) of Lemma~\ref{lemma-ring-char}. We already showed that $X_1,\dots,X_k$ satisfy (a). Further, it is clear that for all $i \in \mathbb{Z}_k$, $x_i$ is complete to $X_{i-1} \cup X_{i+1}$; thus, (c) holds. 

We now prove (b). Suppose otherwise. By symmetry, we may assume that for some index $j \in \mathbb{Z}_k \setminus \{k,1,2\}$ and vertices $y_1 \in X_1$ and $y_j \in X_j$, we have that $y_1y_j \in E(G)$. By construction, $x_1$ is anticomplete to $X_j$, and $x_j$ is anticomplete to $X_1$; since $y_1y_j \in E(G)$, it follows that $y_1 \neq x_1$ and $y_j \neq x_j$. But now the hole $y_1,x_2,\dots,x_k,y_1$ and vertex $y_j$ contradict Lemma~\ref{lemma-vertex-attach-GUT}. Thus, (b) holds. 

It remains to prove (d); by symmetry, it suffices to prove this for $i = 1$. Let $y_1,y_1' \in X_1$ be distinct; we claim that one of $y_1,y_1'$ dominates the other in $H^*$. Suppose otherwise. Since $X_1$ is a clique that is anticomplete to $V(H^*) \setminus (X_k \cup X_1 \cup X_2)$, it follows that there exist $z,z' \in X_k \cup X_2$ such that $y_1z,y_1'z' \in E(G)$ and $y_1z',y_1'z \notin E(G)$. By symmetry, we may assume that either $z \in X_k$ and $z' \in X_2$, or that $z,z' \in X_2$. Suppose first that $z \in X_k$ and $z' \in X_2$. Then $H' = y_1,y_1',z',x_3,\dots,x_{k-1},z,y_1$ is a hole. Furthermore, since $x_2$ is complete to $X_1$, while $z'y_1 \notin E(G)$, it follows that $x_2 \neq z'$. Thus, we see that $x_2 \notin V(H')$, and that $x_2$ has precisely four neighbors (namely, $y_1,y_1',z',x_3$) in the hole $H'$. Since the hole $H'$ is of length $k+1 \geq 6$, it follows that $(H',x_2)$ is a proper wheel in $G$, a contradiction. Suppose now that $z,z' \in X_2$. But then $G[y_1,y_1',z,z',x_3,\dots,x_k]$ is a $3PC(x_ky_1y_1',x_3zz')$, a contradiction. Thus, one of $y_1,y_1'$ dominates the other in $H^*$, and (d) holds. 

We have now shown that $H^*$ and $(X_1,\dots,X_k)$ satisfy (a)-(d) of Lemma~\ref{lemma-ring-char}, and it follows that $H^*$ is a $k$-ring with good partition $(X_1,\dots,X_k)$. 
\end{proof} 

\begin{lemma} \label{lemma-ring-or-clique-cut-GUT} Let $G \in \mathcal{G}_{\text{UT}}$, and assume that $G$ contains a long hole. Then either some anticomponent of $G$ is a long ring, or $G$ admits a clique-cutset. 
\end{lemma} 
\begin{proof} 
Let $H = x_1,\dots,x_k,x_1$ be a hole of maximum length in $G$ (thus, $k \geq 5$, and $G$ contains no holes of length greater than $k$), and subject to that, assume that $H$ was chosen so that $|V(H^*)|$ is maximum. For all $i \in \mathbb{Z}_k$, set $X_i = X_{x_i}(H)$, and set $K = G[V(H^*) \cup U_{H^*}]$. By Lemma~\ref{lemma-hole-ring-GUT}, $H^*$ is a $k$-ring with good partition $(X_1,\dots,X_k)$; Lemma~\ref{lemma-ring-char} now implies that $X_1,\dots,X_k$ are cliques, and that for all $i \in \mathbb{Z}_k$, $X_i$ is anticomplete to $V(H^*) \setminus (X_{i-1} \cup X_i \cup X_{i+1})$. Clearly, $H^*$ is anticonnected. Thus, the long ring $H^*$ is an anticomponent of $K$, and so if $K = G$, then we are done. So from now on, we assume that $V(K) \subsetneqq V(G)$. 

\begin{quote} 
\emph{(1) $U_{H^*} = U_H$.} 
\end{quote} 
\begin{proof}[Proof of (1)] 
Clearly, $U_{H^*} \subseteq U_H$. Suppose that $U_H \not\subseteq U_{H^*}$, and fix some $x \in U_H \setminus U_{H^*}$. Fix $i \in \mathbb{Z}_k$ and a vertex $y_i \in X_i$ such that $xy_i \notin E(G)$. But now the hole $y_i,x_{i+1},x_{i+2},\dots,x_{i-1},y_i$ and vertex $x$ contradict Lemma~\ref{lemma-vertex-attach-GUT}. This proves (1). 
\end{proof} 

\begin{quote} 
\emph{(2) For all $x \in V(G) \setminus V(K)$, $N_G(x) \cap V(H^*)$ is a clique, and in particular, there exists some $i \in \mathbb{Z}_k$ such that $N_G(x) \cap V(H^*) \subseteq X_i \cup X_{i+1}$.} 
\end{quote} 
\begin{proof}[Proof of (2)] 
Fix $x \in V(G) \setminus V(K)$. By (1), $x$ is not complete to $V(H)$. Since $x \notin V(K)$, we know that $x$ is not a twin of a vertex of $H$ with respect to $H$. Lemma~\ref{lemma-vertex-attach-GUT} now implies that $N_G(x) \cap V(H)$ is a clique of size at most two. 

We first show that there exists some $i \in \mathbb{Z}_k$ such that $N_G(x) \cap V(H^*) \subseteq X_i \cup X_{i+1}$. Suppose otherwise. By symmetry, we may assume that there exists some $j \in \mathbb{Z}_k \setminus \{k,1,2\}$ such that $x$ has a neighbor in $X_1$ and in $X_j$. For each $i \in \{1,j\}$, if $x$ is adjacent to $x_i$, then set $y_i = x_i$, and otherwise, let $y_i$ be any neighbor of $x$ in $X_i$. Since $X_1$ is anticomplete to $X_j$, we have that $y_1y_j \notin E(G)$, and it follows that $Y = y_1,x_2,\dots,x_{j-1},y_j,x_{j+1},\dots,x_k,y_1$ is a $k$-hole. Since $x$ has at most two neighbors in $H$, we know that $x$ is not complete to $Y$. Since $x$ is complete to $\{y_1,y_j\}$, Lemma~\ref{lemma-vertex-attach-GUT} now implies that $x$ is a twin of a vertex of $Y$ with respect to $Y$. It follows that either $j = 3$ and $x$ is a twin of $x_2$ with respect to $Y$ (and in particular, $xx_2 \in E(G)$), or $j = k-1$ and $x$ is a twin of $x_k$ with respect to $Y$ (and in particular $xx_k \in E(G)$); by symmetry, we may assume that the former holds. Since $x$ is not a twin of $x_2$ with respect to $H$, we know that $x$ is nonadjacent to at least one of $x_1,x_3$ (and consequently, either $y_1 \neq x_1$ or $y_3 \neq x_3$). Set $Y_1 = X_{y_1}(Y)$, $Y_3 = X_{y_3}(Y)$, and $Y_i = X_{x_i}(Y)$ for all $i \in \mathbb{Z}_k \setminus \{1,3\}$. By Lemma~\ref{lemma-hole-ring-GUT}, $Y^*$ is $k$-ring with good partition $(Y_1,\dots,Y_k)$. Our goal is to show that $V(H^*) \subsetneqq V(Y^*)$, contrary to the maximality of $|V(H^*)|$. Note that $x \in Y_2$ and $x \notin V(H^*)$, and so $V(H^*) \neq V(Y^*)$. Thus, it suffices to show that $V(H^*) \subseteq V(Y^*)$; we prove this by showing that $X_i \subseteq Y_i$ for all $i \in \mathbb{Z}_k$. 

First of all, in view of Lemma~\ref{lemma-vertex-attach-GUT}, it is easy to see that $X_i = Y_i$ for all $i \in \mathbb{Z}_k \setminus \{k,2,4\}$. Next, we claim that $X_4 \subseteq Y_4$ and $X_k \subseteq Y_k$; by symmetry, it suffices to show that $X_4 \subseteq Y_4$. Fix $y_4 \in X_4$; we must show that $y_4 \in Y_4$. Clearly, it suffices to show that $y_4y_3 \in E(G)$. Suppose otherwise. Since $x_3y_4 \in E(G)$, we see that $y_3 \neq x_3$, and so by the choice of $y_3$, it follows that $xx_3 \notin E(G)$. Furthermore, we have that $xy_4 \notin E(G)$, for otherwise, the hole $y_1,x_2,x_3,y_4,x_5,\dots,x_k,y_1$ and vertex $x$ would contradict Lemma~\ref{lemma-vertex-attach-GUT}. But now $y_1,x,y_3,x_3,y_4,x_5,\dots,x_k,y_1$ is a hole of length $k+1$ in $G$, contrary to the fact that $G$ contains no holes of length greater than $k$. It follows that $X_4 \subseteq Y_4$, and similarly, $X_k \subseteq Y_k$. 

It remains to show that $X_2 \subseteq Y_2$. Suppose otherwise, and fix $z_2 \in X_2 \setminus Y_2$. Then $z_2 \neq x_2$, and furthermore, $z_2$ is complete to $\{x_1,x_2,x_3\}$, anticomplete to $V(H) \setminus \{x_1,x_2,x_3\} = V(Y) \setminus \{y_1,x_2,y_3\}$, and nonadjacent to at least one of $y_1,y_3$. 

Suppose that $xz_2 \notin E(G)$. For each $i \in \{1,3\}$, fix a minimum-length induced path $P_i$ between $x$ and $z_2$, all of whose internal vertices are in $X_i$ (such a path exists because $x$ is adjacent to $y_i \in X_i$, $z_2$ is adjacent to $x_i \in X_i$, and either $x_i = y_i$ or $x_iy_i \in E(G)$; clearly, $P_i$ is of length two or three). But now $G[V(P_1) \cup V(P_3) \cup \{x_4,\dots,x_k\}]$ is a 3PC, a contradiction. Thus, $xz_2 \in E(G)$. 

Suppose that $y_1 \neq x_1$ and $y_3 \neq x_3$, so that (by the choice of $y_1,y_3$) $x$ is anticomplete to $\{x_1,x_3\}$. We know that $z_2$ is nonadjacent to at least one of $y_1,y_3$; by symmetry, we may assume that $z_2y_3 \notin E(G)$. But now $G[x_1,z_2,x,x_3,y_3,x_4,x_5,\dots,x_k]$ is a $3PC(x_3y_3x_4,z_2)$, a contradiction. Thus, either $y_1 = x_1$ or $y_3 = x_3$; by symmetry, we may assume that $y_3 = x_3$, and consequently, $y_1 \neq x_1$. Note that this implies that $z_2x_1,xy_1 \in E(G)$ and $z_2y_1,xx_1 \notin E(G)$. But now $G[x_1,x_3,x_4,\dots,x_k,y_1,z_2,x]$ is a $3PC(xz_2x_3,y_1x_1x_k)$, a contradiction. This proves that there exists some $i \in \mathbb{Z}_k$ such that $N_G(x) \cap V(H^*) \subseteq X_i \cup X_{i+1}$. 

By symmetry, we may now assume that $N_G(x) \cap V(H^*) \subseteq X_1 \cup X_2$. Suppose that $N_G(x) \cap V(H^*)$ is not a clique. Since $X_1$ and $X_2$ are cliques, it follows that there exist nonadjacent vertices $y_1 \in X_1$ and $y_2 \in X_2$ such that $xy_1,xy_2 \in E(G)$. But now $y_1,x,y_2,x_3,x_4,\dots,x_k,y_1$ is a $(k+1)$-hole in $G$, contrary to the fact that $G$ contains no holes of length greater than $k$. This proves (2). 
\end{proof} 

Let $C$ be a component of $G \setminus V(K)$. Our goal is to show that $N_G(C) \cap V(K)$ is a clique. This is enough because it implies that $N_G(C) \cap V(K)$ is a clique-cutset of $G$. 

\begin{quote} 
\emph{(3) $N_G(C) \cap V(H^*)$ is a clique.} 
\end{quote} 
\begin{proof}[Proof of (3)] 
Suppose otherwise. Let $P$ be a minimal connected induced subgraph of $C$ such that $N_G(P) \cap V(H^*)$ is not a clique. Fix $a_1,a_2 \in V(P)$ such that some vertex in $N_G(a_1) \cap V(H^*)$ is nonadjacent to some vertex of $N_G(a_2) \cap V(H^*)$. By (2), $a_1 \neq a_2$. Furthermore, $P$ is a path between $a_1$ and $a_2$, for otherwise, any induced path in $P$ between $a_1$ and $a_2$ would contradict the minimality of $P$. Set $P = p_1,\dots,p_n$ with $p_1 = a_1$ and $p_n = a_2$. 

By the minimality of $P$, we know that $N_G(P \setminus p_1) \cap V(H^*)$ and $N_G(P \setminus p_n) \cap V(H^*)$ are cliques; consequently, $N_G(P) \cap V(H^*)$ is the union of two cliques. Since for every clique $X$ of $H^*$, there exists some $i \in \mathbb{Z}_k$ such that $X \subseteq X_i \cup X_{i+1}$, we deduce that there exist at most four indices $i \in \mathbb{Z}_k$ such that $N_G(P) \cap X_i \neq \emptyset$; since $k \geq 5$, we deduce that there exists an index $i \in \mathbb{Z}_k$ such that $N_G(P) \cap X_i = \emptyset$. On the other hand, since each $X_i$ is a clique and $N_G(P) \cap V(H^*)$ is not a clique, we see that there exist at least two indices $i \in \mathbb{Z}_k$ such that $N_G(P) \cap X_i \neq \emptyset$. 

Now, let $X_i,X_{i+1},\dots,X_j$ be a sequence of maximum length having the property that $N_G(P)$ intersects both $X_i$ and $X_j$, but fails to intersect $X_{i+1} \cup \dots \cup X_{j-1}$. By what we just showed, the length of the sequence $X_i,X_{i+1},\dots,X_j$ is at least three, and at most $k$; in particular, $i \neq j$. Furthermore, $N_G(P) \cap V(H^*) \subseteq X_j \cup X_{j+1} \cup X_{i-1} \cup X_i$. 

Fix nonadjacent vertices $y_i \in N_G(P) \cap X_i$ and $y_j \in N_G(P) \cap X_j$. (If $i \neq j+1$, then any two vertices $y_i \in N_G(P) \cap X_i$ and $y_j \in N_G(P) \cap X_j$ are nonadjacent. On the other hand, if $i = j+1$, then we have that $N_G(P) \cap V(H^*) \subseteq X_i \cup X_j$, and the existence of $y_i$ and $y_j$ follows from the fact that $N_G(P) \cap V(H^*)$ is not a clique, whereas both $X_i$ and $X_j$ are cliques.) By the minimality of $P$, all interior vertices of $P$ are anticomplete to $\{y_i,y_j\}$, and either $p_1y_i,p_ny_j \in E(G)$ and $p_1y_j,p_ny_i \notin E(G)$, or $p_1y_j,p_ny_i \in E(G)$ and $p_1y_i,p_ny_j \notin E(G)$; by symmetry, we may assume that the latter holds, that is, that $p_1y_j,p_ny_i \in E(G)$ and $p_1y_i,p_ny_j \notin E(G)$. Then $y_i,x_{i+1},\dots,x_{j-1},y_j,p_1,\dots,p_n,y_i$ is a hole in $G$, and its length is the sum of $n$ and the length of the sequence $X_i,\dots,X_j$. Since $G$ contains no holes of length greater than $k$, we see that the length of the sequence $X_i,\dots,X_j$ is at most $k-n \leq k-2$, and it follows that the cliques $X_j,X_{j+1},X_{i-1},X_i$ are pairwise distinct. 

Now, recall that $N_G(P) \cap V(H^*) \subseteq X_j \cup X_{j+1} \cup X_{i-1} \cup X_i$, and that $N_G(P \setminus p_1) \cap V(H^*)$ and $N_G(P \setminus p_n) \cap V(H^*)$ are both cliques. Since $p_1$ has a neighbor in $X_j$, and $p_n$ has a neighbor in $X_i$, we deduce that $N_G(P \setminus p_n) \cap V(H^*) \subseteq X_j \cup X_{j+1}$ and $N_G(P \setminus p_1) \cap V(H^*) \subseteq X_{i-1} \cup X_i$, and it follows that $N_G(P \setminus \{p_1,p_n\}) \cap V(H^*) \subseteq (X_j \cup X_{j+1}) \cap (X_{i-1} \cup X_i) = \emptyset$. Thus, the interior vertices of $P$ are anticomplete to $V(H^*)$. We also know that $N_G(p_1) \cap V(H^*) \subseteq X_j \cup X_{j+1}$ and $N_G(p_n) \cap V(H^*) \subseteq X_{i-1} \cup X_i$. But now $G[V(P) \cup \{y_i,x_{i+1},\dots,x_{j-1},y_j,x_{j+1},\dots,x_{i-1}\}]$ is a 3PC, a contradiction. This proves (3). 
\end{proof} 

\begin{quote} 
\emph{(4) $N_G(C) \cap V(K)$ is a clique.}
\end{quote} 
\begin{proof}[Proof of (4)]  
In view of (3), it suffices to show that $N_G(C) \cap U_{H^*}$ is a clique. Suppose otherwise, and fix a minimal connected induced subgraph $P$ of $C$ such that $N_G(P) \cap U_{H^*}$ is not a clique. Fix nonadjacent vertices $u_1,u_2 \in N_G(P) \cap U_{H^*}$, and fix (not necessarily distinct) vertices $a_1,a_2 \in V(P)$ such that $a_1u_1,a_2u_2 \in E(G)$. 
It is clear that $P$ is a path between $a_1$ and $a_2$ (if $a_1 = a_2$, then $P$ is a one-vertex path), for otherwise, any induced path in $P$ between $a_1$ and $a_2$ would contradict the minimality of $P$. Set $P = p_1,\dots,p_n$ (with $n \geq 1$) so that $p_1 = a_1$ and $p_n = a_2$. By the minimality of $P$, $u_1$ is anticomplete to $V(P) \setminus \{a_1\}$, and $u_2$ is anticomplete to $V(P) \setminus \{a_2\}$. Thus, $P' = u_1,p_1,\dots,p_n,u_2$ is an induced path in $G$. 

Since $N_G(C) \cap V(H^*)$ is a clique, we know that there exists some $i \in \mathbb{Z}_k$ such that $N_G(C) \cap V(H^*) \subseteq X_i \cup X_{i+1}$; by symmetry we may assume that $N_G(C) \cap V(H^*) \subseteq X_1 \cup X_2$. But now $G[V(P) \cup \{u_1,u_2,x_3,x_5\}]$ is a $3PC(u_1,u_2)$, a contradiction. This proves (4). 
\end{proof} 

Since $K$ is not a complete graph, (4) implies that $N_G(C) \cap V(K)$ is a clique-cutset of $G$. This completes the argument. 
\end{proof} 

We remind the reader that $\mathcal{B}_{\text{UT}}$ is the class of all graphs $G$ that satisfy at least one of the following: 
\begin{itemize} 
\item $G$ has exactly one nontrivial anticomponent, and this anticomponent is a long ring; 
\item $G$ is (long hole, $K_{2,3}$, $\overline{C_6}$)-free; 
\item $\alpha(G) = 2$, and every anticomponent of $G$ is either a 5-hyperhole or a $(C_5,\overline{C_6})$-free graph. 
\end{itemize} 
We are now ready to prove Theorem~\ref{decomp-thm-GUT}, restated below for the reader's convenience. 

\begin{decomp-thm-GUT} Every graph in $\mathcal{G}_{\text{UT}}$ either belongs to $\mathcal{B}_{\text{UT}}$ or admits a clique-cutset. 
\end{decomp-thm-GUT}
\begin{proof} 
Fix $G \in \mathcal{G}_{\text{UT}}$. We assume that $G$ does not admit a clique-cutset, and we show that $G \in \mathcal{B}_{\text{UT}}$. Clearly, $G$ is $(K_{2,3},\overline{C_6})$-free. If $G$ contains no long holes, then $G \in \mathcal{B}_{\text{UT}}$, and we are done. So assume that $G$ contains a long hole. By Lemma~\ref{lemma-ring-or-clique-cut-GUT}, some anticomponent of $G$ is a long ring; if this anticomponent is the only nontrivial anticomponent of $G$, then $G \in \mathcal{B}_{\text{UT}}$, and we are done. So assume that $G$ has at least two nontrivial anticomponents. Lemma~\ref{rmrk-K23-anticomp} then implies that $\alpha(G) = 2$. We claim that every anticomponent of $G$ is either a 5-hyperhole or a $(C_5,\overline{C_6})$-free graph (this will imply that $G \in \mathcal{B}_{\text{UT}}$). Let $H$ be an anticomponent of $G$. If $H$ contains no long holes, then $H$ is $(C_5,\overline{C_6})$-free, and we are done. So assume that $H$ does contain a long hole. Since $\alpha(H) \leq \alpha(G) = 2$, Lemma~\ref{lemma-clique-cut-alpha-2}(a) implies that $H$ does not admit a clique-cutset, and so by Lemma~\ref{lemma-ring-or-clique-cut-GUT}, $H$ is a long ring. Since $\alpha(H) \leq \alpha(G) = 2$, we deduce that $H$ is a 5-hyperhole (indeed, any long ring other than a 5-hyperhole contains a stable set of size three). This completes the argument. 
\end{proof}

\section{A decomposition theorem for the class $\mathcal{G}_{\text{U}}$}\label{sec:decompGU}

Our goal in this section is to prove Theorem~\ref{decomp-thm-GU}, which states that every graph in $\mathcal{G}_{\text{U}}$ either belongs to $\mathcal{B}_{\text{U}}$ or admits a clique-cutset. 

\begin{lemma} \label{square-clique-cut-GU} Let $G \in \mathcal{G}_{\text{U}}$, and let $H = x_1,x_2,x_3,x_4,x_1$ be a 4-hole in $G$. Then either $V(G) = V(H) \cup U_H$, or $G$ admits a clique-cutset. 
\end{lemma} 
\begin{proof} 
We may assume that $V(H) \cup U_H \subsetneqq V(G)$, for otherwise we are done. 

\begin{quote} 
\emph{(1) For all $x \in V(G) \setminus (V(H) \cup U_H)$, there exists some $i \in \mathbb{Z}_4$ such that $N_G(x) \cap V(H) \subseteq \{x_i,x_{i+1}\}$.}
\end{quote} 
\begin{proof}[Proof of (1)] 
Fix $x \in V(G) \setminus (V(H) \cup U_H)$. Then there exists some $i \in \mathbb{Z}_4$ such that either $N_G(x) \cap V(H) \subseteq \{x_i,x_{i+1}\}$, or $N_G(x) \cap V(H) = \{x_i,x_{i+2}\}$, or $N_G(x) \cap V(H) = \{x_{i-1},x_i,x_{i+1}\}$. In the first case, we are done. In the second case, $G[x_1,x_2,x_3,x_4,x]$ is a $3PC(x_i,x_{i+2})$, a contradiction. In the third case, $(H,x)$ is a twin wheel in $G$, again a contradiction. This proves (1). 
\end{proof} 

Let $C$ be a component of $G \setminus (V(H) \cup U_H)$. 

\begin{quote} 
\emph{(2) $N_G(C) \cap V(H)$ is a clique.} 
\end{quote} 
\begin{proof}[Proof of (2)]  
Suppose otherwise, and fix a minimal connected induced subgraph $P$ of $C$ such that $N_G(P) \cap V(H)$ is not a clique. Then for some $i \in \mathbb{Z}_4$, we have that $x_i,x_{i+2} \in N_G(P) \cap V(H)$; by symmetry, we may assume that $x_1,x_3 \in N_G(P) \cap V(H)$. Fix $a_1,a_3 \in V(P)$ such that $a_1x_1,a_3x_3 \in E(G)$; by (1), we have that $a_1x_3,a_3x_1 \notin E(G)$, and in particular, $a_1 \neq a_3$. Clearly, $P$ is a path between $a_1$ and $a_3$, for otherwise, any induced path in $P$ between $a_1$ and $a_3$ would contradict the minimality of $P$. Further, the minimality of $P$ implies that all interior vertices of $P$ are anticomplete to $\{x_1,x_3\}$. Set $P = p_1,\dots,p_n$, with $p_1 = a_1$ and $p_n = a_3$. 

Suppose first that both $x_2$ and $x_4$ have a neighbor in $P$. Then both $x_2,x_4$ are anticomplete to the interior of $P$. (Indeed, suppose that some interior vertex $p$ of $P$ is adjacent to $x_2$, and let $p'$ be any vertex of $P$ adjacent to $x_4$. Then the subpath of $P$ between $p$ and $p'$ contradicts the minimality of $P$. Similarly, no interior vertex of $P$ is adjacent to $x_4$.) By (1), each of $a_1,a_3$ is adjacent to at most one of $x_2,x_4$; by symmetry, we may now assume that $N_G(a_1) \cap V(H) = \{x_1,x_2\}$ and $N_G(a_3) \cap V(H) = \{x_3,x_4\}$. But now $G[V(H) \cup V(P)]$ is a $3PC(a_1x_1x_2,a_3x_4x_3)$, a contradiction. 

From now on, we assume that at most one of $x_2,x_4$ has a neighbor in $P$; by symmetry, we may assume that $x_2$ is anticomplete to $V(P)$. Now, if $x_4$ has a neighbor in $P$, then we observe that $H' = x_3,x_2,x_1,p_1,\dots,p_n,x_3$ is a hole and $(H',x_4)$ a proper wheel in $G$, a contradiction. On the other hand, if $x_4$ has no neighbors in $P$, then $G[V(H) \cup V(P)]$ is a $3PC(x_1,x_3)$, again a contradiction. This proves (2). 
\end{proof} 

\begin{quote} 
\emph{(3) $N_G(C) \cap (V(H) \cup U_H)$ is a clique.} 
\end{quote} 
\begin{proof}[Proof of (3)] 
In view of (2), we need only show that $N_G(C) \cap U_H$ is a clique. Suppose otherwise, and let $P$ be a minimal connected induced subgraph of $C$ such that $N_G(P) \cap U_H$ is not a clique. Fix nonadjacent vertices $u_1,u_2 \in N_G(P) \cap U_H$, and fix (not necessarily distinct) vertices $a_1,a_2 \in V(P)$ such that $a_1u_1,a_2u_2 \in E(G)$. Clearly, $P$ is a path between $a_1$ and $a_2$ (if $a_1 = a_2$, then $P$ is a one-vertex path), for otherwise, any induced path between $a_1$ and $a_2$ in $P$ would contradict the minimality of $P$. Set $P = p_1,\dots,p_n$ with $p_1 = a_1$ and $p_n = a_2$. By the minimality of $P$, we have that $P' = u_1,p_1,\dots,p_n,u_2$ is an induced path in $G$. By (2), and by symmetry, we may assume that $N_G(C) \cap V(H) \subseteq \{x_3,x_4\}$. Then $H' = x_1,u_1,p_1,\dots,p_n,u_2,x_1$ is a hole in $G$, and $x_2 \in X_{x_1}(H')$. Thus, $(H',x_2)$ a twin wheel in $G$, a contradiction. This proves (3). 
\end{proof} 

Clearly, (3) implies that $N_G(C) \cap (V(H) \cup U_H)$ is a clique-cutset of $G$. 
\end{proof} 

We remind the reader that $\mathcal{B}_{\text{U}}$ is the class of all graphs $G$ that satisfy one of the following: 
\begin{itemize} 
\item $G$ has exactly one nontrivial anticomponent, and this anticomponent is a long hole; 
\item all nontrivial anticomponents of $G$ are isomorphic to $\overline{K_2}$. 
\end{itemize} 
We are now ready to prove Theorem~\ref{decomp-thm-GU}, restated below for the reader's convenience. 

\begin{decomp-thm-GU} Every graph in $\mathcal{G}_{\text{U}}$ either belongs to $\mathcal{B}_{\text{U}}$ or admits a clique-cutset. 
\end{decomp-thm-GU}
\begin{proof} 
Fix $G \in \mathcal{G}_{\text{U}}$, and assume that $G$ does not admit a clique-cutset; we must show that $G \in \mathcal{B}_{\text{U}}$. 

\begin{quote} 
\emph{(1) If some anticomponent of $G$ contains more than two vertices, then all other anticomponents of $G$ are trivial.} 
\end{quote} 
\begin{proof}[Proof of (1)] 
Suppose otherwise. Then $G$ has at least two nontrivial anticomponents, and so by Lemma~\ref{rmrk-K23-anticomp}, $\alpha(G) = 2$. Let $C_1$ be an anticomponent of $G$ that contains at least three vertices, and let $C_2$ be some other nontrivial anticomponent of $G$. Since $\alpha(G) = 2$ and the anticomponents $C_1,C_2$ are nontrivial, we have that $\alpha(C_1) = \alpha(C_2) = 2$. Since $|V(C_1)| \geq 3$, we deduce that $C_1$ is not edgeless, and so since $C_1$ is anticonnected, it follows that there exist pairwise distinct vertices $a,b,c \in V(C_1)$ such that $ab,bc \notin E(G)$ and $ac \in E(G)$. Fix nonadjacent vertices $x,y \in V(C_2)$. But now $H = a,x,b,y,a$ is a hole and $(H,c)$ a twin wheel in $G$, a contradiction. This proves (1). 
\end{proof} 

Suppose first that $G$ contains a 4-hole $H$. Then by Lemma~\ref{square-clique-cut-GU}, $V(G) = V(H) \cup U_H$. $H$ has two anticomponents, both isomorphic to $\overline{K_2}$, and clearly, these anticomponents of $H$ are also anticomponents of $G$. It now follows from (1) that no anticomponent of $G$ has more than two vertices. Thus, all nontrivial anticomponents of $G$ are isomorphic to $\overline{K_2}$, and it follows that $G \in \mathcal{B}_{\text{U}}$. 

Suppose next that $G$ contains a long hole. Then by Lemma~\ref{lemma-ring-or-clique-cut-GUT}, some anticomponent $H$ of $G$ is a long ring. But then $H$ is a long hole, for otherwise, the ring $H$ would contain a twin wheel. By (1), $H$ is the only nontrivial anticomponent of $G$. Thus, $G \in \mathcal{B}_{\text{U}}$. 

It remains to consider the case when $G$ contains no holes. But then by definition, $G$ is chordal. Since $G$ does not admit a clique-cutset, Theorem~\ref{thm-Dirac61} implies that $G$ is a complete graph, and consequently, $G \in \mathcal{B}_{\text{U}}$. This completes the argument. 
\end{proof}

\section{A decomposition theorem for the class $\mathcal{G}_{\text{T}}$}\label{sec:decompGT}

In this section, we prove Theorem~\ref{decomp-thm-GT}, which states that every graph in $\mathcal{G}_{\text{T}}$ either belongs to $\mathcal{B}_{\text{T}}$ or admits a clique-cutset. 

\begin{lemma} \label{lemma-long-hole-GT} Let $G \in \mathcal{G}_{\text{T}}$. Then $G$ contains no antiholes of length six, and no antiholes of length greater than seven. Furthermore, if $G$ contains a long hole, then either $G$ is a long ring, or $G$ admits a clique-cutset. 
\end{lemma} 
\begin{proof} 
Since $\overline{C_6}$ is a prism, we see that $G$ contains no antiholes of length six. Furthermore, we observe that if $A = x_1,\dots,x_k,x_1$ (with $k \geq 8$) is an antihole in $G$, then $H = x_1,x_4,x_2,x_5,x_1$ is a 4-hole and $(H,x_7)$ a universal wheel in $G$, a contradiction. This proves the first statement. 

It remains to prove the second statement. Suppose that $G$ contains a long hole. Then by Lemma~\ref{lemma-ring-or-clique-cut-GUT}, either some anticomponent of $G$ is a ring, or $G$ admits a clique-cutset. In the latter case, we are done; so assume that some anticomponent of $G$, call it $R$, is a ring. If $U_R \neq \emptyset$, then $G$ contains a universal wheel, a contradiction. Thus, $U_R = \emptyset$, and it follows that $G = R$. Thus, $G$ is a ring. 
\end{proof} 

\begin{lemma} \label{lemma-7antihole-GT} Let $G \in \mathcal{G}_{\text{T}}$, and assume that $G$ contains no long holes, but does contain a 7-antihole. Then either $G$ is a 7-hyperantihole, or $G$ admits a clique-cutset. 
\end{lemma} 
\begin{proof} 
Let $A = x_1,x_2,\dots,x_7,x_1$ be a 7-antihole in $G$, and for all $i \in \mathbb{Z}_7$, set $X_i = X_{x_i}(A)$. Thus, $A^* = G[\bigcup_{i \in \mathbb{Z}_7} X_i]$. 

\begin{quote} 
\emph{(1) $A^*$ is a 7-hyperantihole with good partition $(X_1,X_2,\dots,X_7)$.} 
\end{quote} 
\begin{proof}[Proof of (1)]  
By symmetry, it suffices to show that $X_1$ is a clique, complete to $X_3 \cup X_4 \cup X_5 \cup X_6$ and anticomplete to $X_2 \cup X_7$. 

Suppose that $X_1$ is not a clique, and fix nonadjacent vertices $y_1,y_1' \in X_1$. By construction, $x_1$ is complete to $X_1 \setminus \{x_1\}$, and so $x_1 \notin \{y_1,y_1'\}$. But now $H = y_1,x_3,y_1',x_4,y_1$ is a 4-hole and $(H,x_1)$ a universal wheel in $G$, a contradiction. 

Next, suppose that $X_1$ is not anticomplete to $X_2 \cup X_7$; by symmetry, we may assume that there exist some $y_1 \in X_1$ and $y_2 \in X_2$ such that $y_1y_2 \in E(G)$. But now $H = y_2,x_5,x_3,x_6,y_2$ is a 4-hole and $(H,y_1)$ a universal wheel in $G$, a contradiction. 

Further, suppose that $X_1$ is not complete to $X_3 \cup X_6$; by symmetry, we may assume that some $y_1 \in X_1$ and $y_3 \in X_3$ are nonadjacent. Since $x_1$ is complete to $X_3$, we have that $y_1 \neq x_1$. But then $H = y_1,x_5,y_3,x_6,y_1$ is a 4-hole and $(H,x_1)$ a universal wheel in $G$, a contradiction. 

It remains to show that $X_1$ is complete to $X_4 \cup X_5$. Suppose otherwise; by symmetry, we may assume that some $y_1 \in X_1$ and $y_4 \in X_4$ are nonadjacent. But now $y_1,x_5,x_7,y_4,x_6,y_1$ is a 5-hole in $G$, contrary to the fact that $G$ contains no long holes. This proves (1). 
\end{proof} 

\begin{quote} 
\emph{(2) For all $x \in V(G) \setminus V(A^*)$, and all $i \in \mathbb{Z}_7$, if $x$ has a neighbor both in $X_i$ and in $X_{i+1}$, then either $x$ is complete to $X_{i-2} \cup X_{i+3}$ and anticomplete to $X_{i-3}$, or $x$ is complete to $X_{i-3}$ and anticomplete to $X_{i-2} \cup X_{i+3}$.} 
\end{quote} 
\begin{proof}[Proof of (2)] 
Fix $x \in V(G) \setminus V(A^*)$, and assume that for some $i \in \mathbb{Z}_7$, $x$ has a neighbor both in $X_i$ and $X_{i+1}$; by symmetry, we may assume that $x$ is adjacent to some $y_1 \in X_1$ and to some $y_2 \in X_2$. We must show that $x$ is complete to one of $X_4 \cup X_6$ and $X_5$, and anticomplete to the other. 

Fix $j \in \{4,5\}$, and suppose that $x$ is adjacent to some $y_j \in X_j$ and to some $y_{j+1} \in X_{j+1}$; then, by (1), $H = y_1,y_j,y_2,y_{j+1},y_1$ is a 4-hole and $(H,x)$ a universal wheel in $G$, a contradiction. Thus, $x$ has a neighbor in at most one of $X_j$ and $X_{j+1}$. Suppose now that $x$ has a nonneighbor $y_j' \in X_j$ and a nonneighbor $y_{j+1}' \in X_{j+1}$. But then, by (1), $G[y_1,y_2,y_j',y_{j+1}',x]$ is a $K_{2,3}$, a contradiction. Thus, $x$ has a nonneighbor in at most one of $X_j$ and $X_{j+1}$. It now follows that $x$ is complete to one of $X_j$ and $X_{j+1}$, and anticomplete to the other. 

We now have that $x$ is complete to one of $X_4$ and $X_5$, and anticomplete to the other, and we also have that $x$ is complete to one of $X_5$ and $X_6$, and anticomplete to the other. It follows that $x$ is complete to one of $X_4 \cup X_6$ and $X_5$, and anticomplete to the other. This proves (2). 
\end{proof} 

\begin{quote} 
\emph{(3) For all $x \in V(G) \setminus V(A^*)$, and all $i \in \mathbb{Z}_7$, if $x$ has a neighbor both in $X_i$ and in $X_{i+1}$, then $x$ is complete to at least one of $X_{i-1}$ and $X_{i+2}$.}
\end{quote} 
\begin{proof}[Proof of (3)] 
Suppose otherwise. By symmetry, we may assume that some vertex $x \in V(G) \setminus V(A^*)$ has a neighbor both in $X_1$ and in $X_2$, and a nonneighbor both in $X_3$ and in $X_7$. Fix $y_1 \in X_1$, $y_2 \in X_2$, $y_3 \in X_3$, and $y_7 \in X_7$ such that $xy_1,xy_2 \in E(G)$, and $xy_3,xy_7 \notin E(G)$. But now, by (1), $x,y_1,y_3,y_7,y_2,x$ is a 5-hole in $G$, contrary to the fact that $G$ contains no long holes. This proves (3). 
\end{proof} 

\begin{quote} 
\emph{(4) For all $x \in V(G) \setminus V(A^*)$, $N_G(x) \cap V(A^*)$ is a clique.}
\end{quote} 
\begin{proof}[Proof of (4)] 
Fix $x \in V(G) \setminus V(A^*)$, and suppose that $N_G(x) \cap V(A^*)$ is not a clique. By (1), and by symmetry, we may assume that there exist $y_1 \in X_1$ and $y_2 \in X_2$ such that $xy_1,xy_2 \in E(G)$. By (3) and by symmetry, we may assume that $x$ is complete to $X_3$. By (2), with $i = 1$, we have that $x$ is complete to one of $X_4 \cup X_6$ and $X_5$, and anticomplete to the other. 

Suppose first that $x$ is complete to $X_4 \cup X_6$ and anticomplete to $X_5$. By (2), with $i = 2$, we see that $x$ is anticomplete to $X_7$. By (2), with $i = 3$, $x$ is complete to $X_1$. By (3), with $i = 3$, $x$ is complete to $X_2$. We now have that $x$ is complete to $X_1 \cup X_2 \cup X_3 \cup X_4 \cup X_6$ and anticomplete to $X_5 \cup X_7$. But then $x$ is a twin of $x_6$ with respect to $A$, and so $x \in X_6$, contrary to the fact that $x \notin V(A^*)$. 

Suppose now that $x$ is complete to $X_5$ and anticomplete to $X_4 \cup X_6$. By (2), with $i = 2$, we see that $x$ is complete to $X_7$. By (3), with $i = 2$, we see that $x$ is complete to $X_1$. By (3), with $i = 7$, we see that $x$ is complete to $X_2$. But now $x$ is complete to $X_1 \cup X_2 \cup X_3 \cup X_5 \cup X_7$ and anticomplete to $X_4 \cup X_6$. It follows that $x$ is a twin of $x_5$ with respect to $A$, and so $x \in X_5$, contrary to the fact that $x \notin V(A^*)$. This proves (4). 
\end{proof} 

\begin{quote} 
\emph{(5) For every component $C$ of $G \setminus V(A^*)$, $N_G(C) \cap V(A^*)$ is a clique.} 
\end{quote} 
\begin{proof}[Proof of (5)] 
Suppose otherwise. Fix a minimal connected induced subgraph $P$ of $G \setminus V(A^*)$ such that $N_G(P) \cap V(A^*)$ is not a clique. By (1) and by symmetry, we may assume that $N_G(P) \cap X_1 \neq \emptyset$ and $N_G(P) \cap X_2 \neq \emptyset$. Fix $a_1 \in V(P)$ such that $a_1$ has a neighbor $y_1 \in X_1$, and fix $a_2 \in V(P)$ such that $a_2$ has a neighbor $y_2 \in X_2$. By (1) and (4), $a_1$ is anticomplete to $X_2 \cup X_7$, and $a_2$ is anticomplete to $X_1 \cup X_3$; it follows that $a_1 \neq a_2$. Clearly, $P$ is a path between $a_1$ and $a_2$, for otherwise, any induced path in $P$ between $a_1$ and $a_2$ would contradict the minimality of $P$. Set $P = p_1,\dots,p_n$, with $p_1 = a_1$ and $p_n = a_2$ (thus, $n \geq 2$). By the minimality of $P$, and by (1), each interior vertex of $P$ is anticomplete to $X_7 \cup X_1 \cup X_2 \cup X_3$, for if some interior vertex $p$ of $P$ had a neighbor in $X_7 \cup X_1 \cup X_2 \cup X_3$, then either the subpath of $P$ between $a_1$ and $p$, or the subpath of $P$ between $a_2$ and $p$, would contradict the minimality of $P$. We now observe the following: 
\begin{itemize} 
\item[(i)] if $a_1x_3,a_2x_7 \notin E(G)$, then $y_1,p_1,\dots,p_n,y_2,x_7,x_3,y_1$ is an $(n+4)$-hole in $G$; 
\item[(ii)] if $a_1x_3 \in E(G)$ and $a_2x_7 \notin E(G)$, then $x_3,p_1,\dots,p_n,y_2,x_7,x_3$ is an $(n+3)$-hole in $G$; 
\item[(iii)] if $a_1x_3 \notin E(G)$ and $a_2x_7 \in E(G)$, then $y_1,p_1,\dots,p_n,x_7,x_3,y_1$ is an $(n+3)$-hole in $G$; 
\item[(iv)] if $a_1x_3,a_2x_7 \in E(G)$, then $x_3,p_1,\dots,p_n,x_7,x_3$ is an $(n+2)$-hole in $G$. 
\end{itemize} 
Since $n \geq 2$ and $G$ contains no long holes, we deduce that (iv) holds, with $n = 2$. (Thus, $P = a_1,a_2$.) 

Now, if some $x \in \{x_4,x_5,x_6\}$ is anticomplete to $\{a_1,a_2\}$, then $x,y_1,a_1,a_2,y_2,x$ is a 5-hole in $G$, a contradiction. Thus each of $x_4,x_5,x_6$ has a neighbor in $\{a_1,a_2\}$. By symmetry, we may assume that $x_5a_1 \in E(G)$. We now have that $a_1x_5,a_2x_7 \in E(G)$, and so since $x_6x_5,x_6x_7 \notin E(G)$, (4) implies that $a_1x_6,a_2x_6 \notin E(G)$, contrary to the fact that $x_6$ has a neighbor in $\{a_1,a_2\}$. This proves (5). 
\end{proof} 

Clearly, (1) and (5) together imply that either $G$ is a 7-hyperantihole, or $G$ admits a clique-cutset. 
\end{proof} 

\begin{lemma} \label{lemma-4-hole-ring-GT} Let $G \in \mathcal{G}_{\text{T}}$, and let $H = x_1,x_2,x_3,x_4,x_1$ be a 4-hole in $G$. For each $i \in \mathbb{Z}_4$, set $X_i = X_{x_i}(H)$. Then $H^*$ is a 4-ring with good partition $(X_1,X_2,X_3,X_4)$. 
\end{lemma} 
\begin{proof} 
Our goal is to show that $H^*$ and $(X_1,X_2,X_3,X_4)$ satisfy (a)-(d) from Lemma~\ref{lemma-ring-char}. Clearly, for all $i \in \mathbb{Z}_4$, we have that $N_{H^*}[x_i] = X_{i-1} \cup X_i \cup X_{i+1}$, and in particular, that $x_i$ is complete to $X_{i-1} \cup X_{i+1}$; thus, (c) holds. Further, by Lemma~\ref{lemma-hole-ring-GUT}, $X_1,X_2,X_3,X_4$ are cliques, and so (a) holds. 

Next, we show that (b) holds. By symmetry, it suffices to show that $X_1$ is anticomplete to $X_3$. Suppose otherwise, and fix $y_1 \in X_1$ and $y_3 \in X_3$ such that $y_1y_3 \in E(G)$. By construction, $x_1$ is anticomplete to $X_3$, and $x_3$ is anticomplete to $X_1$, and so we see that $y_1 \neq x_1$ and $y_3 \neq x_3$. But now $H' = y_1,x_2,x_3,x_4,y_1$ is a hole and $(H',y_3)$ a universal wheel in $G$, a contradiction. Thus, (b) holds. 

It remains to show that (d) holds; by symmetry, it suffices to prove (d) for $i = 1$. Fix distinct $y_1,y_1' \in X_1$; we claim that one of $y_1,y_1'$ dominates the other in $H^*$. Suppose otherwise. By (a), $y_1y_1' \in E(G)$, and by (b), both $y_1$ and $y_1'$ are anticomplete to $X_3$. Thus, by symmetry, we may assume that one of the following holds: 
\begin{itemize} 
\item[(i)] there exist $y_2,y_2' \in X_2$ such that $y_1y_2,y_1'y_2' \in E(G)$ and $y_1y_2',y_1'y_2 \notin E(G)$; 
\item[(ii)] there exist $y_2 \in X_2$ and $y_4' \in X_4$ such that $y_1y_2,y_1'y_4' \in E(G)$ and $y_1y_4',y_1'y_2 \notin E(G)$. 
\end{itemize} 
If (i) holds, then $G[y_1,y_1',y_2,y_2',x_3,x_4]$ is a $3PC(y_1y_1'x_4,y_2y_2'x_3)$, a contradiction. Suppose now that (ii) holds. Since $x_1$ is complete to $X_2 \cup X_4$, we have that $x_1 \notin \{y_1,y_1'\}$. Using (a) and (b), we now deduce that $H' = y_1,y_2,x_3,y_4',y_1',y_1$ is a 5-hole in $G$, and $x_1$ has precisely four neighbors (namely, $y_1,y_1',y_2,y_4'$) in $V(H')$; thus, $(H',x_1)$ is a proper wheel in $G$, a contradiction. It follows that one of $y_1,y_1'$ dominates the other in $H^*$. This proves (d). 

Lemma~\ref{lemma-ring-char} now implies that $H^*$ is a 4-ring with good partition $(X_1,X_2,X_3,X_4)$. 
\end{proof} 

\begin{lemma} \label{lemma-4-hole-vertex-attach-GT} Let $G \in \mathcal{G}_{\text{T}}$, assume that $G$ contains no long holes and no 7-antiholes, and let $H = x_1,x_2,x_3,x_4$ be a 4-hole in $G$, chosen so that $|V(H^*)|$ is maximum. Then for all $x \in V(G) \setminus V(H^*)$, $N_G(x) \cap V(H^*)$ is a clique. 
\end{lemma} 
\begin{proof} 
For each $i \in \mathbb{Z}_4$, let $X_i = X_{x_i}(H)$. By Lemma~\ref{lemma-4-hole-ring-GT}, $H^*$ is a 4-ring with good partition $(X_1,X_2,X_3,X_4)$; in particular, $X_1,X_2,X_3,X_4$ are cliques, $X_1$ is anticomplete to $X_3$, and $X_2$ is anticomplete to $X_4$. 

Suppose that for some $x \in V(G) \setminus V(H^*)$, $N_G(x) \cap V(H^*)$ is not a clique. Suppose first that $N_G(x) \cap V(H^*) \subseteq X_i \cup X_{i+1}$ for some $i \in \mathbb{Z}_4$; by symmetry, we may assume that $N_G(x) \cap V(H^*) \subseteq X_1 \cup X_2$. Since $N_G(x) \cap V(H^*)$ is not a clique, there exist nonadjacent vertices $y_1 \in X_1$ and $y_2 \in X_2$ such that $xy_1,xy_2 \in E(G)$. But now $x,y_2,x_3,x_4,y_1,x$ is a 5-hole in $G$, contrary to the fact that $G$ contains no long holes. It follows that for some $i \in \mathbb{Z}_4$, $x$ has a neighbor both in $X_i$ and in $X_{i+2}$. 

By symmetry, we may assume that $x$ has a neighbor both in $X_1$ and in $X_3$. For each $i \in \{1,3\}$, if $xx_i \in E(G)$, then set $y_i = x_i$, and otherwise, let $y_i$ be any neighbor of $x$ in $X_i$. Note that if $x$ were complete to $\{x_2,x_4\}$, then $H' = y_1,x_2,y_3,x_4,y_1$ would be a hole and $(H',x)$ a universal wheel in $G$, a contradiction. On the other hand, if $x$ were anticomplete to $\{x_2,x_4\}$, then $G[y_1,y_3,x,x_2,x_4]$ would be a $K_{2,3}$, a contradiction. Thus, $x$ is adjacent to precisely one of $x_2,x_4$; by symmetry, we may assume that $x$ is adjacent to $x_2$ and nonadjacent to $x_4$. Further, note that $x$ is adjacent to at most one of $x_1,x_3$, for otherwise, $x$ would be a twin of $x_2$ with respect to $H$, and we would have that $x \in X_2$, a contradiction. Thus, either $y_1 \neq x_1$ or $y_3 \neq x_3$. Now, $Y = y_1,x_2,y_3,x_4,y_1$ is a 4-hole in $G$. Our goal is to show that $V(H^*) \subsetneqq V(Y^*)$, contrary to the maximality of $|V(H^*)|$. 

For $i \in \{1,3\}$, set $Y_i = X_{y_i}(Y)$, and for $i \in \{2,4\}$, set $Y_i = X_{x_i}(Y)$. By Lemma~\ref{lemma-4-hole-ring-GT}, $Y^*$ is a 4-ring with good partition $(Y_1,Y_2,Y_3,Y_4)$; in particular, $Y_1,Y_2,Y_3,Y_4$ are cliques, $Y_1$ is anticomplete to $Y_3$, and $Y_2$ is anticomplete to $Y_4$. Now, to show that $V(H^*) \subsetneqq V(Y^*)$, it suffices to show that $X_i \subseteq Y_i$ for all $i \in \{1,3,4\}$, and that $X_2 \subsetneqq Y_2$. 

\begin{quote} 
\emph{(1) $X_1 = Y_1$ and $X_3 = Y_3$.} 
\end{quote} 
\begin{proof}[Proof of (1)] 
By symmetry, it suffices to show that $X_1 = Y_1$. But this readily follows from the definition of $X_1$ and $Y_1$, from the fact that $X_1$ is a clique, anticomplete to $X_3$, and from the fact that $Y_1$ is a clique, anticomplete to $Y_3$. This proves (1). 
\end{proof} 

\begin{quote} 
\emph{(2) Vertices $y_1$ and $y_3$ are complete to $X_4$, and consequently, $X_4 \subseteq Y_4$.} 
\end{quote} 
\begin{proof}[Proof of (2)] 
Since $X_4$ is a clique, the first statement clearly implies the second. Suppose that the first statement is false, and fix $y_4 \in X_4$ such that $y_4$ is nonadjacent to at least one of $y_1$ and $y_3$; by symmetry, we may assume that $y_4$ is nonadjacent to $y_3$, and consequently (since $x_3$ is complete to $X_4$, and $x_4$ is complete to $X_3$), we have that $y_3 \neq x_3$ and $y_4 \neq x_4$. By the choice of $y_3$, it follows that $xx_3 \notin E(G)$. 

Now, suppose that $xy_4 \notin E(G)$. Suppose additionally that $y_1y_4 \notin E(G)$; in particular then, $y_1 \neq x_1$, and by the choice of $y_1$, we see that $xx_1 \notin E(G)$. But then $x_1,y_1,x,y_3,x_3,y_4,x_1$ is a 6-hole in $G$, contrary to the fact that $G$ contains no long holes. Thus, $y_1y_4 \in E(G)$. But then $y_1,x,y_3,x_3,y_4,y_1$ is a 5-hole in $G$, again a contradiction. This proves that $xy_4 \in E(G)$. 

Next, if $y_1y_4 \in E(G)$, then $y_1,y_3,y_4,x_2,x_4,x,x_3,y_1$ is a 7-antihole in $G$, a contradiction. This proves that $y_1y_4 \notin E(G)$. Since $x_1$ is complete to $X_4$, it follows that $y_1 \neq x_1$, and by the choice of $y_1$, it follows that $xx_1 \notin E(G)$. But now $G[x_2,y_4,x,x_1,x_3]$ is a $K_{2,3}$, a contradiction. This proves (2). 
\end{proof} 

\begin{quote} 
\emph{(3) Vertex $x$ is complete to $X_2$.} 
\end{quote} 
\begin{proof}[Proof of (3)] 
Suppose that $x$ has a nonneighbor $y_2 \in X_2$. Since $xx_2 \in E(G)$, we have that $y_2 \neq x_2$. Suppose that $y_2$ is anticomplete to $\{y_1,y_3\}$. Then $y_1 \neq x_1$ and $y_3 \neq x_3$, and so by the choice of $y_1$ and $y_3$, we have that $xx_1,xx_3 \notin E(G)$. But now $y_2,x_1,y_1,x,y_3,x_3,y_2$ is a 6-hole in $G$, contrary to the fact that $G$ contains no long holes. Thus, $y_2$ is adjacent to at least one of $y_1,y_3$; by symmetry, we may assume that $y_1y_2 \in E(G)$. If $y_2y_3 \notin E(G)$, then $y_3 \neq x_3$, and we have that $y_2,y_1,x,y_3,x_3,y_2$ is a 5-hole in $G$, contrary to the fact that $G$ contains no long holes. Thus, $y_2y_3 \in E(G)$. But now $H' = y_2,y_1,x,y_3,y_2$ is a 4-hole and $(H',x_2)$ a universal wheel in $G$, a contradiction. Thus, $x$ is complete to $X_2$. This proves (3). 
\end{proof} 

\begin{quote} 
\emph{(4) $X_2 \subsetneqq Y_2$.}
\end{quote} 
\begin{proof} 
First of all, we know that $x \in Y_2 \setminus X_2$, and so $X_2 \neq Y_2$. It remains to show that $X_2 \subseteq Y_2$. Since $X_2$ is a clique, it suffices to show that $y_1$ and $y_3$ are complete to $X_2$. Suppose otherwise. By symmetry, we may assume that $y_1$ has a nonneighbor $y_2 \in X_2$. Since $x_1$ is complete to $X_2$, it follows that $y_1 \neq x_1$ (by the choice of $y_1$, this implies that $xx_1 \notin E(G)$) and that $y_2 \neq x_2$. By (3), we have that $xx_2,xy_2 \in E(G)$. We now have that $H' = x_1,y_1,x,y_2,x_1$ is a 4-hole and $(H',x_2)$ a universal wheel in $G$, a contradiction. This proves (4). 
\end{proof} 

Statements (1), (2), and (4) imply that $V(H^*) \subsetneqq V(Y^*)$, contrary to the maximality of $|V(H^*)|$. 
\end{proof} 

\begin{lemma} \label{lemma-4-hole-GT} Let $G \in \mathcal{G}_{\text{T}}$, assume that $G$ contains no long holes and no 7-antiholes, and let $H = x_1,x_2,x_3,x_4$ be a 4-hole in $G$, chosen so that $|V(H^*)|$ is maximum. Then either $G = H^*$ (and consequently, $G$ is a 4-ring), or $G$ admits a clique-cutset. 
\end{lemma} 
\begin{proof} 
For each $i \in \mathbb{Z}_4$, set $X_i = X_{x_i}(H)$. By Lemma~\ref{lemma-4-hole-ring-GT}, $H^*$ is a 4-ring with good partition $(X_1,X_2,X_3,X_4)$; in particular, $X_1,X_2,X_3,X_4$ are cliques, $X_1$ is anticomplete to $X_3$, and $X_2$ is anticomplete to $X_4$. If $G = H^*$, then we are done. So assume that $G \neq H^*$, and let $C$ be a component of $G \setminus V(H^*)$. Our goal is to show that $N_G(C) \cap V(H^*)$ is a clique; since $H^*$ is not complete, this will readily imply that $N_G(C) \cap V(H^*)$ is a clique-cutset of $G$, which is what we need. 

Suppose otherwise, that is, suppose that $N_G(C) \cap V(H^*)$ is not a clique. Let $P$ be a minimal connected induced subgraph of $C$ such that $N_G(P) \cap V(H^*)$ is not a clique. Fix $a_1,a_3 \in V(P)$ such that some vertex in $N_G(a_1) \cap V(H^*)$ is nonadjacent to some vertex of $N_G(a_3) \cap V(H^*)$; by Lemma~\ref{lemma-4-hole-vertex-attach-GT}, $a_1 \neq a_3$. Note that $P$ is a path between $a_1$ and $a_3$, for otherwise, any induced path in $P$ between $a_1$ and $a_3$ would contradict the minimality of $P$. Set $P = p_1,\dots,p_n$ so that $p_1 = a_1$ and $p_n = a_3$. 

Now, suppose that for some $i \in \mathbb{Z}_4$, $\{a_1,a_3\}$ is anticomplete to $X_i \cup X_{i+1}$; by symmetry, we may assume that $\{a_1,a_3\}$ is anticomplete to $X_3 \cup X_4$, so that $N_G(\{a_1,a_3\}) \cap V(H^*) \subseteq X_1 \cup X_2$. Since some vertex in $N_G(a_1) \cap V(H^*)$ is nonadjacent to some vertex of $N_G(a_3) \cap V(H^*)$, we may assume by symmetry that there exist nonadjacent vertices $z_1 \in X_1$ and $z_2 \in X_2$ such that $a_1z_1,a_3z_2 \in E(G)$. By Lemma~\ref{lemma-4-hole-vertex-attach-GT}, we know that $a_1z_2,a_3z_1 \notin E(G)$. Next, we claim that all interior vertices of $P$ are anticomplete to $\{z_1,z_2\} \cup X_3 \cup X_4$. Suppose otherwise, and assume that some interior vertex $p$ of $P$ has a neighbor in $\{z_1,z_2\} \cup X_3 \cup X_4$. By symmetry, we may assume that $p$ has a neighbor $z' \in \{z_2\} \cup X_3$. But then $z_1z' \notin E(G)$, and we see that the subpath of $P$ between $a_1$ and $p$ contradicts the minimality of $P$. This proves our claim. But now $z_1,p_1,\dots,p_n,z_2,x_3,x_4,z_1$ is a long hole in $G$, a contradiction. 

By symmetry, we may now assume that $\{a_1,a_3\}$ is anticomplete neither to $X_1$ nor to $X_3$. We know that $X_1$ is anticomplete to $X_3$; by Lemma~\ref{lemma-4-hole-vertex-attach-GT} and by symmetry, we may now assume that $a_1$ has a neighbor $y_1 \in X_1$ and is anticomplete to $X_3$, and that $a_3$ has a neighbor $y_3 \in X_3$ and is anticomplete to $X_1$. Note that $x_2$ has a neighbor in $P$, for otherwise, $y_1,p_1,\dots,p_n,y_3,x_2,y_1$ is a long hole in $G$, a contradiction. Similarly, $x_4$ has a neighbor in $P$. 

Now, we claim that interior vertices of $P$ are anticomplete to $\{x_2,x_4\}$. Suppose otherwise. By symmetry, we may assume that some interior vertex $p$ of $P$ is adjacent to $x_2$. Let $p' \in V(P)$ be such that $p'x_4 \in E(G)$. But now the subpath of $P$ between $p$ and $p'$ contradicts the minimality of $P$. This proves our claim. Since the interior vertices of $P$ are anticomplete to $\{y_1,y_3\}$, we deduce that the interior vertices of $P$ are anticomplete to $\{y_1,x_2,y_3,x_4\}$. It follows that each of $x_2,x_4$ has a neighbor in $\{a_1,a_3\}$. By Lemma~\ref{lemma-4-hole-vertex-attach-GT}, and by symmetry, we may assume that $a_1x_2,a_3x_4 \in E(G)$ and $a_1x_4,a_3x_2 \notin E(G)$. But now $G[\{y_1,x_2,y_3,x_4\} \cup V(P)]$ is a $3PC(y_1x_2a_1,x_4y_3a_3)$, a contradiction. 
\end{proof} 

We remind the reader that $\mathcal{B}_{\text{T}}$ is the class of all complete graphs, rings, and 7-hyperantiholes. We are now ready to prove Theorem~\ref{decomp-thm-GT}, restated below for the reader's convenience. 

\begin{decomp-thm-GT} Every graph in $\mathcal{G}_{\text{T}}$ either belongs to $\mathcal{B}_{\text{T}}$ or admits a clique-cutset. 
\end{decomp-thm-GT}
\begin{proof} 
Fix $G \in \mathcal{G}_{\text{T}}$. If $G$ contains a long hole, then we are done by Lemma~\ref{lemma-long-hole-GT}. So assume that $G$ contains no long holes. If $G$ contains a 7-antihole, then we are done by Lemma~\ref{lemma-7antihole-GT}. So assume that $G$ contains no 7-antiholes. If $G$ contains a 4-hole, then we are done by Lemma~\ref{lemma-4-hole-GT}. So we may assume that $G$ contains no 4-holes. We now have that $G$ contains no holes, and so by definition, $G$ is chordal. But then by Theorem~\ref{thm-Dirac61}, either $G$ is a complete graph, or $G$ admits a clique-cutset, and in either case, we are done. 
\end{proof}

\section{A decomposition theorem for the class $\mathcal{G}_{\text{UT}}^{\text{cap-free}}$}\label{sec:decompGUTcap}

In this section, we prove Theorem~\ref{decomp-thm-GUTcap}, which states that every graph in $\mathcal{G}_{\text{UT}}^{\text{cap-free}}$ either belongs to $\mathcal{B}_{\text{UT}}^{\text{cap-free}}$ or admits a clique-cutset. We remind the reader that the {\em house} is the (unique) cap on five vertices; note that the house is isomorphic to $\overline{P_5}$. Clearly, every cap-free graph is house-free. 

\begin{lemma} \label{lemma-ring-GUTcap-free} Let $G \in \mathcal{G}_{\text{UT}}^{\text{cap-free}}$, and assume that $G$ contains a long hole. Then either some anticomponent of $G$ is a long hyperhole, or $G$ admits a clique-cutset. 
\end{lemma} 
\begin{proof} 
This follows immediately from Lemmas~\ref{lemma-ring-or-clique-cut-GUT} and~\ref{ring-in-GT}(e). 
\end{proof} 

A {\em domino} is a six-vertex graph $D$ with vertex set $V(D) = \{a_1,a_2,a_3,b_1,b_2,b_3\}$ and edge set $E(D) = \{a_1a_2,a_2a_3,b_1b_2,b_2b_3,a_1b_1,a_2b_2,a_3b_3\}$; under these circumstances, we write that ``$D = (a_1,a_2,a_3;b_1,b_2,b_3)$ is a domino.''

\begin{lemma} \label{lemma-domino} Let $G \in \mathcal{G}_{\text{UT}}^{\text{cap-free}}$. Assume that $G$ contains no long holes, but does contain a domino. Then $G$ admits a clique-cutset. 
\end{lemma} 
\begin{proof} 
Let $D = (a_1,a_2,a_3;b_1,b_2,b_3)$ be an induced domino in $G$. Let $S$ be the set of all vertices in $V(G) \setminus V(D)$ that are complete to $\{a_2,b_2\}$. Our goal is to show that $\{a_2,b_2\} \cup S$ is a clique-cutset of $G$. 

\begin{quote} 
\emph{(1) Every vertex in $S$ has a neighbor both in $\{a_1,b_1\}$ and in $\{a_3,b_3\}$.} 
\end{quote} 
\begin{proof}[Proof of (1)] 
Fix $x \in S$ and $i \in \{1,3\}$. If $x$ is anticomplete to $\{a_i,b_i\}$, then $G[a_i,b_i,a_2,b_2,x]$ is a house, contrary to the fact that $G$ is cap-free. This proves (1). 
\end{proof} 

\begin{quote} 
\emph{(2) $\{a_2,b_2\} \cup S$ is a clique.} 
\end{quote} 
\begin{proof}[Proof of (2)] 
Since $a_2b_2 \in E(G)$, and since $S$ is complete to $\{a_2,b_2\}$, it suffices to show that $S$ is a clique. Suppose otherwise, and fix nonadjacent vertices $x,y \in S$. By (1), each of $x,y$ has a neighbor both in $\{a_1,b_1\}$ and in $\{a_3,b_3\}$. Further, $x,y$ have a common neighbor in each of $\{a_1,b_1\}$ and $\{a_3,b_3\}$, for otherwise, it is easy to see that $G[a_1,b_1,a_3,b_3,x,y]$ contains either a 5-hole or a 6-hole, contrary to the fact that $G$ contains no long holes. Now, $\{x,y\}$ is not complete to $\{a_1,a_3\}$, for otherwise, $G[x,y,a_1,b_2,a_3]$ would be a $K_{2,3}$, a contradiction. Similarly, $\{x,y\}$ is not complete to $\{b_1,b_3\}$. By symmetry, we may now assume that $\{x,y\}$ is complete to $\{a_1,b_3\}$, and that $y$ is nonadjacent to $a_3$. But now $G[a_1,a_2,a_3,b_3,y]$ is a house, contrary to the fact that $G$ is house-free. This proves (2). 
\end{proof} 

It remains to show that $\{a_2,b_2\} \cup S$ is a cutset of $G$. Suppose otherwise. Since $\{a_1,b_1\}$ is anticomplete to $\{a_3,b_3\}$, it follows that there exists an induced path $P = p_1,\dots,p_n$ in $G \setminus (V(D) \cup S)$ such that $p_1$ has a neighbor in $\{a_1,b_1\}$, and $p_n$ has a neighbor in $\{a_3,b_3\}$; we may assume that $P$ was chosen so that its length is as small as possible. Note that the minimality of $P$ implies that all interior vertices of $P$ are anticomplete to $\{a_1,b_1,a_3,b_3\}$.   

\begin{quote} 
\emph{(3) At most one of $a_2,b_2$ has a neighbor in $V(P)$.} 
\end{quote} 
\begin{proof}[Proof of (3)] 
Suppose otherwise. Fix $i,j \in \{1,\dots,n\}$, such that $p_ia_2,p_jb_2 \in E(G)$, and subject to that, such that $|i-j|$ is minimum. By symmetry, we may assume that $i \leq j$. If $i = j$, then $p_i = p_j$ belongs to $S$, a contradiction; thus, $i < j$. If $i+1 < j$, then $p_i,p_{i+1},\dots,p_j,b_2,a_2,p_i$ is a long hole in $G$, a contradiction; thus, $j = i+1$. 

Next, we claim that $b_2$ is anticomplete to $\{p_1,\dots,p_{i-1}\}$. Suppose otherwise, and fix a maximum $\ell \in \{1,\dots,i-1\}$ such that $b_2p_{\ell} \in E(G)$. Then $b_2,p_{\ell},\dots,p_i,p_{i+1},b_2$ is a hole in $G$; since $G$ contains no long holes, it follows that $\ell = i-1$. Since $p_{\ell} \notin S$, we see that $p_{\ell}a_2 \notin E(G)$, and it follows that $G[p_i,b_2,p_{i-1},a_2,p_{i+1}]$ is a $K_{2,3}$, a contradiction. This proves our claim. 

Now, if $p_1b_1 \notin E(G)$, then $p_1a_1 \in E(G)$, and $p_1,\dots,p_{i+1},b_2,b_1,a_1,p_1$ is a long hole in $G$, a contradiction. Thus, $p_1b_1 \in E(G)$, and we have that $H = p_1,\dots,p_{i+1},b_2,b_1,p_1$ is a hole in $G$. Since $G$ contains no long holes, it follows that $H$ is a 4-hole, and consequently, $i = 1$. But now $G[p_1,b_2,b_1,a_2,p_2]$ is a $K_{2,3}$, a contradiction. This proves (3). 
\end{proof} 

By (3), and by symmetry, we may assume that $a_2$ is anticomplete to $V(P)$. Then $n = 1$ (i.e.\ $P$ is a trivial path), and $p_1$ is complete to $\{a_1,a_3\}$, for otherwise, we readily deduce that $G[V(P) \cup (V(D) \setminus \{b_2\})]$ contains a long hole, a contradiction. Now $p_1,a_1,a_2,a_3,p_1$ is a 4-hole, and since $G$ is house-free, we deduce that $p_1$ is anticomplete to $\{b_1,b_3\}$. Then $p_1b_2 \in E(G)$, for otherwise, $p_1,a_1,b_1,b_2,b_3,a_3,p_1$ is a 6-hole in $G$, a contradiction. But now $G[p_1,a_2,a_1,b_2,a_3]$ is a $K_{2,3}$, a contradiction. 
\end{proof} 

\begin{lemma} \label{lemma-4-hole-GUTcap} Let $G \in \mathcal{G}_{\text{UT}}^{\text{cap-free}}$. Assume that $G$ contains a 4-hole, contains no long holes, and does not admit a clique-cutset. Then $G$ has at least two nontrivial anticomponents. 
\end{lemma} 
\begin{proof} 
Let $H = x_1,x_2,x_3,x_4,x_1$ be a 4-hole in $G$, and for all $i \in \mathbb{Z}_4$, set $X_i = X_{x_i}(H)$. Thus, $(X_1,X_2,X_3,X_4)$ is a partition of $V(H^*)$. 

\begin{quote}
\emph{(1) For all $i \in \mathbb{Z}_4$, $X_i$ is a clique, complete to $X_{i-1} \cup X_{i+1}$.} 
\end{quote} 
\begin{proof}[Proof of (1)] 
By Lemma~\ref{lemma-hole-ring-GUT}, $X_1,X_2,X_3,X_4$ are cliques. By symmetry, it now suffices to show that $X_1$ is complete to $X_2$. Suppose otherwise, and fix nonadjacent vertices $y_1 \in X_1$ and $y_2 \in X_2$; since $x_1$ is complete to $X_2$, we have that $y_1 \neq x_1$. But now $G[y_1,x_1,y_2,x_3,x_4]$ is a house, a contradiction. This proves (1). 
\end{proof} 

\begin{quote} 
\emph{(2) For all $x \in V(G) \setminus (V(H^*) \cup U_H)$, there exists some $i \in \mathbb{Z}_4$ such that $N_G(x) \cap V(H^*) \subseteq X_i \cup X_{i+2}$.} 
\end{quote} 
\begin{proof}[Proof of (2)] 
Suppose otherwise. By symmetry, we may assume that there exist some $x \in V(G) \setminus (V(H^*) \cup U_H)$, $y_1 \in X_1$, and $y_2 \in X_2$ such that $xy_1,xy_2 \in E(G)$. By (1), $y_1y_2 \in E(G)$. 

Since $x \notin V(H^*) \cup U_H$, we know that $x$ has at most two neighbors in $V(H)$. But if $x$ has precisely two neighbors in $V(H)$, then $G[V(H) \cup \{x\}]$ is either a house or a $K_{2,3}$, a contradiction in either case. Thus, $x$ has at most one neighbor in $V(H)$. 

If $x$ is anticomplete to $\{x_3,x_4\}$, then $G[x,y_1,y_2,x_3,x_4]$ is a house, a contradiction. By symmetry, we may now assume that $xx_3 \in E(G)$. Since $x$ has at most one neighbor in $V(H)$, we deduce that $x_3$ is the unique neighbor of $x$ in $V(H)$. But now $G[x_1,y_2,x_3,x_4,x]$ is a house, a contradiction. This proves (2). 
\end{proof} 

\begin{quote} 
\emph{(3) For all components $C$ of $G \setminus (V(H^*) \cup U_H)$, there exists some $i \in \mathbb{Z}_4$ such that $N_G(C) \cap V(H^*) \subseteq X_i \cup X_{i+2}$.} 
\end{quote} 
\begin{proof}[Proof of (3)] 
Suppose otherwise, and let $C$ be a component of $G \setminus (V(H^*) \cup U_H)$ that contradicts (3). Then for some $i \in \mathbb{Z}_4$, $N_G(C)$ intersects both $X_i$ and $X_{i+1}$. Let $P$ be a minimal connected induced subgraph of $C$ such that there exists some $i \in \mathbb{Z}_4$ such that $N_G(P)$ intersects both $X_i$ and $X_{i+1}$; by symmetry, we may assume that $N_G(P)$ intersects both $X_1$ and $X_2$. Let $a_1,a_2 \in V(P)$ be such that $a_1$ has a neighbor $y_1 \in X_1$ and $a_2$ has a neighbor $y_2 \in X_2$. By (1), $y_1y_2 \in E(G)$. By (2), $N_G(a_1) \cap V(H^*) \subseteq X_1 \cup X_3$ and $N_G(a_2) \cap V(H^*) \subseteq X_2 \cup X_4$; in particular, $a_1 \neq a_2$. Clearly, $P$ is a path between $a_1$ and $a_2$, for otherwise, any induced path in $P$ between $a_1$ and $a_2$ would contradict the minimality of $P$. Furthermore, by the minimality of $P$, all interior vertices of $P$ are anticomplete to $X_1 \cup X_2$. Further, $P$ is of length one, for otherwise, $G[V(P) \cup \{y_1,y_2\}]$ would be a long hole in $G$, a contradiction; in particular, $a_1a_2 \in E(G)$. 

Next, we have that $a_1x_3 \notin E(G)$, for otherwise, $G[y_1,x_3,a_1,y_2,x_4]$ would be a $K_{2,3}$, a contradiction. Similarly, $a_2x_4 \notin E(G)$. But now $G[a_1,y_1,x_4,a_2,y_2,x_3]$ is a domino, and so by Lemma~\ref{lemma-domino}, $G$ admits a clique-cutset, a contradiction. This proves (3). 
\end{proof} 

\begin{quote} 
\emph{(4) For all components $C$ of $G \setminus (V(H^*) \cup U_H)$, $N_G(C) \cap (V(H^*) \cup U_H)$ is a clique.} 
\end{quote} 
\begin{proof}[Proof of (4)]  
Suppose otherwise, and let $C$ be a component of $G \setminus (V(H^*) \cup U_H)$ such that $N_G(C) \cap (V(H^*) \cup U_H)$ is not a clique. Let $P$ be a minimal connected induced subgraph of $C$ such that $N_G(P) \cap (V(H^*) \cup U_H)$ is not a clique. By (3), and by symmetry, we may assume that $N_G(P) \cap V(H^*) \subseteq X_1 \cup X_3$. Now, fix nonadjacent $y,y' \in N_G(P) \cap (V(H^*) \cup U_H)$, and fix (not necessarily distinct) $a,a' \in V(P)$ such that $ay,a'y' \in E(G)$. Note that $P$ is a path between $a$ and $a'$ (if $a = a'$, then we simply have that $P$ is a one-vertex path), for otherwise, any induced path in $P$ between $a$ and $a'$ would contradict the minimality of $P$. Set $P = p_1,\dots,p_n$, with $p_1 = a$ and $p_n = a'$; by the minimality of $P$, we have that $P' = y,p_1,\dots,p_n,y'$ is an induced path in $G$. Now, since $N_G(P) \cap V(H^*) \subseteq X_1 \cup X_3$, we see that $x_2$ and $x_4$ are anticomplete to $V(P)$. Since $\{x_2,x_4\}$ is complete to $X_1 \cup X_3 \cup U_H$, we deduce that $\{x_2,x_4\}$ is complete to $\{y,y'\}$. But now $G[V(P) \cup \{y,y',x_2,x_4\}]$ is a $3PC(y,y')$, a contradiction. This proves (4). 
\end{proof} 

\begin{quote} 
\emph{(5) $V(G) = V(H^*) \cup U_H$.} 
\end{quote} 
\begin{proof}[Proof of (5)]  
Suppose otherwise, and let $C$ be a component of $G \setminus (V(H^*) \cup U_H)$. It then follows from (4) that $N_G(C) \cap (V(H^*) \cup U_H)$ is a clique-cutset of $G$, a contradiction. This proves (5). 
\end{proof} 

\begin{quote} 
\emph{(6) Every vertex in $U_H$ is complete to at least three of the sets $X_1,X_2,X_3,X_4$.}
\end{quote} 
\begin{proof}[Proof of (6)] 
Let $x \in U_H$. By symmetry, it suffices to show that if $x$ has a nonneighbor in $X_1$, then $x$ is complete to $X_2 \cup X_3 \cup X_4$. So suppose that $x$ is nonadjacent to some $y_1 \in X_1$. Suppose that $x$ has a nonneighbor $y_2 \in X_2$. By (1), $y_1y_2 \in E(G)$, and we deduce that $G[x,y_1,y_2,x_3,x_4]$ is a house, a contradiction. Thus, $x$ is complete to $X_2$, and similarly, $x$ is complete to $X_4$. Suppose that $x$ has a nonneighbor $y_3 \in X_3$. If $y_1y_3 \notin E(G)$, then $G[x_2,x_4,y_1,y_3,x]$ is a $K_{2,3}$, a contradiction. Thus, $y_1y_3 \in E(G)$. Since $x \in U_H$, we know that $xx_1,xx_3 \in E(G)$; since $xy_1,xy_3 \notin E(G)$, it follows that $y_1 \neq x_1$ and $y_3 \neq x_3$. But now, by (1), $x,x_1,y_1,y_3,x_3,x$ is a 5-hole in $G$, a contradiction. This proves (6). 
\end{proof} 

\begin{quote} 
\emph{(7) Every nontrivial anticomponent of $G[U_H]$ is complete to $V(H^*)$.} 
\end{quote} 
\begin{proof}[Proof of (7)]  
Suppose otherwise, and let $C$ be the vertex set of a nontrivial anticomponent of $G[U_H]$ such that $C$ is not complete to $V(H^*)$. Fix $x \in C$ such that $x$ has a nonneighbor in $V(H^*)$, and let $y \in C$ be a nonneighbor of $x$ ($y$ exists because $G[C]$ is anticonnected and has at least two vertices). By symmetry, we may assume that $x$ has a nonneighbor $y_1 \in X_1$ (clearly, $y_1 \neq x_1$). But now if $yy_1 \notin E(G)$, then $G[x_2,x_4,x,y,y_1]$ is a $K_{2,3}$, a contradiction, and if $yy_1 \in E(G)$, then $G[x,y,y_1,x_1,x_3]$ is a house, again a contradiction. This proves (7). 
\end{proof} 

Suppose first that $U_H$ is not a clique, and let $C$ be the vertex set of a nontrivial anticomponent of $G[U_H]$. By (5) and (7), $C$ is the vertex set of a nontrivial anticomponent of $G$. Since $C \cap V(H) = \emptyset$, we see that some other anticomponent of $G$ (for example, the one containing $x_1$ and $x_3$) is also nontrivial, and it follows that $G$ contains at least two nontrivial anticomponents, which is what we needed to show. 

From now on, we assume that $U_H$ is a clique. Let $Y$ be the set of all vertices in $U_H$ that are complete to $V(H^*)$, and for all $i \in \mathbb{Z}_4$, let $Y_i$ be the set of all vertices in $U_H$ that have a nonneighbor in $X_i$. Clearly, $U_H = Y \cup Y_1 \cup Y_2 \cup Y_3 \cup Y_4$. By (6), we have that $Y_i$ is complete to $X_{i+1} \cup X_{i+2} \cup X_{i+3}$ for all $i \in \mathbb{Z}_4$, and we deduce that $Y,Y_1,Y_2,Y_3,Y_4$ are pairwise disjoint. By (5), we now have that $V(G) = (X_1 \cup X_3 \cup Y_1 \cup Y_3) \cup (X_2 \cup X_4 \cup Y_2 \cup Y_4) \cup Y$. By (1), $X_1 \cup X_3$ is complete to $X_2 \cup X_4$, and we now deduce that the sets $X_1 \cup X_3 \cup Y_1 \cup Y_3$, $X_2 \cup X_4 \cup Y_2 \cup Y_4$, and $Y$ are pairwise complete to each other. Since $x_1x_3 \notin E(G)$, we know that $G[X_1 \cup X_3 \cup Y_1 \cup Y_3]$ contains at least one nontrivial anticomponent, and since $x_2x_4 \notin E(G)$, $G[X_2 \cup X_4 \cup Y_2 \cup Y_4]$ contains at least one nontrivial anticomponent. It follows that $G$ contains at least two nontrivial anticomponents, and we are done. 
\end{proof} 

We remind the reader that $\mathcal{B}_{\text{UT}}^{\text{cap-free}}$ is the class of all graphs $G$ that satisfy one of the following: 
\begin{itemize} 
\item $G$ has exactly one nontrivial anticomponent, and this anticomponent is a hyperhole of length at least six; 
\item each anticomponent of $G$ is either a 5-hyperhole or a chordal cobipartite graph. 
\end{itemize} 
We are now ready to prove Theorem~\ref{decomp-thm-GUTcap}, restated below for the reader's convenience. 

\begin{decomp-thm-GUTcap} Every graph in $\mathcal{G}_{\text{UT}}^{\text{cap-free}}$ either belongs to $\mathcal{B}_{\text{UT}}^{\text{cap-free}}$ or admits a clique-cutset. 
\end{decomp-thm-GUTcap}
\begin{proof} 
Let $G \in \mathcal{G}_{\text{UT}}^{\text{cap-free}}$, and assume that $G$ does not admit a clique-cutset; we must show that $G \in \mathcal{B}_{\text{UT}}^{\text{cap-free}}$.

\begin{quote} 
\emph{(1) Every anticomponent of $G$ is either a long hyperhole or a chordal cobipartite graph.} 
\end{quote} 
\begin{proof}[Proof of (1)] 
Let $H$ be an anticomponent of $G$. We must show that $H$ is either a long hyperhole or a chordal cobipartite graph. 

Suppose first that $H$ admits a clique-cutset $C$. Clearly, $\alpha(H) \geq 2$. If $U_H$ is a (possibly empty) clique, then $C \cup U_H$ is a clique-cutset of $G$, a contradiction. Thus, $U_H$ is not a clique, and we deduce that $G$ has at least two nontrivial anticomponents. Lemma~\ref{rmrk-K23-anticomp} now implies that $\alpha(G) = 2$; since $\alpha(H) \geq 2$, it follows that $\alpha(H) = 2$, and so by Lemma~\ref{lemma-clique-cut-alpha-2}, $H$ is a chordal cobipartite graph, and we are done. From now on, we assume that $H$ does not admit a clique-cutset. 

Suppose that $H$ contains a long hole. Then by Lemma~\ref{lemma-ring-GUTcap-free}, $H$ is a long hyperhole, and we are done. So from now on, we assume that $H$ contains no long holes. Since $H$ is anticonnected, Lemma~\ref{lemma-4-hole-GUTcap} implies that $H$ contains no 4-holes. Thus, $H$ contains no holes, and so by definition, $H$ is chordal. Since $H$ does not admit a clique-cutset, Theorem~\ref{thm-Dirac61} implies that $H$ is a complete graph (in fact, since $H$ is anticonnected, $H$ is isomorphic to $K_1$), and in particular, $H$ is a chordal cobipartite graph. This proves (1). 
\end{proof} 

If $G$ contains at most one nontrivial anticomponent, then (1) implies that $G \in \mathcal{B}_{\text{UT}}^{\text{cap-free}}$, and we are done. So assume that $G$ has at least two nontrivial anticomponents; by Lemma~\ref{lemma-clique-cut-alpha-2}, it follows that $\alpha(G) = 2$. Since every hyperhole of length greater than five contains a stable set of size three, (1) now implies that every anticomponent of $G$ is either a 5-hyperhole or a chordal cobipartite graph, and it follows that $G \in \mathcal{B}_{\text{UT}}^{\text{cap-free}}$. 
\end{proof}

\section{$\chi$-Boundedness} \label{sec:chi}

In this section, we obtain polynomial $\chi$-bounding functions for the classes $\mathcal{G}_{\text{UT}},\mathcal{G}_{\text{U}},\mathcal{G}_{\text{T}},\mathcal{G}_{\text{UT}}^{\text{cap-free}}$. 

In subsection~\ref{subsec6.1}, we deal with classes $\mathcal{G}_{\text{U}},\mathcal{G}_{\text{T}},\mathcal{G}_{\text{UT}}^{\text{cap-free}}$. For each of the three classes, we obtain a linear $\chi$-bounding function; the proofs rely on our decomposition theorems for these classes (i.e.\ Theorems~\ref{decomp-thm-GU},~\ref{decomp-thm-GT}, and~\ref{decomp-thm-GUTcap}), as well as on results from~\cite{WeaklyTriangulatedPerfect, HyperholeColoring}. 

In subsection~\ref{subsec6.2}, we obtain a fourth-degree polynomial $\chi$-bounding function for the class $\mathcal{G}_{\text{UT}}$. Instead of relying on Theorem~\ref{decomp-thm-GUT} (the decomposition theorem for $\mathcal{G}_{\text{UT}}$ that we stated in the introduction and proved in section~\ref{sec:decompGUT}), we prove a new decomposition theorem for the class $\mathcal{G}_{\text{UT}}$, one that ``decomposes'' graphs in $\mathcal{G}_{\text{UT}}$ into ``basic'' cap-free induced subgraphs via ``double-star-cutsets'' that are ``small'' relative to the clique number of the graph. We then rely on Theorem~\ref{chi-bound-GUTcap} (which states that the class $\mathcal{G}_{\text{UT}}^{\text{cap-free}}$ is $\chi$-bounded by a linear function), as well as a result of~\cite{small-cut}, to obtain a polynomial $\chi$-bounding function for the class $\mathcal{G}_{\text{UT}}$.

\subsection{Classes $\mathcal{G}_{\text{U}},\mathcal{G}_{\text{T}},\mathcal{G}_{\text{UT}}^{\text{cap-free}}$} \label{subsec6.1} 

We begin with an easy lemma, which essentially states that clique-cutsets ``preserve $\chi$-boundedness'' (by the same $\chi$-bounding function). 

\begin{lemma} \label{lemma-chi-clique-cut} Let $\mathcal{G}$ be a hereditary class, and let $f:\mathbb{N}^+ \rightarrow \mathbb{N}^+$ be a nondecreasing function. Assume that every graph $G \in \mathcal{G}$ either satisfies $\chi(G) \leq f(\omega(G))$ or admits a clique-cutset. Then every graph $G \in \mathcal{G}$ satisfies $\chi(G) \leq f(\omega(G))$. 
\end{lemma} 
\begin{proof} 
Clearly, if $(A,B,C)$ is a clique-cut-partition of a graph $G$, then $\chi(G) = \max\{\chi(G[A \cup C]),\chi(G[B \cup C])\}$. The result now follows by an easy induction. 
\end{proof} 

A function $f:\mathbb{N}^+ \rightarrow \mathbb{N}^+$ is {\em superadditive} if for all $m,n \in \mathbb{N}^+$, we have that $f(m)+f(n) \leq f(m+n)$. Note that every superadditive function is nondecreasing. 

\begin{lemma} \label{lemma-chi-anticomp} Let $f:\mathbb{N}^+ \rightarrow \mathbb{N}^+$ be a superadditive function, let $G$ be a graph, and assume that all anticomponents $H$ of $G$ satisfy $\chi(H) \leq f(\omega(H))$. Then $\chi(G) \leq f(\omega(G))$. 
\end{lemma} 
\begin{proof} 
Let $G_1,\dots,G_t$ be the anticomponents of $G$. Clearly, $\omega(G) = \sum_{i=1}^t \omega(G_i)$ and $\chi(G) = \sum_{i=1}^t \chi(G_i)$. By hypothesis, $\chi(G_i) \leq f(\omega(G_i))$ for all $i \in \{1,\dots,t\}$. Since $f$ is superadditive, it follows that $\chi(G) = \sum_{i=1}^t \chi(G_i) \leq \sum_{i=1}^t f(\omega(G_i)) \leq f(\sum_{i=1}^t \omega(G_i)) = f(\omega(G))$. 
\end{proof}

\begin{lemma} \label{lemma-ring-chi} Every ring $R$ satisfies $\chi(R) \leq \lfloor \frac{3}{2}\omega(R) \rfloor$. In particular, every hyperhole $H$ satisfies $\chi(H) \leq \lfloor \frac{3}{2}\omega(H) \rfloor$. 
\end{lemma} 
\begin{proof} 
Since every hyperhole is a ring, the second statement follows from the first. To prove the first statement, we let $R$ be a ring, and we assume inductively that every ring $R'$ on fewer than $|V(R)|$ vertices satisfies $\chi(R') \leq \lfloor \frac{3}{2}\omega(R') \rfloor$. We must show that $\chi(R) \leq \lfloor \frac{3}{2} \omega(R) \rfloor$. 

Let $(X_1,\dots,X_k)$, with $k \geq 4$, be a good partition of the ring $R$. By symmetry, we may assume that $|X_2| = \max\{|X_i| \mid i \in \mathbb{Z}_k\}$. If $|X_2| = 1$, then $R$ is a hole, we deduce that $\omega(R) = 2$ and $\chi(R) \leq 3$, and the result follows. So from now on, we assume that $|X_2| \geq 2$. Further, it readily follows from Lemma~\ref{lemma-ring-char}(b) that $\omega(R) \leq \max\{|X_i|+|X_{i+1}| \mid i \in \mathbb{Z}_k\}$, and so the maximality of $|X_2|$ implies that $|X_2| \geq \frac{1}{2}\omega(R)$. 

Let $x_2 \in X_2$ be such that for all $x_2' \in X_2$, $N_R[x_2] \subseteq N_R[x_2']$ (the existence of the vertex $x_2$ follows from the definition of a ring). Set $Y_1 = N_R(x_2) \cap X_1$ and $Y_3 = N_R(x_2) \cap X_3$; then $N_R(x_2) = Y_1 \cup (X_2 \setminus \{x_2\}) \cup Y_3$, and it follows that $d_R(x_2) = |Y_1|+|X_2|+|Y_3|-1$. Now, the choice of $x_2$ guarantees that $X_2$ is complete to $Y_1 \cup Y_3$, which in turn implies that $\max\{|Y_1|+|X_2|,|Y_3|+|X_2|\} \leq \omega(R)$. It follows that $|Y_1|+|X_2|+|Y_3| \leq 2\omega(R)-|X_2|$, and so 
\begin{displaymath} 
\begin{array}{ccccccc}
d_R(x_2) & = & |Y_1|+|X_2|+|Y_3|-1 & \leq & 2\omega(R)-|X_2|-1 & \leq & \frac{3}{2}\omega(R)-1.
\end{array} 
\end{displaymath} 
Since $d_R(x_2)$ is an integer, it follows that $d_R(x_2) \leq \lfloor \frac{3}{2}\omega(R) \rfloor-1$. 

Now, since $|X_2| \geq 2$, the choice of $x_2$ guarantees that $R \setminus x_2$ is a ring. By the induction hypothesis, we have that $\chi(R \setminus x_2) \leq \lfloor \frac{3}{2}\omega(R \setminus x_2) \rfloor \leq \lfloor \frac{3}{2}\omega(R) \rfloor$. Since $d_R(x_2) \leq \lfloor \frac{3}{2}\omega(R) \rfloor-1$, it follows that $\chi(R) \leq \lfloor \frac{3}{2} \omega(R) \rfloor$. 
\end{proof} 

A graph is {\em weakly chordal} (or {\em weakly triangulated}) if it contains no long holes and no long antiholes. It was shown in~\cite{WeaklyTriangulatedPerfect} that weakly chordal graphs are perfect (note that this can also be deduced from the Strong Perfect Graph Theorem~\cite{SPGT}). 

\begin{lemma} \label{lemma-7hyperantihole-chi} Every 7-hyperantihole $G$ satisfies $\chi(G) \leq \lfloor \frac{4}{3} \omega(G) \rfloor$. 
\end{lemma} 
\begin{proof} 
Let $G = X_1,X_2,\dots,X_7,X_1$ be a 7-hyperantihole. By symmetry, we may assume that $|X_7| = \min \{|X_i| \mid i \in \mathbb{Z}_7\}$. Since $X_7 \cup X_2 \cup X_4$ is a clique, the minimality of $|X_7|$ implies that $|X_7| \leq \frac{1}{3}\omega(G)$. Now, note that $G \setminus X_7$ is weakly chordal, and therefore (by~\cite{WeaklyTriangulatedPerfect}) perfect. Thus, $\chi(G \setminus X_7) = \omega(G \setminus X_7) \leq \omega(G)$. Since $|X_7| \leq \frac{1}{3}\omega(G)$, it follows that $\chi(G) \leq \frac{4}{3}\omega(G)$. The result now follows from the fact that $\chi(G)$ is an integer. 
\end{proof} 

We are now ready to show that each of the classes $\mathcal{G}_{\text{U}},\mathcal{G}_{\text{T}},\mathcal{G}_{\text{UT}}^{\text{cap-free}}$ is $\chi$-bounded by a linear function. 

\begin{theorem} \label{chi-bound-GU} Every graph in $\mathcal{G}_{\text{U}}$ satisfies $\chi(G) \leq \omega(G)+1$. 
\end{theorem} 
\begin{proof} 
In view of Theorem~\ref{decomp-thm-GU} and Lemma~\ref{lemma-chi-clique-cut}, it suffices to show that every graph $G \in \mathcal{B}_{\text{U}}$ satisfies $\chi(G) \leq \omega(G)+1$. But this readily follows from the definition of $\mathcal{B}_{\text{U}}$. 
\end{proof} 

\begin{theorem} \label{chi-bound-GT} Every graph in $\mathcal{G}_{\text{T}}$ satisfies $\chi(G) \leq \lfloor \frac{3}{2}\omega(G) \rfloor$. 
\end{theorem} 
\begin{proof} 
In view of Theorem~\ref{decomp-thm-GT} and Lemma~\ref{lemma-chi-clique-cut}, it suffices to show that every graph $G \in \mathcal{B}_{\text{T}}$ satisfies $\chi(G) \leq \lfloor \frac{3}{2} \omega(G) \rfloor$. But this easily follows from the definition of $\mathcal{B}_{\text{T}}$, and from Lemmas~\ref{lemma-ring-chi} and~\ref{lemma-7hyperantihole-chi}. 
\end{proof} 

\begin{theorem} \label{chi-bound-GUTcap} Every graph in $\mathcal{G}_{\text{UT}}^{\text{cap-free}}$ satisfies $\chi(G) \leq \lfloor \frac{3}{2} \omega(G) \rfloor$. 
\end{theorem} 
\begin{proof}
In view of Theorem~\ref{decomp-thm-GUTcap} and Lemma~\ref{lemma-chi-clique-cut}, it suffices to show that every graph $G \in \mathcal{B}_{\text{UT}}^{\text{cap-free}}$ satisfies $\chi(G) \leq \lfloor \frac{3}{2} \omega(G) \rfloor$. But this easily follows from the definition of $\mathcal{B}_{\text{UT}}^{\text{cap-free}}$, from Lemmas~\ref{lemma-chi-anticomp} and~\ref{lemma-ring-chi}, and from the fact that (by Theorem~\ref{thm-Dirac61}) chordal graphs are perfect. 
\end{proof}

\subsection{Class $\mathcal{G}_{\text{UT}}$} \label{subsec6.2} 

It was proven in~\cite{alon87} that every graph of ``large'' chromatic number contains a ``highly connected'' induced subgraph of ``large'' chromatic number. The bound from~\cite{alon87} was subsequently improved in~\cite{substitution}, and it was further improved in~\cite{small-cut}. We state the result from~\cite{small-cut} below. 

\begin{theorem}\cite{small-cut} \label{thm-small-cut} Let $k$ be a positive and $c$ a nonnegative integer, and let $G$ be a graph such that $\chi(G)>\max\{c+2k-2,2k^2\}$. Then G contains a $(k+1)$-connected induced subgraph of chromatic number greater than $c$. 
\end{theorem} 

Our next result is an easy corollary of Theorem~\ref{thm-small-cut}. 

\begin{theorem} \label{cor-small-cut}
Let $\mathcal{G}$ and $\mathcal{G}^*$ be hereditary classes, and let $f,h:\mathbb{N}^+ \rightarrow \mathbb{N}^+$ be nondecreasing functions. Assume that $\mathcal{G}$ is $\chi$-bounded by $f$, and assume that every graph $G \in \mathcal{G}^*$ either belongs to $\mathcal{G}$ or admits a cutset of size at most $h(\omega(G))$. Then $\mathcal{G}^*$ is $\chi$-bounded by the function $g:\mathbb{N}^+ \rightarrow \mathbb{N}^+$ given by $g(n) = \max\{f(n)+2h(n)-2,2h(n)^2\}$ for all $n \in \mathbb{N}^+$. 
\end{theorem} 
\begin{proof} 
Suppose otherwise. Fix $G \in \mathcal{G}^*$ such that $\chi(G) > g(\omega(G))$. Set $k = h(\omega(G))$ and $c = f(\omega(G))$; then $\chi(G) > \max\{c+2k-2,2k^2\}$, and so by Theorem~\ref{thm-small-cut}, $G$ contains a $(k+1)$-connected induced subgraph $H$ such that $\chi(H) > c$. Since $\mathcal{G}^*$ is hereditary, we know that $H \in \mathcal{G}^*$. Since $f$ is nondecreasing, we have that $\chi(H) > c = f(\omega(G)) \geq f(\omega(H))$; since $\mathcal{G}$ is $\chi$-bounded by $f$, it follows that $H \notin \mathcal{G}$. Since $H \in \mathcal{G}^*$ and $H \notin \mathcal{G}$, it follows that $H$ has a cutset of size at most $h(\omega(H))$. But since $h$ is nondecreasing, we have that $h(\omega(H)) \leq h(\omega(G)) = k$, and so $H$ has a cutset of size at most $k$, contrary to the fact that $H$ is $(k+1)$-connected. This proves that $\mathcal{G}^*$ is $\chi$-bounded by $g$. 
\end{proof} 

Given $k,\ell\in\mathbb{N}^+$, the \emph{Ramsey number} $R(k,\ell)$ is the smallest integer such that all graphs on $R(k,\ell)$ vertices contain a clique of size $k$ or a stable set of size $\ell$ (see, for instance, chapter 8.3 of \cite{West}). A {\em double-star-cutset} of a graph $G$ is a cutset $S$ of $G$ such that there exist two adjacent vertices $u,v \in S$ (called the {\em centers} of the double-star-cutset $S$) such that $S \subseteq N_G[u] \cup N_G[v]$.

\begin{theorem} \label{GUT-cap-double-star-cut} Every graph $G \in \mathcal{G}_{\text{UT}}$ satisfies at least one of the following: 
\begin{itemize} 
\item $G$ is cap-free (and so $G \in \mathcal{G}_{\text{UT}}^{\text{cap-free}}$); 
\item $\omega(G) \geq 3$, and $G$ admits a double-star-cutset of size at most $R(\omega(G)-1,3)+4\omega(G)-7$. 
\end{itemize} 
\end{theorem} 
\begin{proof} 
Fix $G \in \mathcal{G}_{\text{UT}}$. We may assume that $G$ is not cap-free, for otherwise, we are done. Since every cap contains a triangle, this implies that $\omega(G) \geq 3$. Since $G$ contains a cap, we know that there exists a hole $H = x,y,x_1,\dots,x_h,x$ (with $h \geq 2$) in $G$, and a vertex $c \in V(G) \setminus V(H)$ such that $N_G(c) \cap V(H) = \{x,y\}$. (Thus, $G[V(H) \cup \{c\}]$ is a cap.) For all $v \in V(H)$, let $T_v$ be the set of all twins of $v$ with respect to $H$, that is, let $T_v = X_v(H) \setminus \{v\}$. Set $S = \{x,y\} \cup T_x \cup T_y \cup T_{x_1} \cup T_{x_h} \cup U_H$. Our goal is to show that $S$ is a double-star-cutset (with centers $x,y$), with $|S| \leq R(\omega(G)-1,3)+4\omega(G)-7$. 

\begin{quote} 
\emph{(1) $|S| \leq R(\omega(G)-1,3)+4\omega(G)-7$.} 
\end{quote} 
\begin{proof}[Proof of (1)] 
By Lemma~\ref{lemma-hole-ring-GUT}, $T_v$ is a clique for all $v \in V(H)$. Furthermore, for all $v \in V(H)$, both $v$ and its two neighbors in $H$ are complete to $T_v$; consequently, $|T_v| \leq \omega(G)-2$ for all $v \in V(H)$, and it follows that $|S \setminus U_H| = |\{x,y\} \cup T_x \cup T_y \cup T_{x_1} \cup T_{x_h}| \leq 4\omega(G)-6$. 

It remains to show that $|U_H| \leq R(\omega(G)-1,3)-1$. Since $U_H$ is complete to the clique $\{x,y\}$, we know that $\omega(G[U_H]) \leq \omega(G)-2$. Next, note that $x$ and $x_1$ are nonadjacent and complete to $U_H$, and so Lemma~\ref{rmrk-K23-anticomp} applied to $G[\{x,x_1\} \cup U_H]$ implies that $\alpha(G[U_H]) \leq 2$. Thus, $G[U_H]$ contains neither a clique of size $\omega(G)-1$ nor a stable set of size three, and so $|U_H| \leq R(\omega(G)-1,3)-1$. This proves (1). 
\end{proof} 

It remains to prove that $S$ is a double-star-cutset with centers $x$ and $y$. First of all, it is clear that $x$ and $y$ are adjacent, that $x$ is complete to $T_x \cup T_{x_h}$, that $y$ is complete to $T_y \cup T_{x_1}$, and that $\{x,y\}$ is complete to $U_H$. Thus, it suffices to show that $S$ is a cutset of $G$ separating $c$ from $\{x_1,\dots,x_h\}$. Clearly, we may now assume that $T_x \cup T_y \cup T_{x_1} \cup T_{x_h} \cup U_H = \emptyset$, and we just need to show that $\{x,y\}$ is a cutset of $G$ that separates $c$ from $\{x_1,\dots,x_h\}$. 

Suppose otherwise, that is, suppose that $\{x,y\}$ does not separate $c$ from $\{x_1,\dots,x_h\}$ in $G$. Fix a minimum-length induced path $P$ in $G \setminus \{x,y\}$ such that one endpoint of $P$ is $c$, and the other endpoint of $P$ belongs to $\{x_1,\dots,x_h\}$. Since $c$ is anticomplete to $\{x_1,\dots,x_h\}$, we know that $P$ is of length at least two. Set $P = p_0,p_1,\dots,p_{n+1}$ (with $n \geq 1$), so that $c = p_0$ and $p_{n+1} \in \{x_1,\dots,x_h\}$. By the minimality of $P$, we know that $c$ is anticomplete to $\{p_2,\dots,p_{n+1}\}$, and that vertices $p_0,\dots,p_{n-1}$ are anticomplete to $\{x_1,\dots,x_h\}$. 

\begin{quote} 
\emph{(2) $N_G(p_n) \cap V(H)$ is a clique of size at most two.} 
\end{quote} 
\begin{proof}[Proof of (2)] 
Suppose otherwise. Since $U_H = \emptyset$, Lemma~\ref{lemma-vertex-attach-GUT} implies that $p_n$ is a twin of some vertex of $H$ with respect to $H$. Since $T_x \cup T_y \cup T_{x_1} \cup T_{x_h} = \emptyset$, we see that $p_n \in T_{x_i}$ for some $i \in \{2,\dots,h-1\}$ (and in particular, $h \geq 3$). Note that $p_0 = c$ is adjacent to $x,y \in V(H)$; let $j \in \{0,\dots,n-1\}$ be maximal with the property that $p_j$ has a neighbor in $V(H)$. We know that $p_j$ is anticomplete to $\{x_1,\dots,x_h\}$, and so $N_G(p_j) \cap V(H) \subseteq \{x,y\}$, and at least one of $x,y$ is adjacent to $p_j$. If $p_j$ is complete to $\{x,y\}$, then $G[(V(H) \setminus \{x_i\}) \cup \{p_j,\dots,p_n\}]$ is a $3PC(p_jxy,p_n)$, a contradiction. Thus, $p_j$ is adjacent to exactly one of $x,y$; by symmetry, we may assume that $p_j$ is adjacent to $x$ and nonadjacent to $y$. But now $G[(V(H) \setminus \{x_i\}) \cup \{p_j,\dots,p_n\}]$ is a $3PC(x,p_n)$, a contradiction. This proves (2). 
\end{proof}

\begin{quote} 
\emph{(3) Vertex $p_n$ is anticomplete to $\{x,y\}$.} 
\end{quote} 
\begin{proof}[Proof of (3)] 
Suppose otherwise. Since $p_n$ has a neighbor (namely $p_{n+1}$) in $\{x_1,\dots,x_h\}$, (2) implies that either $N_G(p_n) \cap V(H) = \{y,x_1\}$ or $N_G(p_n) \cap V(H) = \{x,x_h\}$; by symmetry, we may assume that $N_G(p_n) \cap V(H) = \{y,x_1\}$. Now, we know that $x$ is adjacent to $c = p_0$; let $j \in \{0,\dots,n-1\}$ be maximal with the property that $p_jx \in E(G)$. But now $Y = x,p_j,\dots,p_n,x_1,\dots,x_h,x$ is a hole and $(Y,y)$ a proper wheel in $G$, a contradiction. This proves (3). 
\end{proof} 

We know that $p_n$ has a neighbor (namely, $p_{n+1}$) in $\{x_1,\dots,x_h\}$. We may assume that $p_n$ has a neighbor in $\{x_1,\dots,x_{h-1}\}$, for the case when $x_h$ is the only neighbor of $p_n$ in $\{x_1,\dots,x_h\}$ is symmetric to the case when $x_1$ is the only neighbor of $p_n$ in $\{x_1,\dots,x_h\}$. Now, let $i \in \{1,\dots,h-1\}$ be minimal with the property that $p_nx_i \in E(G)$; it now follows from (2) and (3) that $x_i \in N_G(p_n) \cap V(H) \subseteq \{x_i,x_{i+1}\}$. 

Recall that $p_0 = c$ is adjacent to $x,y \in V(H)$; let $j \in \{0,\dots,n-1\}$ be maximal with the property that $p_j$ has a neighbor in $V(H)$. We know that $p_j$ is anticomplete to $\{x_1,\dots,x_h\}$, and so we have that $N_G(p_j) \cap V(H) \subseteq \{x,y\}$, and that $p_j$ is adjacent to at least one of $x,y$. Set $K = G[V(H) \cup \{p_j,\dots,p_n\}]$. It then follows from (3) and routine checking that $y$ is the only neighbor of $p_j$ in $V(H)$, and $x_1$ is the only neighbor of $p_n$ in $V(H)$, for otherwise, $K$ is a $3PC$, a contradiction. Note that we now have that $x$ is anticomplete to $\{p_j,\dots,p_n\}$. Recall that $x$ is adjacent to $p_0 = c$ (thus, $j \geq 1$); let $\ell \in \{0,\dots,j-1\}$ be maximal with the property that $xp_{\ell} \in E(G)$. But now $Y = x,p_{\ell},\dots,p_n,x_1,\dots,x_h,x$ is a hole and $(Y,y)$ a proper wheel in $G$, a contradiction. This completes the argument. 
\end{proof} 

\begin{theorem} \label{chi-bound-GUT-Ramsey}
The class $\mathcal{G}_{\text{UT}}$ is $\chi$-bounded by the function $g:\mathbb{N}^+ \rightarrow \mathbb{N}^+$ given by $g(1) = 1$, $g(2) = 3$, and $g(n) = 2(R(n-1,3)+4n-7)^2$ for $n \geq 3$. 
\end{theorem} 
\begin{proof} Let $f:\mathbb{N}^+ \rightarrow \mathbb{N}^+$ be given by $f(n) = \lfloor \frac{3}{2} n \rfloor$. Let $h:\mathbb{N}^+ \rightarrow \mathbb{N}^+$ be given by $h(1) = h(2) = 1$, and $h(n) = R(n-1,3)+4n-7$ for $n \geq 3$. Define $\widetilde{g}:\mathbb{N}^+ \rightarrow \mathbb{N}^+$ by setting $\widetilde{g}(n) = \max\{f(n)+2h(n)-2,2h(n)^2\}$. By Theorem~\ref{chi-bound-GUTcap}, $\mathcal{G}_{\text{UT}}^{\text{cap-free}}$ is $\chi$-bounded by $f$. On the other hand, Theorem~\ref{GUT-cap-double-star-cut} guarantees that every graph $G \in \mathcal{G}_{\text{UT}}$ either belongs to $\mathcal{G}_{\text{UT}}^{\text{cap-free}}$ or admits a cutset of size at most $h(\omega(G))$. Therefore, by Theorem~\ref{cor-small-cut}, we have that $\mathcal{G}_{\text{UT}}$ is $\chi$-bounded by $\widetilde{g}$. 

Now, to show that $\mathcal{G}_{\text{UT}}$ is in fact $\chi$-bounded by $g$, we fix $G \in\mathcal{G}_{\text{UT}}$, we set $\omega = \omega(G)$, and we prove that $\chi(G) \leq g(\omega)$. If $\omega = 1$, then the result is immediate. Next, suppose that $\omega = 2$. Since every cap contains a triangle, this implies that $G$ is cap-free. It follows that $G \in \mathcal{G}_{\text{UT}}^{\text{cap-free}}$, and so $\chi(G) \leq f(2) = 3 = g(2)$. From now on, we assume that $\omega \geq 3$. Since $\mathcal{G}_{\text{UT}}$ is $\chi$-bounded by $\widetilde{g}$, we just need to show that $g(\omega) = \widetilde{g}(\omega)$. By the definition of $g$ and $\widetilde{g}$, and by an easy calculation, we get the following: 
\begin{displaymath} 
\begin{array}{rcl} 
\widetilde{g}(\omega) & = & \max\{f(\omega)+2h(\omega)-2,2h(\omega)^2\} 
\\
\\
& = & \max\{\lfloor \frac{3}{2}\omega \rfloor+2R(\omega-1,3)+8\omega-14-2,2(R(\omega-1,3)+4\omega-7)^2\}
\\
\\
& = & 2(R(\omega-1,3)+4\omega-7)^2 
\\
\\
& = & g(\omega). 
\end{array} 
\end{displaymath} 
This completes the argument. 
\end{proof} 

Since $R(k,3)$ is of order $k^2/\log k$ \cite{Kim}, Theorem~\ref{chi-bound-GUT-Ramsey} implies that there exists a constant $c > 0$ such that every graph $G\in\mathcal{G}_{\text{UT}}$ that has at least one edge satisfies $\chi(G) \leq c\frac{\omega(G)^4}{\log^2 \omega(G)}$. We also have the following corollary of Theorem~\ref{chi-bound-GUT-Ramsey}. 

\begin{theorem}\label{chi-bound-GUT}
Every graph $G\in\mathcal{G}_{\text{UT}}$ satisfies $\chi(G) \leq 2\omega(G)^4$. 
\end{theorem} 
\begin{proof} Let $\omega = \omega(G)$. If $\omega \leq 2$, then the result follows immediately from Theorem~\ref{chi-bound-GUT-Ramsey}. So assume that $\omega \geq 3$. In view of Theorem~\ref{chi-bound-GUT-Ramsey}, we need only show that $2(R(\omega-1,3)+4\omega-7)^2 \leq 2\omega^4$, which is, in turn, equivalent to showing that $R(\omega-1,3)+4\omega-7 \leq \omega^2$. By the Erd\H os-Szekeres upper bound for Ramsey numbers (see \cite{West}), we know that $R(k,\ell) \leq {k+\ell-2 \choose \ell-1}$ for all $k,\ell \in \mathbb{N}^+$; thus, $R(\omega-1,3) \leq {\omega \choose 2}$, and consequently, $R(\omega-1,3)+4\omega-7 \leq {\omega \choose 2}+4\omega-7$. A simple calculation now shows that ${\omega \choose 2}+4\omega-7 \leq \omega^2$, and the result follows. 
\end{proof}

\section{Algorithms} \label{sec:alg}

Unless stated otherwise, in all our algorithms, $n$ denotes the number of vertices and $m$ the number of edges of the input graph. 

We remark that our algorithms are {\em robust}, that is, they either produce a correct solution to the problem in question for the input (weighted) graph, or they correctly determine that the graph does not belong to the class under consideration.

Our decomposition theorems for classes $\mathcal{G}_{\text{UT}},\mathcal{G}_{\text{U}},\mathcal{G}_{\text{T}},\mathcal{G}_{\text{UT}}^{\text{cap-free}}$ all involve clique-cutsets, and for this reason, the algorithmic tools developed in~\cite{Tarjan} for handling clique-cutsets will be used extensively in this section. Our next subsection (subsection~\ref{sec:alg:Tarjan}) heavily borrows from~\cite{Tarjan}. 

\subsection{Clique-cutset decomposition tree} \label{sec:alg:Tarjan} 

A function $f:\mathbb{N}^p \rightarrow \mathbb{N}$ is said to be {\em nondecreasing} if it satisfies the property that for all $n_1,\dots,n_p,n_1',\dots,n_p' \in \mathbb{N}$ such that $n_1 \leq n_1',\dots,n_p \leq n_p'$, we have that $f(n_1,\dots,n_p) \leq f(n_1',\dots,n_p')$; $f$ is said to be {\em superadditive} if for all $n_1,\dots,n_p,n_1',\dots,n_p' \in \mathbb{N}$, we have that $f(n_1,\dots,n_p)+f(n_1',\dots,n_p') \leq f(n_1+n_1',\dots,n_p+n_p')$. Clearly, any superadditive function is nondecreasing. Note also that any polynomial function, all of whose coefficients are nonnegative, and whose free coefficient is zero, is superadditive. 

A {\em rooted tree} is an ordered pair $(T,r)$, where $T$ is a tree, and $r$ is a node of $T$ called the {\em root}. If $T$ has at least two nodes, then the {\em leaves} of $(T,r)$ are the nodes in $V(T) \setminus \{r\}$ that are of degree one in $T$; and if $V(T) = \{r\}$, then we consider the root $r$ to be a {\em leaf} of $T$. The set of all leaves of $(T,r)$ is denoted by $\mathcal{L}(T,r)$. The {\em internal nodes} of $(T,r)$ are the nodes in $V(T) \setminus \mathcal{L}(T,r)$. If $u,v \in V(T)$, then we say that $v$ is a {\em descendant} of $u$, and that $u$ is an {\em ancestor} of $v$ in $(T,r)$, provided that $u \neq v$ and $u$ belongs to the unique path between $r$ and $v$ in $T$. Given $u,v \in V(T)$, we say that $v$ is a {\em child} of $u$, and that $u$ is the {\em parent} of $v$ in $(T,r)$ provided that $v$ is a descendant of $u$ in $(T,r)$, and $uv \in E(T)$. Clearly, every node other than the root has a unique parent in $(T,r)$, leaves have no children in $(T,r)$, and all internal nodes have at least one child in $(T,r)$. If $u \in V(T)$, then the {\em subtree of $(T,r)$ rooted at $u$} is the rooted tree $(T_u,u)$, where $T_u$ is the subtree of $T$ induced by $u$ and all the descendants of $u$ in $(T,r)$. 

A clique-cut-partition $(A,B,C)$ of a graph $G$ is {\em extreme} if $G[A \cup C]$ admits no clique-cutset. It is easy to see that if $G$ admits a clique-cutset, then $G$ admits an extreme clique-cut-partition. (To see this, suppose that $G$ admits a clique-cutset. Choose a clique-cut-partition $(A,B,C)$ of $G$ such that $|A \cup C|$ is as small as possible. Then $(A,B,C)$ is readily seen to be an extreme clique-cut-partition of $G$.) This implies that for every graph $G$, there exists a {\em clique-cutset decomposition tree} of $G$, which is a rooted tree $(T_G,r)$ equipped with an associated family $\{V^u\}_{u \in V(T_G)}$ of subsets of $V(G)$, having the following properties: 
\begin{itemize} 
\item if $G$ admits no clique-cutset, then $V(T_G) = \{r\}$ and $V^r = V(G)$; 
\item if $G$ does admit a clique-cutset, then there exists an extreme clique-cut-partition $(A,B,C)$ of $G$ such that $V^r = C$, $r$ has precisely two children in $(T_G,r)$, one of them (call it $x$) is a leaf of $(T_G,r)$ and satisfies $V^x = A \cup C$, and the subtree of $(T_G,r)$ rooted at the other child of $r$ is a clique-cutset decomposition tree of $G[B \cup C]$. 
\end{itemize} 
Note that if $(T_G,r)$ is a clique-cutset decomposition tree of a graph $G$, then $|V(T_G)| \leq 2|V(G)|-1$ and $|\mathcal{L}(T_G,r)| \leq |V(G)|$. It was shown in~\cite{Tarjan} that a clique-cutset decomposition tree of an arbitrary input graph can be computed in $O(nm)$ time. We remark that a clique-cutset decomposition tree of a given graph need not be unique. 

If $G$ is a graph, $(T_G,r)$ a clique-cutset decomposition tree of $G$, and $u$ a node of $T_G$, then we set 
\begin{displaymath} 
G^u = G[\bigcup\{V^x \mid \text{$x = u$ or $x$ is a descendant of $u$ in $(T_G,r)$}\}]. 
\end{displaymath} 
Note that the family $\{G^u\}_{u \in V(T_G)}$ can be computed in $O(n^2+nm)$ time. We also remark that for all $u \in V(T_G)$, if $u$ is a leaf of $(T_G,r)$, then $G^u$ admits no clique-cutset, and if $u$ is an internal node of $(T_G,r)$, then $V^u$ is a clique-cutset of $G^u$. 

The following simple lemma will be used repeatedly. 

\begin{lemma} \label{lemma-leaves-basic} Let $\mathcal{B}$ and $\mathcal{G}$ be hereditary classes, and assume that every graph in $\mathcal{G}$ either belongs to $\mathcal{B}$ or admits a clique-cutset. Let $G \in \mathcal{G}$, let $(T_G,r)$ be a clique-cutset decomposition tree of $G$, and let $\{G^u\}_{u \in V(T_G)}$ be the associated family of induced subgraphs of $G$. Then all graphs in the family $\{G^u\}_{u \in \mathcal{L}(T_G,r)}$, and all their induced subgraphs, belong to $\mathcal{B}$. 
\end{lemma} 
\begin{proof} 
Since $\mathcal{G}$ is hereditary and $G \in \mathcal{G}$, we know that all induced subgraphs of $G$ belong to $\mathcal{G}$; in particular, each graph in $\{G^u\}_{u \in \mathcal{L}(T_G,r)}$ belongs to $\mathcal{G}$. By the definition of a clique-cutset decomposition tree, no graph in $\{G^u\}_{u \in \mathcal{L}(T_G,r)}$ admits a clique-cutset. Since (by hypothesis) all graphs in $\mathcal{G}$ that do not admit a clique-cutset belong to $\mathcal{B}$, we deduce that all graphs in $\{G^u\}_{u \in \mathcal{L}(T_G,r)}$ belong to $\mathcal{B}$. The result now follows from the fact that $\mathcal{B}$ is hereditary. 
\end{proof} 

Our next lemma can be seen as a partial converse of Lemma~\ref{lemma-leaves-basic}. 

\begin{lemma} \label{lemma-no-clique-cut-ind-sg} Let $G$ and $H$ be graphs, let $(T_G,r)$ be a clique-cutset decomposition tree of $G$, and let $\{G^u\}_{u \in V(T_G)}$ be the associated family of induced subgraphs of $G$. Assume that $H$ does not admit a clique-cutset, and assume that for all $u \in \mathcal{L}(T_G,r)$, $G^u$ is $H$-free. Then $G$ is $H$-free. 
\end{lemma} 
\begin{proof} 
Clearly, if $(A,B,C)$ is a clique-cut-partition of a graph $K$, then the fact that $H$ admits no clique-cutset implies that $K$ is $H$-free if and only if both $K[A \cup C]$ and $K[B \cup C]$ are $H$-free. The result now easily follows from the definition of a clique-cutset decomposition tree. 
\end{proof} 

We now show how a clique-cutset decomposition tree can be used to solve the optimal coloring problem, as well as the maximum weight clique and maximum weight stable set problems, in certain classes of graphs. Lemmas~\ref{clique-cutset decomposition tree-coloring} and~\ref{clique-cutset decomposition tree-clique} (which deal with the optimal coloring and maximum weight clique problems, respectively) and their proofs are very similar to the results and arguments from~\cite{Tarjan}, and we include them here for the sake of completeness. The maximum weight stable set problem is dealt with in a slightly different way than in~\cite{Tarjan} (see Lemmas~\ref{lemma-combining-stable} and~\ref{clique-cutset decomposition tree-stable}, and the discussion that follows them).

\begin{lemma} \label{clique-cutset decomposition tree-coloring} Let $\mathcal{B}$ and $\mathcal{G}$ be hereditary classes, and assume that every graph in $\mathcal{G}$ either belongs to $\mathcal{B}$ or admits a clique-cutset. Let $f:\mathbb{N} \times \mathbb{N} \rightarrow \mathbb{N}$ be a nondecreasing function. Assume that there exists an algorithm $\mathbf{A}$ with the following specifications: 
\begin{itemize} 
\item Input: A graph $G$; 
\item Output: Either an optimal coloring of $G$, or the true statement that $G \notin \mathcal{B}$; 
\item Running time: At most $f(n,m)$. 
\end{itemize} 
Then there exists an algorithm $\mathbf{B}$ with the following specifications: 
\begin{itemize} 
\item Input: A graph $G$; 
\item Output: Either an optimal coloring of $G$, or the true statement that $G \notin \mathcal{G}$; 
\item Running time: $O(nf(n,m)+n^2+nm)$. 
\end{itemize} 
\end{lemma} 
\begin{proof} 
Let $G$ be an input graph. We first compute a clique-cutset decomposition tree $(T_G,r)$ of $G$ and the associated family $\{G^u\}_{u \in V(T_G)}$ of induced subgraphs of $G$ in $O(n^2+nm)$ time. By Lemma~\ref{lemma-leaves-basic}, if $G \in \mathcal{G}$, then all graphs in the family $\{G^u\}_{u \in \mathcal{L}(T_G,r)}$ belong to $\mathcal{B}$. 

Suppose first that $T_G$ has just one node (namely, the root $r$). In this case, we have that either $G \in \mathcal{B}$ or $G \notin \mathcal{G}$. We now run the algorithm $\mathbf{A}$ with input $G$; this takes at most $f(n,m)$ time. If the algorithm $\mathbf{A}$ returns the answer that $G \notin \mathcal{B}$, then our algorithm $\mathbf{B}$ returns the answer that $G \notin \mathcal{G}$ and stops. On the other hand, if the algorithm $\mathbf{A}$ returns an optimal coloring of $G$, then the algorithm $\mathbf{B}$ returns this coloring and stops. 

Suppose now that $T_G$ has more than one node. Let $x$ and $y$ be the children of the root $r$ in $(T_G,r)$; by symmetry, we may assume that $x \in \mathcal{L}(T_G,r)$. We first run the algorithm $\mathbf{A}$ with input $G^x$; this takes at most $f(n,m)$ time. If we obtain the answer that $G^x \notin \mathcal{B}$, then the algorithm $\mathbf{B}$ returns the answer that $G \notin \mathcal{G}$ and stops. Suppose now that the algorithm $\mathbf{A}$ returned an optimal coloring of $G^x$. We now recursively either determine that $G^y \notin \mathcal{G}$ or obtain an optimal coloring of $G^y$. If we obtain the answer that $G^y \notin \mathcal{G}$, then the algorithm $\mathbf{B}$ returns the answer that $G \notin \mathcal{G}$ and stops. Suppose now that we obtained an optimal coloring of $G^y$. We then permute and rename the colors used by the colorings of $G^x$ and $G^y$ to ensure that the two colorings agree on $V^r$, and that the set of colors used on one of $G^x,G^y$ is a subset of the set of colors used on the other; this takes $O(n)$ time. Finally, we take the union of the colorings of $G^x$ and $G^y$ in $O(n)$ time, and we obtain an optimal coloring of $G$; we return this coloring of $G$ and stop. 

Clearly, the algorithm is correct; it remains to estimate its running time. We run the algorithm $\mathbf{A}$ at most $|\mathcal{L}(T_G,r)| \leq n$ times, and each time, the input is an induced subgraph of the graph $G$; thus, the running time of all the calls to $\mathbf{A}$ together take at most $nf(n,m)$ time. Further, since $|V(T_G)| \leq 2n-1$, it is easy to see that all other steps of the algorithm take $O(n^2+nm)$ time. It follows that the total running time of the algorithm is $O(nf(n,m)+n^2+nm)$. 
\end{proof} 

\begin{lemma}\label{clique-cutset decomposition tree-clique} Let $\mathcal{B}$ and $\mathcal{G}$ be hereditary classes, and assume that every graph in $\mathcal{G}$ either belongs to $\mathcal{B}$ or admits a clique-cutset. Let $f:\mathbb{N} \times \mathbb{N} \rightarrow \mathbb{N}$ be a nondecreasing function, and assume that there exists an algorithm $\mathbf{A}$ with the following specifications: 
\begin{itemize} 
\item Input: A weighted graph $(G,w)$; 
\item Output: Either a maximum weight clique $C$ of $(G,w)$, or the true statement that $G \notin \mathcal{B}$; 
\item Running time: At most $f(n,m)$. 
\end{itemize} 
Then there exists an algorithm $\mathbf{B}$ with the following specifications: 
\begin{itemize} 
\item Input: A weighted graph $(G,w)$; 
\item Output: Either a maximum weight clique $C$ of $(G,w)$, or the true statement that $G \notin \mathcal{G}$; 
\item Running time: $O(nf(n,m)+n^2+nm)$. 
\end{itemize} 
\end{lemma} 
\begin{proof} 
Let $(G,w)$ be an input weighted graph. We first compute a clique-cutset decomposition tree $(T_G,r)$ of $G$ and the associated family $\{G^u\}_{u \in V(T_G)}$ of induced subgraphs of $G$ in $O(n^2+nm)$ time. Clearly, $\omega(G,w) = \max\{\omega(G^u,w) \mid u \in \mathcal{L}(T_G,r)\}$. By Lemma~\ref{lemma-leaves-basic}, we know that if $G \in \mathcal{G}$, then all graphs in the family $\{G^u \mid u \in \mathcal{L}(T_G,r)\}$ belong to $\mathcal{B}$. For each $u \in \mathcal{L}(T_G,r)$, we call the algorithm $\mathbf{A}$ with input $(G^u,w)$; since $|\mathcal{L}(T_G,r)| \leq n$, we see that running the algorithm $\mathbf{A}$ for all graphs in the family $\{G^u\}_{u \in \mathcal{L}(T_G,r)}$ takes at most $nf(n,m)$ time. If for some $u \in \mathcal{L}(T_G,r)$, the algorithm $\mathbf{A}$ returns the answer that $G^u \notin \mathcal{B}$, then we return the answer that $G \notin \mathcal{G}$ and stop. Suppose now that for each $u \in \mathcal{L}(T_G,r)$, the algorithm $\mathbf{A}$ returned a maximum weight clique $C^u$ of $(G^u,w)$. We now find a node $x \in \mathcal{L}(T_G,r)$ such that $\omega(G^x,w) = \max\{\omega(G^u,w) \mid u \in \mathcal{L}(T_G,r)\}$; since $|\mathcal{L}(T_G,r)| \leq n$, this takes $O(n^2)$ time. Clearly, $C^x$ is a maximum weight clique of $(G,w)$; we return $C^x$ and stop. It is clear that the algorithm is correct, and that its running time is $O(nf(n,m)+n^2+nm)$. 
\end{proof} 

\begin{lemma} \label{lemma-combining-stable} Let $(G,w)$ be a weighted graph, and let $(A,B,C)$ be a clique-cut-partition of $G$. Define $w_B:B \cup C \rightarrow \mathbb{R}$ by setting $w_B \upharpoonright B = w \upharpoonright B$, and for all $c \in C$, setting $w_B(c) = \alpha(G[A \cup \{c\}],w)-\alpha(G[A],w)$. For each $C' \subseteq C$ such that $|C'| \leq 1$, let $S_{A \cup C'}$ be a maximum weight stable set of $(G[A \cup C'],w)$. Let $S_B$ be a maximum weight stable set of $(G[B \cup C],w_B)$, and assume that $w_B(v) > 0$ for all $v \in S_B$. Let $\widetilde{C} = S_B \cap C$. Then $|\widetilde{C}| \leq 1$, and $S_{A \cup \widetilde{C}} \cup S_B$ is a maximum weight stable set of $(G,w)$. 
\end{lemma} 
\begin{proof} 
Since $C$ is a clique and $S_B$ a stable set of $G$, we have that $|\widetilde{C}| \leq 1$. Set $S = S_{A \cup \widetilde{C}} \cup S_B$. We must show that $S$ is a maximum weight stable set of $(G,w)$. 

\begin{quote} 
\emph{(1) For all $C' \subseteq C$ such that $|C'| \leq 1$, we have that $w_B(C') = \alpha(G[A \cup C'],w)-\alpha(G[A],w)$.}
\end{quote} 
\begin{proof}[Proof of (1)] 
Fix $C' \subseteq C$ such that $|C'| \leq 1$. If $C' = \emptyset$, then $w_B(C') = 0$ and $A \cup C' = A$, and the result is immediate. So assume that $|C'| = 1$, and let $c$ be the unique vertex of $C'$. Then by construction, 
\begin{displaymath} 
\begin{array}{rcl} 
w_B(C') & = & w_B(c) 
\\
\\
& = & \alpha(G[A \cup \{c\}],w)-\alpha(G[A],w) 
\\
\\
& = & \alpha(G[A \cup C'],w)-\alpha(G[A],w), 
\end{array} 
\end{displaymath} 
which is what we needed. This proves (1). 
\end{proof} 

\begin{quote} 
\emph{(2) $S_{A \cup \widetilde{C}} \cap C = \widetilde{C}$. Consequently, $S$ is a stable set.} 
\end{quote} 
\begin{proof}[Proof of (2)] 
By construction, $S_B \cap C = \widetilde{C}$. Thus, since $S_{A \cup \widetilde{C}}$ and $S_B$ are stable sets of $G$, and since $A$ is anticomplete to $B$ in $G$, the first statement clearly implies the second. 

It remains to show that $S_{A \cup \widetilde{C}} \cap C = \widetilde{C}$. By construction, $S_{A \cup \widetilde{C}} \subseteq A \cup \widetilde{C}$; consequently, $S_{A \cup \widetilde{C}} \cap C \subseteq \widetilde{C}$. It remains to show that $\widetilde{C} \subseteq S_{A \cup \widetilde{C}} \cap C$. If $\widetilde{C} = \emptyset$, then this is immediate. So assume that $\widetilde{C} \neq \emptyset$, so that $|\widetilde{C}| = 1$. Let $c$ be the unique vertex of $\widetilde{C}$. Since $c \in S_B$, we have that $w(c) > 0$. By construction, $w_B(c) = \alpha(G[A \cup \{c\}],w)-\alpha(G[A],w)$, and so $\alpha(G[A],w) < \alpha(G[A \cup \{c\}],w)$. Thus, every maximum weight stable set of $(G[A \cup \{c\}],w) = (G[A \cup \widetilde{C}],w)$ contains $c$; in particular, $c \in S_{A \cup \widetilde{C}}$, and it follows that $\widetilde{C} \subseteq S_{A \cup \widetilde{C}} \cap C$. This proves (2). 
\end{proof} 

\begin{quote} 
\emph{(3) $w(S) = \alpha(G[B \cup C],w_B)+\alpha(G[A],w)$.} 
\end{quote} 
\begin{proof}[Proof of (3)] 
By (2), and by construction, we have that $S_{A \cup \widetilde{C}} \cap C = S_B \cap C = \widetilde{C}$. We know that $|\widetilde{C}| \leq 1$, and so by (1), $w_B(\widetilde{C}) = \alpha(G[A \cup \widetilde{C}],w)-\alpha(G[A],w)$. But now we have that 
\begin{displaymath} 
\begin{array}{rcl} 
w(S) & = & w(S_{A \cup \widetilde{C}})+w(S_B \setminus \widetilde{C}) 
\\
\\
& = & w(S_{A \cup \widetilde{C}})+w_B(S_B \setminus \widetilde{C}) 
\\
\\
& = & w(S_{A \cup \widetilde{C}})+w_B(S_B)-w_B(\widetilde{C}) 
\\
\\
& = & \alpha(G[A \cup \widetilde{C}],w)+\alpha(G[B \cup C],w_B)-(\alpha(G[A \cup \widetilde{C}],w)-\alpha(G[A],w)) 
\\
\\
& = & \alpha(G[B \cup C],w_B)+\alpha(G[A],w), 
\end{array} 
\end{displaymath} 
which is what we needed. This proves (3). 
\end{proof} 

\begin{quote}
\emph{(4) Every stable set $S'$ of $G$ satisfies $w(S') \leq \alpha(G[B \cup C],w_B)+\alpha(G[A],w)$.} 
\end{quote} 
\begin{proof}[Proof of (4)] 
Fix a stable set $S'$ of $G$; we must show that $w(S') \leq \alpha(G[B \cup C],w_B)+\alpha(G[A],w)$. Set $C' = S' \cap C$; since $S'$ is a stable set and $C$ a clique of $G$, we have that $|C'| \leq 1$. By (1), we have that $w_B(C') = \alpha(G[A \cup C'],w)-\alpha(G[A],w)$. We then have that 
\begin{displaymath} 
\begin{array}{rcl} 
w(S') & = & w(S' \cap (A \cup C))+w(S' \cap B) 
\\
\\
& = & w((S' \cap A) \cup C')+w_B(S' \cap B) 
\\
\\
& = & w((S' \cap A) \cup C')+w_B(S' \cap (B \cup C))-w_B(C') 
\\
\\
& \leq & \alpha(G[A \cup C'],w)+\alpha(G[B \cup C],w_B)-(\alpha(G[A \cup C'],w)-\alpha(G[A],w)) 
\\
\\
& = & \alpha(G[B \cup C],w_B)+\alpha(G[A],w), 
\end{array} 
\end{displaymath} 
which is what we needed. This proves (4). 
\end{proof} 

Clearly, (2), (3), and (4) imply that $S$ is a maximum weight stable set of $(G,w)$. 
\end{proof}

\begin{lemma} \label{clique-cutset decomposition tree-stable} Let $\mathcal{B}$ and $\mathcal{G}$ be hereditary classes, and assume that every graph in $\mathcal{G}$ either belongs to $\mathcal{B}$ or admits a clique-cutset. Let $f:\mathbb{N} \times \mathbb{N} \rightarrow \mathbb{N}$ be a superadditive polynomial function. Assume that there exists an algorithm $\mathbf{A}$ with the following specifications: 
\begin{itemize} 
\item Input: A weighted graph $(G,w)$; 
\item Output: Either a maximum weight stable set of $(G,w)$, or the true statement that $G \notin \mathcal{B}$; 
\item Running time: At most $f(n,m)$. 
\end{itemize} 
Then there exists an algorithm $\mathbf{B}$ with the following specifications: 
\begin{itemize} 
\item Input: A weighted graph $(G,w)$; 
\item Output: Either a maximum weight stable set of $(G,w)$, or the true statement that $G \notin \mathcal{G}$; 
\item Running time: $O(nf(n,m)+n^2+nm)$. 
\end{itemize} 
\end{lemma} 
\begin{proof} 
Let $(G,w)$ be an input weighted graph. We begin by computing a clique-cutset decomposition tree $(T_G,r)$ of $G$ and the associated family $\{G^u\}_{u \in V(T_G)}$ of induced subgraphs of $G$ in $O(n^2+nm)$ time. By Lemma~\ref{lemma-leaves-basic}, if $G \in \mathcal{G}$, then all graphs in the family $\{G^u\}_{u \in \mathcal{L}(T_G,r)}$, and all their induced subgraphs, belong to $\mathcal{B}$. 

Suppose first that $T_G$ has precisely one node (namely, the root $r$). In this case, we have that either $G \in \mathcal{B}$ or $G \notin \mathcal{G}$. We call the algorithm $\mathbf{A}$ with input $(G,w)$; this takes at most $f(n,m)$ time. If we obtain the answer that $G \notin \mathcal{B}$, then we return the answer that $G \notin \mathcal{G}$ and stop. Otherwise, $\mathbf{A}$ returns a maximum weight stable set of $(G,w)$, and we return that stable set and stop. 

From now on, we assume that $T_G$ has more than one node; in particular, $r \notin \mathcal{L}(T_G,r)$. For each $u \in \mathcal{L}(T_G,r)$, let $p(u)$ denote the parent of $u$ in $(T_G,u)$. Now, for each $u \in \mathcal{L}(T_G,r)$, we compute the sets $A^u = V^u \setminus V^{p(u)}$, $B^u = V(G^{p(u)}) \setminus V^u$, and $C^u = V^{p(u)}$; clearly, $(A^u,B^u,C^u)$ is an extreme clique-cut-partition of $G^{p(u)}$, and since $|\mathcal{L}(T_G,r)| \leq n$, computing the families $\{A^u\}_{u \in \mathcal{L}(T_G,r)}$, $\{B^u\}_{u \in \mathcal{L}(T_G,r)}$, and $\{C^u\}_{u \in \mathcal{L}(T_G,r)}$ takes $O(n^2)$ time. Next, we will use the following notation: for each $u \in \mathcal{L}(T_G,r)$, we set $n_u = |A^u|$, and we let $m_u$ be the number of edges of $G^u$, at least one of whose endpoints belongs to $A^u$. Note that $\sum_{u \in \mathcal{L}(T_G,r)} n_u \leq n$ and $\sum_{u \in \mathcal{L}(T_G,r)} m_u \leq m$. 

Let $x$ and $y$ be the children of the root $r$ in $T_G$; by symmetry, we may assume that $x \in \mathcal{L}(T_G,r)$. We form the graph $G[A^x]$ in $O(n+m)$ time, and then for each $c \in C^x$, we form the graph $G[A^x \cup \{c\}]$ in $O(n_x+m_x)$ time. Clearly, forming the family $\{G[A^x \cup C'] \mid C' \subseteq C^x, |C'| \leq 1\}$ takes $O(n+m+n(n_x+m_x))$ time. Now, for each $C' \subseteq C^x$ with $|C'| \leq 1$, we call the algorithm $\mathbf{A}$ with input $G[A^x \cup C']$. Clearly, we make $O(n)$ calls to the algorithm $\mathbf{A}$, and each input graph has at most $n_x+1$ vertices and $m_x$ edges; thus, together, these calls to the algorithm $\mathbf{A}$ take at most $nf(n_x+1,m_x)$ time, which is $O(nf(n_x,m_x))$ time (we use the fact that $f$ is polynomial and superadditive). If for some $C' \subseteq C$ with $|C'| \leq 1$, the algorithm $\mathbf{A}$ returns the answer that $G[A^x \cup C'] \notin \mathcal{B}$, then we return the answer that $G \notin \mathcal{G}$ and stop. Assume now that for all $C' \subseteq C$ such that $|C'| \leq 1$, the algorithm $\mathbf{A}$ returned a maximum weight stable set $S_{A^x \cup C'}$ of $(G[A^x \cup C'],w)$. Clearly, for all $C' \subseteq C^x$ with $|C'| \leq 1$, we have that $\alpha(G[A^x \cup C'],w) = w(S_{A^x \cup C'})$, and we see that the family $\{\alpha(G[A^x \cup C'],w) \mid C' \subseteq C, |C'| \leq 1\}$ can be computed in $O(n_xn)$ time. Next, we form the weight function $w_B$ for $G^y = G[B^x \cup C^x]$ as in Lemma~\ref{lemma-combining-stable}; this takes $O(n)$ time. Then, we recursively either determine that $G^y \notin \mathcal{G}$ or obtain a maximum weight stable set $S_B$ of $(G^y,w_B)$. In the former case, we return the answer that $G \notin \mathcal{G}$ and stop. Suppose now that we obtained a maximum weight stable set $S_B$ of $(G^y,w_B)$. Clearly, $w_B(v) \geq 0$ for all $v \in S_B$, and furthermore, we may assume that $w_B(v) > 0$ for all $v \in S_B$, for otherwise, we simply delete from $S_B$ all the vertices to which $w_B$ assigns weight zero. Set $C' = C^x \cap S_B$; since $C'$ is a clique and $S_B$ a stable set, we know that $|C'| \leq 1$. Set $S = S_{A^x \cup C'} \cup S_B$. By Lemma~\ref{lemma-combining-stable}, $S$ is a maximum weight stable set of $(G,w)$. We now return the set $S$ and stop. 

It is clear that the algorithm is correct; it remains to estimate its running time. Let $u^*$ be the last leaf of $(T_G,r)$ that our algorithm $\mathbf{B}$ reaches. With the possible exception of the leaf $u^*$, for each leaf $u$ of $(T_G,r)$ reached by the algorithm $\mathbf{B}$, we call the algorithm $\mathbf{A}$ on at most $n$ induced subgraphs of $G^u$, and as we see from the description of the algorithm, this takes $O(nf(n_u,m_u))$ time. Furthermore, we may possibly call the algorithm $\mathbf{A}$ on the graph $G^{u^*}$; this takes at most $f(n,m)$ time. Thus, the total time that all the calls to the algorithm $\mathbf{A}$ take is $O((\sum_{u \in \mathcal{L}(T_G,r)} nf(n_u,m_u))+f(n,m))$; since $\sum_{u \in \mathcal{L}(T_G,r)} n_u \leq n$ and $\sum_{u \in \mathcal{L}(T_G,r)} m_u \leq m$, and since $f$ is superadditive and polynomial, this is $O(nf(n,m))$. Using the fact that $|V(T_G)| \leq 2n-1$, and the fact that $\sum_{u \in \mathcal{L}(T_G,r)} n_u \leq n$ and $\sum_{u \in \mathcal{L}(T_G,r)} m_u \leq m$, we readily see that all other steps of the algorithm take $O(n^2+nm)$ time. It now follows that the total running time of the algorithm $\mathbf{B}$ is $O(nf(n,m)+n^2+nm)$. 
\end{proof} 

Let us now briefly discuss the ways in which Lemmas~\ref{lemma-combining-stable} and~\ref{clique-cutset decomposition tree-stable} differ from their analogs in~\cite{Tarjan}. First of all, in Lemma~\ref{lemma-combining-stable} (which is used in the proof of Lemma~\ref{clique-cutset decomposition tree-stable}), the weight function $w_B$ is defined in a slightly different way than the corresponding weight function from~\cite{Tarjan}; the advantage of our approach is that we never introduce negative weights, that is to say, if the weight function $w$ assigns only nonnegative weights, then so does the weight function $w_B$. Second of all, one of the hypotheses of Lemma~\ref{clique-cutset decomposition tree-stable} is that the function $f$ is polynomial and superadditive (this hypothesis is absent from~\cite{Tarjan}); this additional hypothesis, together with a more involved complexity analysis, allows us to obtain a running time that is slightly better than the one from~\cite{Tarjan}. We remark that if, in the statement of Lemma~\ref{clique-cutset decomposition tree-stable}, we replaced the hypothesis that $f$ is polynomial and superadditive with the (weaker) hypothesis that $f$ is nondecreasing, then we would simply obtain a running time of $O(n^2f(n,m)+n^2+nm)$ for the algorithm $\mathbf{B}$. 

\subsection{Algorithms for chordal graphs and hyperholes}

A vertex $v$ in a graph $G$ is {\em simplicial} if $N_G(v)$ is a (possibly empty) clique of $G$. A {\em simplicial elimination ordering} of a graph $G$ is an ordering $v_1,\dots,v_n$ of the vertices of $G$ such that for all $i \in \{1,\dots,n\}$, $v_i$ is a simplicial vertex of $G[v_i,\dots,v_n]$. It is well-known (and easy to show) that a graph is chordal if and only if it has a simplicial elimination ordering. There is an $O(n+m)$ time algorithm that either produces a simplicial elimination ordering of the input graph, or determines that the graph is not chordal~\cite{RTL76}. Clearly, given a chordal graph $G$ and a simplicial elimination ordering $v_1,\dots,v_n$ of $G$, an optimal coloring of $G$ can be found in $O(n+m)$ time (we simply color greedily, using the ordering $v_n,\dots,v_1$, that is, the reverse of the input simplicial elimination ordering). Further, there is an $O(n+m)$ time algorithm that, given a weighted chordal graph $(G,w)$ and a simplicial elimination ordering $v_1,\dots,v_n$ for $G$, finds a maximum weight stable set of $(G,w)$~\cite{FrankChordalStable}. Finally, given a weighted chordal graph $(G,w)$ and a simplicial elimination ordering $v_1,\dots,v_n$ of $G$, a maximum weight clique of $G$ can be found in $O(n+m)$ time as follows. First, we may assume that $w$ assigns positive weight to all vertices of $G$. (If $w$ does not assign positive weight to any vertex of $G$, then $\emptyset$ is a maximum weight clique of $G$. If $w$ assigns positive weight to some, but not all, vertices of $G$, then we find and delete from $G$ and from the sequence $v_1,\dots,v_n$ all the vertices of $G$ to which $w$ assigns negative or zero weight.) For each $i \in \{1,\dots,n\}$, we form the set $C_i = \{v_j \mid j \geq i, v_j \in N_G[v_i]\}$. We then find an index $i \in \{1,\dots,n\}$ such that $w(C_i) = \max\{|C_j| \mid 1 \leq j \leq n\}$. It is easy to see that $C_i$ is a maximum weight clique of $G$. For the sake of future reference, we summarize these results in the lemma below. 

\begin{lemma} \label{lemma-chordal} Chordal graphs can be recognized and optimally colored in $O(n+m)$ time. A maximum weight clique and a maximum weight stable set of a weighted chordal graph can be found in $O(n+m)$ time. 
\end{lemma} 

Given a graph $G$, two distinct vertices $u,v \in V(G)$ are said to be {\em true twins} in $G$ if $N_G[u] = N_G[v]$. Clearly, the relation of being a true twin is an equivalence relation; a {\em true twin class} of $G$ is an equivalence class with respect to the true twin relation. Thus, $V(G)$ can be partitioned into true twin classes of $G$ in a unique way, and clearly, every true twin class of $G$ is a clique of $G$. An exercise from~\cite{Spinrad} states that, given an input graph $G$, all true twin classes of $G$ can be found in $O(n+m)$ time; a detailed proof of this result can be found in~\cite{CapEvenHoleFree}. Given a graph $G$ and a partition $\mathcal{P}$ of $V(G)$ into true twin classes of $G$, we define the graph $G_\mathcal{P}$ (called the {\em quotient graph} of $G$ with respect to $\mathcal{P}$) to be the graph whose vertex set is $\mathcal{P}$, and in which distinct $A,B \in \mathcal{P}$ are adjacent if and only if $A$ and $B$ are complete to each other in $G$. Clearly, given $G$ and $\mathcal{P}$, the graph $G_\mathcal{P}$ can be found in $O(n+m)$ time. We summarize these results below for future reference. 

\begin{lemma} \label{true-twin-alg} There exists an algorithm with the following specifications: 
\begin{itemize}
\item Input: A graph $G$; 
\item Output: The partition $\mathcal{P}$ of $V(G)$ into true twin classes, and the quotient graph $G_\mathcal{P}$; 
\item Running time: $O(n+m)$. 
\end{itemize} 
\end{lemma} 

Clearly, a graph $G$ is a hole (resp. long hole) if and only if $G$ has at least four vertices (resp. at least five vertices), $G$ is connected (this can be checked using, for example, BFS), and all vertices of $G$ are of degree two. Thus, holes and long holes can be recognized in $O(n+m)$ time. The proof of our next lemma (Lemma~\ref{lemma-hyperhole-rec}) is an easy exercise, and we leave it to the reader. 

\begin{lemma} \label{lemma-hyperhole-rec} Let $G$ be a graph, and let $\mathcal{P}$ be a partition of $V(G)$ into true twin classes of $G$. Then $G$ is a hyperhole (resp. long hyperhole) if and only if $G_\mathcal{P}$ is a hole (resp. long hole). Consequently, there exists an $O(n+m)$ time recognition algorithm for hyperholes (resp. for long hyperholes). 
\end{lemma}

Given a weighted graph $(G,w)$, where $w$ is positive integer valued, a {\em proper weighted coloring} of $(G,w)$ is a function $c$ that assigns to each vertex $v \in V(G)$ a set of precisely $w(v)$ colors, and furthermore, satisfies the property that $c(v_1) \cap c(v_2) = \emptyset$ for all adjacent vertices $v_1,v_2 \in V(G)$. An {\em optimal weighted coloring} of $(G,w)$ is a proper weighted coloring that uses as few colors as possible. An $O(n)$ time weighted coloring algorithm for holes was given in~\cite{HyperholeColoring}. Together with Lemmas~\ref{true-twin-alg} and~\ref{lemma-hyperhole-rec}, this yields the following result. 

\begin{lemma} \label{lemma-hyperhole-col} There exists an algorithm with the following specifications: 
\begin{itemize} 
\item Input: A graph $G$; 
\item Output: Either an optimal coloring of $G$, or the true statement that $G$ is not a hyperhole; 
\item Running time: $O(n+m)$. 
\end{itemize} 
\end{lemma} 
\begin{proof} 
Let $G$ be an input graph. We first find a partition $\mathcal{P}$ of $V(G)$ into true twin classes of $G$, and we form the quotient graph $G_\mathcal{P}$; by Lemma~\ref{true-twin-alg}, this can be done in $O(n+m)$ time. Clearly, all members of $\mathcal{P}$ are cliques of $G$. Next, we check in $O(n+m)$ time whether $G_{\mathcal{P}}$ is a hole, and if not, then we return the answer that $G$ is not a hyperhole (by Lemma~\ref{lemma-hyperhole-rec}, this is correct) and stop. From now on, we assume that $G_{\mathcal{P}}$ is a hole (and consequently, by Lemma~\ref{lemma-hyperhole-rec}, $G$ is a hyperhole). We define $w_\mathcal{P}:\mathcal{P} \rightarrow \mathbb{N}^+$ by setting $w_\mathcal{P}(X) = |X|$ for all $X \in \mathcal{P}$; this takes $O(n)$ time. Using the algorithm from~\cite{HyperholeColoring}, we then find an optimal weighted coloring $c$ of $(G_\mathcal{P},w_\mathcal{P})$; this takes a further $O(n)$ time. Using the weighted coloring $c$ of $(G_\mathcal{P},w_\mathcal{P})$, we easily obtain an optimal coloring of $G$: for each $X \in \mathcal{P}$, we assign to each vertex of $X$ one of the colors from the set $c(X)$, making sure that each vertex in $X$ gets a different color; this takes $O(n)$ time. Clearly, the algorithm is correct, and its total running time is $O(n+m)$. 
\end{proof}

\begin{lemma} \label{lemma-hyperhole-clique-stable} There exists an algorithm with the following specifications: 
\begin{itemize} 
\item Input: A weighted graph $(G,w)$; 
\item Output: Either a maximum weight clique $C$ and a maximum weight stable set $S$ of $(G,w)$, or the true statement that $G$ is not a hyperhole; 
\item Running time: $O(n+m)$. 
\end{itemize} 
\end{lemma} 
\begin{proof} 
Let $(G,w)$ be an input weighted graph. If $w$ assigns zero or negative weight to all vertices of $G$ (note that this can be checked in $O(n)$ time), then $\emptyset$ is both a maximum weight clique and a maximum weight stable set of $(G,w)$, and we are done. Otherwise, we first update $(G,w)$ by deleting all vertices of $G$ to which $w$ assigns zero or negative weight; this takes $O(n+m)$ time. Clearly, any induced subgraph of a hyperhole is either a hyperhole or a chordal graph. Using Lemma~\ref{lemma-chordal}, we now check whether $G$ is chordal, and if so, we find a maximum weight clique $C$ and a maximum weight stable set $S$ of $(G,w)$, and we return $C$ and $S$ and stop; this takes $O(n+m)$ time. Suppose now that the algorithm from Lemma~\ref{lemma-chordal} returned the answer that $G$ is not a chordal graph. We then find a partition $\mathcal{P}$ of $V(G)$ into true twin classes of $G$, and we form the quotient graph $G_\mathcal{P}$; by Lemma~\ref{true-twin-alg}, this can be done in $O(n+m)$ time. Clearly, all members of $\mathcal{P}$ are cliques of $G$. We check in $O(n+m)$ time whether $G_{\mathcal{P}}$ is a hole; if not, then we return the answer that $G$ is not a hyperhole (by Lemma~\ref{lemma-hyperhole-rec}, this is correct) and stop. So from now on, we assume that $G_{\mathcal{P}}$ is a hole. 

We find a maximum weight clique $C$ of $(G,w)$ as follows. We define $w_\mathcal{P}:\mathcal{P} \rightarrow \mathbb{R}$ by setting $w_\mathcal{P}(X) = \sum_{v \in X} w(v)$ for all $X \in \mathcal{P}$; finding $w_\mathcal{P}$ takes $O(n)$ time. We then find an edge $XY$ of the hole $G_\mathcal{P}$ for which the sum of weights (with respect to $w_\mathcal{P}$) of its endpoints is maximum; this takes $O(n)$ time. Set $C = X \cup Y$. Clearly, $C$ is a maximum weight clique of $(G,w)$. 

We find a maximum weight stable set $S$ of $(G,w)$ as follows. For each $X \in \mathcal{P}$, we find a vertex $v_X \in X$ such that $w(v_X) = \max\{w(v) \mid v \in X\}$; finding the family $\{v_X\}_{X \in \mathcal{P}}$ takes $O(n)$ time. We then form the graph $H = G[\{v_X \mid X \in \mathcal{P}\}]$ in $O(n+m)$ time. Since $G$ is a hyperhole, we see that $H$ is a hole. Clearly, $\alpha(G,w) = \alpha(H,w)$, and furthermore, any maximum weight stable set of $(H,w)$ is a maximum weight stable set of $(G,w)$. 

We find a maximum weight stable set of $(H,w)$ as follows. Let $x$ be any vertex of $H$, and let $y$ and $z$ be the two neighbors of $x$ in $H$. We form induced subgraphs $H \setminus x$ and $H \setminus \{x,y,z\}$ of $H$ in $O(n)$ time, and using the $O(n)$ time algorithm from~\cite{MaxWeightStableSetInTree}, we find a maximum weight stable set $S_1$ of the weighted path $(H \setminus x,w)$, and a maximum weight stable set $S_2$ of the weighted path $(H \setminus \{x,y,z\},w)$. (Note that we can also find $S_1$ and $S_2$ using the algorithm from Lemma~\ref{lemma-chordal}.) Clearly, $\{x\} \cup S_2$ is a stable set of $H$. If $w(S_1) \geq w(\{x\} \cup S_2)$, then we set $S = S_1$, and otherwise, we set $S = \{x\} \cup S_2$. Clearly, $S$ is a maximum weight stable set of $(H,w)$, and therefore of $(G,w)$ as well. 

The algorithm now returns the clique $C$ and the stable set $S$ and stops. It is clear that the algorithm is correct, and that its running time is $O(n+m)$. 
\end{proof}

\subsection{Class $\mathcal{G}_{\text{UT}}$} 

In this subsection, we give a polynomial time recognition algorithm for the class $\mathcal{G}_{\text{UT}}$, and we prove that the maximum clique problem is NP-hard for this class. The complexity of the optimal coloring and maximum stable set problems is still open. 

\begin{theorem} The maximum clique problem is NP-hard for the class of (long hole, $K_{2,3}$, $\overline{C_6}$)-free graphs. Consequently, the maximum clique problem is NP-hard for the class $\mathcal{G}_{\text{UT}}$. 
\end{theorem} 
\begin{proof} 
Since every 3PC other than $K_{2,3}$ and $\overline{C_6}$ contains a long hole, as does every proper wheel, we see that every (long hole, $K_{2,3}$, $\overline{C_6}$)-free graph belongs to $\mathcal{G}_{\text{UT}}$. Thus, the first statement implies the second. 

Let us now prove the first statement. First of all, it is easy to show that the maximum stable set problem is NP-hard for the class of graphs of girth at least nine. To see this, consider the operation of subdividing every edge of a graph $G$ twice (i.e.\ the operation of replacing each edge by an induced three-edge path); this yields a graph $G'$ of girth at least nine. As observed in~\cite{Pol74}, $\alpha(G') = \alpha(G)+|E(G)|$, and so computing the stability number of a graph of girth at least nine is as hard as computing it in a general graph. Now, note that if $G$ is a graph of girth at least nine, then $\overline{G}$ is (long hole, $K_{2,3}$, $\overline{C_6}$)-free. Therefore, if we could compute the clique number of a (long hole, $K_{2,3}$, $\overline{C_6}$)-free graph in polynomial time, then we could also compute the stability number of a graph of girth at least nine in polynomial time. It follows that the problem of computing the clique number of a (long hole, $K_{2,3}$, $\overline{C_6}$)-free graph is NP-hard. 
\end{proof} 

We now turn to the recognition problem for the class $\mathcal{G}_{\text{UT}}$. We begin with a corollary of Theorem~\ref{decomp-thm-GUT}, which is more convenient than Theorem~\ref{decomp-thm-GUT} itself for algorithmic purposes. 

\begin{lemma} \label{GUT-decomp-cor} Let $G$ be a graph, let $(T_G,r)$ be a clique-cutset decomposition tree of $G$, and let $\{G^u\}_{u \in V(T_G)}$ be the associated family of induced subgraphs of $G$. Then the following are equivalent: 
\begin{itemize} 
\item[(a)] $G \in \mathcal{G}_{\text{UT}}$; 
\item[(b)] $G$ is $(K_{2,3},\overline{C_6},W_5^4)$-free, and furthermore, for all $u \in \mathcal{L}(T_G,r)$, and all anticomponents $H$ of $G^u$, either $H$ is a long ring, or $H$ contains no long holes, or $\alpha(H) \leq 2$. 
\end{itemize} 
\end{lemma} 
\begin{proof}
It is clear that every graph in $\mathcal{G}_{\text{UT}}$ is $(K_{2,3},\overline{C_6},W_5^4)$-free. The fact that (a) implies (b) now follows immediately from Theorem~\ref{decomp-thm-GUT}. 

Suppose now that $G$ satisfies (b); we must show that $G$ satisfies (a), that is, that $G$ is (3PC, proper wheel)-free. Clearly, no 3PC, and no proper wheel admits a clique-cutset, and so by Lemma~\ref{lemma-no-clique-cut-ind-sg}, it suffices to show that each graph in $\{G^u\}_{u \in \mathcal{L}(T_G,r)}$ is (3PC, proper wheel)-free. Fix $u \in \mathcal{L}(T_G,r)$. Clearly, every 3PC other than $K_{2,3}$ is anticonnected, as is every proper wheel; since $G$ (and therefore, $G^u$ as well) is $K_{2,3}$-free, it now suffices to show that every anticomponent of $G^u$ is (3PC, proper wheel)-free. Let $H$ be an anticomponent of $G^u$; by hypothesis, $H$ is $(K_{2,3},\overline{C_6},W_5^4)$-free, and furthermore, either $H$ is a long ring, or $H$ contains no long holes, or $\alpha(H) \leq 2$. If $H$ is a long ring, then Lemma~\ref{ring-in-GT} implies that $H$ is (3PC, proper wheel)-free. So assume that $H$ either contains no long holes or satisfies $\alpha(H) \leq 2$. Clearly, every 3PC or proper wheel other than $K_{2,3}$ and $\overline{C_6}$ contains a long hole; furthermore, every 3PC or proper wheel other than $K_{2,3},\overline{C_6},W_5^4$ contains a stable set of size three. Since $H$ is $(K_{2,3},\overline{C_6},W_5^4)$-free, it follows that $H$ is (3PC, proper wheel)-free, and we are done. 
\end{proof} 

It can be determined in $O(n+m^2)$ time whether a graph contains a long hole~\cite{LongHoleDetection}. In view of this, and of Lemma~\ref{GUT-decomp-cor}, the problem of recognizing graphs in $\mathcal{G}_{\text{UT}}$ essentially reduces to the problem of recognizing long rings. 

\begin{lemma} \label{lemma-detect-ring} There exists an algorithm with the following specifications: 
\begin{itemize} 
\item Input: A graph $G$; 
\item Output: Either the true statement that $G$ is a ring, together with the length and good partition of the ring, or the true statement that $G$ is not a ring; 
\item Running time: $O(n^2)$. 
\end{itemize} 
\end{lemma} 
\begin{proof} 
\textbf{Step~0.} We first check in $O(n+m)$ time whether $G$ is connected; if not, then the algorithm returns the answer that $G$ is not a ring and stops. From now on, we assume that $G$ is connected. Next, we check in $O(n+m)$ time whether $G$ is chordal (we use Lemma~\ref{lemma-chordal}); if so, then the algorithm returns the answer that $G$ is not a ring and stops (this is correct because every ring contains a hole). From now on, we assume that $G$ is not chordal, and in particular, that $G$ is not complete, and we go to Step~1. 

\textbf{Step~1.} For each vertex $v \in V(G)$, we compute $d_G(v)$, and we find a vertex $x \in V(G)$ such that $d_G(x) = \Delta(G)$; this takes $O(n+m)$ time. Next, we let $X_1$ be the set of all vertices $y$ of $G$ such that $N_G[y] \subseteq N_G[x]$; computing $X_1$ takes $O(n^2)$ time. Set $n_1 = |X_1|$. We order $X_1$ as $X_1 = \{u_1^1,\dots,u_{n_1}^1\}$ so that $d_G(u_{n_1}^1) \leq \dots \leq d_G(u_1^1)$; this takes $O(n_1^2)$ time. Next, we check in $O(n_1n)$ time whether $N_G[u_{n_1}^1] \subseteq \dots \subseteq N_G[u_1^1]$; if not, then the algorithm returns the statement that $G$ is not a ring and stops. So assume that the algorithm found that $N_G[u_{n_1}^1] \subseteq \dots \subseteq N_G[u_1^1]$. (Note that this implies that $X_1$ is a clique. Since $G$ is not a complete graph, it follows that $X_1 \subsetneqq V(G)$. Since $G$ is connected, and since $N_G[u_{n_1}^1] \subseteq \dots \subseteq N_G[u_1^1]$, we see that $u_1^1$ has a neighbor in $V(G) \setminus X_1$.) Next, we check in $O(n^2)$ time whether $G \setminus X_1$ is chordal (we use Lemma~\ref{lemma-chordal}); if not, then the algorithm returns the statement that $G$ is not a ring and stops (this is correct by Lemma~\ref{ring-in-GT}). So assume that $G \setminus X_1$ is indeed chordal. Let $X_2$ be the vertex set of a component of $G[N_G(u_1^1) \setminus X_1]$; clearly, $X_2$ can be found in $O(n^2)$ time. Set $n_2 = |X_2|$. We order $X_2$ as $X_2 = \{u_1^2,\dots,u_{n_2}^2\}$ so that $d_G(u_{n_2}^2) \leq \dots \leq d_G(u_1^2)$, and then we check whether $N_G[u_{n_2}^2] \subseteq \dots \subseteq N_G[u_1^2]$; this takes $O(n_2n)$ time. If it is not the case that $N_G[u_{n_2}^2] \subseteq \dots \subseteq N_G[u_1^2]$, then the algorithm returns the answer that $G$ is not a ring and stops. So assume that $N_G[u_{n_2}^2] \subseteq \dots \subseteq N_G[u_1^2]$. We now set $k = 2$, and we go to Step~2. 

\textbf{Step~2.} Having constructed ordered sets $X_1 = \{u_1^1,\dots,u_{n_1}^1\},X_2 = \{u_1^2,\dots,u_{n_2}^2\},\dots,X_k = \{u_1^k,\dots,u_{n_k}^k\}$, we proceed as follows. We compute the set $X_{k+1} = N_G(u_1^k) \setminus (X_1 \cup \dots \cup X_k)$; this takes $O(n)$ time. Set $n_{k+1} = |X_{k+1}|$. If $n_{k+1} = 0$, then we go to Step~3. So assume that $n_{k+1} \geq 1$. In this case, we order $X_{k+1}$ as $X_{k+1} = \{u_1^{k+1},\dots,u_{n_{k+1}}^{k+1}\}$ so that $d_G(u_{n_{k+1}}^{k+1}) \leq \dots \leq d_G(u_1^{k+1})$, and then we check whether $N_G[u_{n_{k+1}}^{k+1}] \subseteq \dots \subseteq N_G[u_1^{k+1}]$; this takes $O(n_{k+1}n)$ time. If it is not the case that $N_G[u_{n_{k+1}}^{k+1}] \subseteq \dots \subseteq N_G[u_1^{k+1}]$, then the algorithm returns the answer that $G$ is not a ring and stops. Otherwise, we update $k := k+1$, and we go back to Step~2. 

\textbf{Step~3.} If $k \leq 3$, or if $X_1 \cup \dots \cup X_k \subsetneqq V(G)$ (this can be checked in $O(n)$ time), then the algorithm returns the answer that $G$ is not a ring and stops. So assume that $k \geq 4$ and $V(G) = X_1 \cup \dots \cup X_k$. We check whether $u_1^1,u_1^2,\dots,u_1^k,u_1^1$ is a hole in $G$ (this takes $O(n^2)$ time), and if so, the algorithm returns the statement that $G$ is a ring of length $k$, together with the good partition $(X_1,\dots,X_k)$ of the ring $G$; otherwise, the algorithm returns the answer that $G$ is not a ring. 

Clearly, the algorithm is correct. The running time of the algorithm is $O(n^2+\sum_{i=1}^k n_in)$; since $\sum_{i=1}^k n_i \leq n$, it follows that the running time of the algorithm is $O(n^2)$. 
\end{proof}

We are now ready to give a recognition algorithm for the class $\mathcal{G}_{\text{UT}}$. 

\begin{theorem} There exists an algorithm with the following specifications: 
\begin{itemize} 
\item Input: A graph $G$; 
\item Output: Either the true statement that $G \in \mathcal{G}_{\text{UT}}$, or the true statement that $G \notin \mathcal{G}_{\text{UT}}$; 
\item Running time: $O(n^6)$. 
\end{itemize} 
\end{theorem} 
\begin{proof}
We test for (b) from Lemma~\ref{GUT-decomp-cor}. We first check in $O(n^6)$ time whether $G$ is $(K_{2,3},\overline{C_6},W_5^4)$-free; if not, then the algorithm returns the answer that $G \notin \mathcal{G}_{\text{UT}}$ and stops. So assume that $G$ is $(K_{2,3},\overline{C_6},W_5^4)$-free. We compute a clique-cutset decomposition tree $(T_G,r)$ of $G$, together with the associated family $\{G^u\}_{u \in V(T_G)}$ of induced subgraphs of $G$; this takes $O(n^2+nm)$ time, which is $O(n^3)$ time. Fix $u \in \mathcal{L}(T_G,r)$. We first compute the anticomponents $H_1^u,\dots,H_t^u$ of $G^u$ in $O(n^2)$ time (this can be done by first computing $\overline{G^u}$, then, using BFS, computing the components of $\overline{G^u}$, and finally computing the complements of those components). For each $i \in \{1,\dots,t\}$, set $n_i^u = |V(H_i^u)|$; clearly, $\sum_{i=1}^t n_i^u = |V(G^u)| \leq n$. Now, for each $i \in \{1,\dots,t\}$, we determine in $O((n_i^u)^4)$ time whether at least one of the following holds: 
\begin{itemize} 
\item[(i)] $H_i^u$ is a long ring (we use the $O(n^2)$ time algorithm from Lemma~\ref{lemma-detect-ring}); 
\item[(ii)] $H_i^u$ contains no long holes (we use the $O(n+m^2)$ time algorithm from~\cite{LongHoleDetection}); 
\item[(iii)] $\alpha(H_i^u) \leq 2$. 
\end{itemize} 
Checking this for all anticomponents of $G^u$ takes $O(\sum_{i=1}^t (n_i^u)^4)$ time, which is $O(n^4)$ time; since $|\mathcal{L}(T_G,r)| \leq n$, performing this computation for all graphs in the family $\{G^u\}_{u \in \mathcal{L}(T_G,r)}$ takes $O(n^5)$ time. Now, if every anticomponent of every graph in the family $\{G^u\}_{\mathcal{L}(T_G,r)}$ satisfies (i), (ii), or (iii), then the algorithm returns the answer that $G \in \mathcal{G}_{\text{UT}}$ and stops; otherwise, the algorithm returns the answer that $G \notin \mathcal{G}_{\text{UT}}$ and stops. The correctness of the algorithm follows from Lemma~\ref{GUT-decomp-cor}, and clearly, the running time of the algorithm is $O(n^6)$. 
\end{proof} 

\subsection{Class $\mathcal{G}_{\text{U}}$} 

In this subsection, we give polynomial time algorithms that solve the recognition, optimal coloring, maximum weight clique, and maximum weight stable set problems for the class $\mathcal{G}_{\text{U}}$. 

Let $\mathcal{B}_{\text{U}}^{\text{h}}$ be the class of all induced subgraphs of graphs in $\mathcal{B}_{\text{U}}$. Clearly, $\mathcal{B}_{\text{U}} \subseteq \mathcal{B}_{\text{U}}^{\text{h}}$, and $\mathcal{B}_{\text{U}}^{\text{h}}$ is hereditary. Furthermore, a graph $G$ belongs to $\mathcal{B}_{\text{U}}^{\text{h}}$ if and only if one of the following holds: 
\begin{itemize} 
\item every nontrivial anticomponent of $G$ is isomorphic to $\overline{K_2}$; 
\item $G$ has exactly one nontrivial anticomponent, and this anticomponent is a long hole; 
\item $G$ has exactly one nontrivial anticomponent, and this anticomponent has at least three vertices and is a disjoint union of paths. 
\end{itemize} 

\begin{lemma} \label{decomp-lemma-GU-hereditray} The class $\mathcal{B}_{\text{U}}^{\text{h}}$ is hereditary, and $\mathcal{B}_{\text{U}}^{\text{h}} \subseteq \mathcal{G}_{\text{U}}$. Furthermore, every graph in $\mathcal{G}_{\text{U}}$ either belongs to $\mathcal{B}_{\text{U}}^{\text{h}}$ or admits a clique-cutset. 
\end{lemma} 
\begin{proof} 
The fact that $\mathcal{B}_{\text{U}}^{\text{h}}$ is hereditary follows immediately from the definition of $\mathcal{B}_{\text{U}}^{\text{h}}$. Next, by Lemma~\ref{BB-in-GG}, we have that $\mathcal{B}_{\text{U}} \subseteq \mathcal{G}_{\text{U}}$. By definition, $\mathcal{B}_{\text{U}}^{\text{h}}$ is the class of all induced subgraphs of graphs in $\mathcal{B}_{\text{U}}$; since $\mathcal{G}_{\text{U}}$ is hereditary, it follows that $\mathcal{B}_{\text{U}}^{\text{h}} \subseteq \mathcal{G}_{\text{U}}$. 

It remains to show that every graph in $\mathcal{G}_{\text{U}}$ either belongs to $\mathcal{B}_{\text{U}}^{\text{h}}$ or admits a clique-cutset. But this follows immediately from Theorem~\ref{decomp-thm-GU}, and from the fact that $\mathcal{B}_{\text{U}} \subseteq \mathcal{B}_{\text{U}}^{\text{h}}$. 
\end{proof} 

\begin{lemma} \label{GU-decomp-cor} Let $G$ be a graph, let $(T_G,r)$ be a clique-cutset decomposition tree of $G$, and let $\{G^u\}_{u \in V(T_G)}$ be the associated family of induced subgraphs of $G$. Then $G \in \mathcal{G}_{\text{U}}$ if and only if all graphs in the family $\{G^u\}_{u \in \mathcal{L}(T_G,r)}$ belong to $\mathcal{B}_{\text{U}}^{\text{h}}$. 
\end{lemma} 
\begin{proof} 
The ``only if'' part follows immediately from Lemma~\ref{decomp-lemma-GU-hereditray} (and in particular, the fact that every graph in $\mathcal{G}_{\text{U}}$ either belongs to $\mathcal{B}_{\text{U}}^{\text{h}}$ or admits a clique-cutset). The ``if'' part follows from Lemma~\ref{lemma-no-clique-cut-ind-sg}, from the fact that (by Lemma~\ref{decomp-lemma-GU-hereditray}) $\mathcal{B}_{\text{U}}^{\text{h}} \subseteq \mathcal{G}_{\text{U}}$, and from the fact that no 3PC and no wheel admits a clique-cutset. 
\end{proof}

\begin{lemma} \label{lemma-BUh-rec} There exists an algorithm with the following specifications: 
\begin{itemize} 
\item Input: A graph $G$; 
\item Output: Exactly one of the following: 
\begin{itemize} 
\item The true statement that $G \in \mathcal{B}_{\text{U}}^{\text{h}}$, together with the anticomponents $G_1,\dots,G_t$ of $G$, and for each $i \in \{1,\dots,t\}$, the correct information whether 
\begin{itemize} 
\item[(i)] $G_i$ is isomorphic to $K_1$, or 
\item[(ii)] $G_i$ is isomorphic to $\overline{K_2}$, or 
\item[(iii)] $G_i$ is an odd long hole, or 
\item[(iv)] $G_i$ is an even long hole, or 
\item[(v)] $G_i$ has at least three vertices and is a disjoint union of paths. 
\end{itemize} 
\item The true statement that $G \notin \mathcal{B}_{\text{U}}^{\text{h}}$; 
\end{itemize} 
\item Running time: $O(n+m)$. 
\end{itemize} 
\end{lemma} 
\begin{proof} 
We first compute the degree of all the vertices of $G$; this takes $O(n+m)$ time. 

Suppose first that we have $d_G(u) \geq n-2$ for all $u \in V(G)$; note that this can be checked in $O(n)$ time. In this case, we have that $m \geq \frac{1}{2}n(n-2)$. We now compute the anticomponents $G_1,\dots,G_t$ of $G$; this takes $O(n^2)$ time, which is $O(n+m)$ time (because $m \geq \frac{1}{2}n(n-2)$). We now have that for each $i \in \{1,\dots,t\}$, $G_i$ is isomorphic to $K_1$ or $\overline{K_2}$, that is, $G_i$ satisfies (i) or (ii); clearly, we can determine in $O(n)$ time which $G_i$'s satisfy (i) and which satisfy (ii). 

Suppose now that at least one vertex of $G$ is of degree at most $n-3$. We first form the set $U$ of all vertices of degree $n-1$ in $G$, and we set $V = V(G) \setminus U$; clearly, computing $U$ and $V$ takes $O(n)$ time, and futhermore, for all $u \in U$, $G[u]$ is a trivial anticomponent of $G$. Note that the vertex of $G$ that is of degree at most $n-3$ must belong to $V$, and furthermore, all nonneighbors of this vertex belong to $V$; thus, $|V| \geq 3$, and it follows that $G[V]$ satisfies neither (i) nor (ii). Now, we form the graph $G[V]$ and check whether $G[V]$ satisfies (iii), (iv), or (v); this takes $O(n+m)$ time. If $G[V]$ satisfies none of (iii), (iv), and (v), then the algorithm returns the answer that $G \notin \mathcal{B}_{\text{U}}^{\text{h}}$ and stops. Suppose now that $G[V]$ satisfies (iii), (iv), or (v). Then $G[V]$ is anticonnected unless it is isomorphic to $P_3$. But if $G[V]$ is isomorphic to $P_3$, then the (unique) interior vertex of the path $G[V]$ is of degree $n-1$ in $G$, and consequently, it belongs to $U$, a contradiction. Thus, $G[V]$ is indeed anticonnected. The algorithm now returns the answer that $G \in \mathcal{B}_{\text{U}}^{\text{h}}$, together with the anticomponents $G[u_1],\dots,G[u_{\ell}],G[V]$, where $U = \{u_1,\dots,u_{\ell}\}$, and furthermore, the algorithm returns the answer that $G[u_1],\dots,G[u_{\ell}]$ satisfy (i), and that $G[V]$ satisfies (iii), (iv), or (v), as determined by the algorithm. (If $U = \emptyset$, then the algorithm simply returns the anticomponent $G[V] = G$, together with the information that $G[V] = G$ satisfies (iii), (iv), or (v), as determined by the algorithm.) 

Clearly, the algorithm is correct, and its running time is $O(n+m)$. 
\end{proof} 

\begin{theorem} There exists an algorithm with the following specifications: 
\begin{itemize} 
\item Input: A graph $G$; 
\item Output: Either the true statement that $G \in \mathcal{G}_{\text{U}}$, or the true statement that $G \notin \mathcal{G}_{\text{U}}$; 
\item Running time: $O(n^2+nm)$. 
\end{itemize} 
\end{theorem} 
\begin{proof} 
First, we compute a clique-cutset decomposition tree $(T_G,r)$ of $G$, together with the associated family $\{G^u\}_{u \in V(T_G)}$ of induced subgraphs of $G$; this takes $O(n^2+nm)$ time. Then, using the $O(n+m)$ time algorithm from Lemma~\ref{lemma-BUh-rec}, we check whether all graphs in the family $\{G^u\}_{u \in \mathcal{L}(T_G,r)}$ belong to $\mathcal{B}_{\text{U}}^{\text{h}}$; since $|\mathcal{L}(T_G,t)| \leq n$, checking this for the entire family $\{G^u\}_{u \in \mathcal{L}(T_G,r)}$ takes $O(n^2+nm)$ time. If all graphs in the family $\{G^u\}_{u \in \mathcal{L}(T_G,r)}$ belong to $\mathcal{B}_{\text{U}}^{\text{h}}$, then the algorithm return the answer that $G \in \mathcal{G}_{\text{U}}$, and otherwise, the algorithm returns the answer that $G \notin \mathcal{G}_{\text{U}}$. The correctness of the algorithm follows from Lemma~\ref{GU-decomp-cor}. Clearly, the running time of the algorithm is $O(n^2+nm)$. 
\end{proof} 

\begin{theorem} There exists an algorithm with the following specifications: 
\begin{itemize} 
\item Input: A graph $G$; 
\item Output: Either an optimal coloring of $G$, or the true statement that $G \notin \mathcal{G}_{\text{U}}$; 
\item Running time: $O(n^2+nm)$. 
\end{itemize} 
\end{theorem} 
\begin{proof} 
In view of Lemmas~\ref{clique-cutset decomposition tree-coloring} and~\ref{decomp-lemma-GU-hereditray}, it suffices to show that there exists an algorithm with the following specifications: 
\begin{itemize} 
\item Input: A graph $G$; 
\item Output: Either an optimal coloring of $G$, or the true statement that $G \notin \mathcal{B}_{\text{U}}^{\text{h}}$; 
\item Running time: $O(n+m)$. 
\end{itemize} 
In view of Lemmas~\ref{lemma-chordal},~\ref{lemma-hyperhole-col}, and~\ref{lemma-BUh-rec}, it is easy to see that such an algorithm exists. 
\end{proof} 

\begin{theorem} There exists an algorithm with the following specifications: 
\begin{itemize} 
\item Input: A weighted graph $(G,w)$; 
\item Output: Either a maximum weight clique $C$ and a maximum weight stable set $S$ of $(G,w)$, or the true statement that $G \notin \mathcal{G}_{\text{U}}$; 
\item Running time: $O(n^2+nm)$. 
\end{itemize} 
\end{theorem} 
\begin{proof} 
In view of Lemmas~\ref{clique-cutset decomposition tree-clique},~\ref{clique-cutset decomposition tree-stable}, and~\ref{decomp-lemma-GU-hereditray}, it suffices to show that there exists an algorithm with the following specifications: 
\begin{itemize} 
\item Input: A weighted graph $(G,w)$; 
\item Output: Either a maximum weight clique $C$ and a maximum weight stable set $S$ of $(G,w)$, or the true statement that $G \notin \mathcal{B}_{\text{U}}^{\text{h}}$; 
\item Running time: $O(n+m)$. 
\end{itemize} 
In view of Lemmas~\ref{lemma-chordal},~\ref{lemma-hyperhole-clique-stable}, and~\ref{lemma-BUh-rec}, it is easy to see that such an algorithm exists. 
\end{proof}

\subsection{Class $\mathcal{G}_{\text{T}}$} 

In this subsection, we give polynomial-time algorithms that solve the recognition, maximum weight clique, and maximum weight stable set problems for the class $\mathcal{G}_{\text{T}}$. We remark that we do not know whether graphs in $\mathcal{G}_{\text{T}}$ can be optimally colored in polynomial time; this is because we do not know whether rings can be optimally colored in polynomial time. We begin with a corollary of Theorem~\ref{decomp-thm-GT}. 

\begin{lemma} \label{GT-decomp-cor} Let $G$ be a graph, let $(T_G,r)$ be a clique-cutset decomposition tree of $G$, and let $\{G^u\}_{u \in V(T_G)}$ be the associated family of induced subgraphs of $G$. For all $u \in V(T_G)$, let $\mathcal{P}_u$ be the partition of $V(G^u)$ into true twin classes of $G^u$. Then the following are equivalent: 
\begin{itemize} 
\item[(a)] $G \in \mathcal{G}_{\text{T}}$; 
\item[(b)] for all $u \in \mathcal{L}(T_G,r)$, the quotient graph $G^u_{\mathcal{P}_u}$ is a ring, a one-vertex graph, or a 7-antihole. 
\end{itemize} 
\end{lemma} 
\begin{proof} 
Since no 3PC and no wheel admits a clique-cutset, Lemma~\ref{BB-in-GG} (and in particular, the fact that $\mathcal{B}_{\text{T}} \subseteq \mathcal{G}_{\text{T}}$), Theorem~\ref{decomp-thm-GT}, and Lemma~\ref{lemma-no-clique-cut-ind-sg} imply that $G \in \mathcal{G}_{\text{T}}$ if and only if all graphs from the family $\{G^u\}_{u \in \mathcal{L}(T_G,r)}$ belong to $\mathcal{B}_{\text{T}}$. On the other hand, it follows from the definition of $\mathcal{B}_{\text{T}}$ that a graph $H$ belongs to $\mathcal{B}_{\text{T}}$ if and only if the quotient graph $H_\mathcal{P}$ (where $\mathcal{P}$ is the partition of $V(H)$ into true twin classes) is either a ring, a one-vertex graph, or a 7-antihole. The result is now immediate. 
\end{proof} 

\begin{theorem} There exists an algorithm with the following specifications: 
\begin{itemize} 
\item Input: A graph $G$; 
\item Output: Either the true statement that $G \in \mathcal{G}_{\text{T}}$, or the true statement that $G \notin \mathcal{G}_{\text{T}}$; 
\item Running time: $O(n^3)$. 
\end{itemize} 
\end{theorem} 
\begin{proof} 
We use Lemma~\ref{GT-decomp-cor}. First, we compute a clique-cutset decomposition tree $(T_G,r)$ of $G$, together with the associated family $\{G^u\}_{u \in V(T_G)}$ of induced subgraphs of $G$; this takes $O(n^2+nm)$ time. For each $u \in \mathcal{L}(T_G,r)$, we compute the partition $\mathcal{P}_u$ of $V(G^u)$ into true twin classes of $G^u$, and we compute the quotient graph $G^u_{\mathcal{P}_u}$; Lemma~\ref{true-twin-alg} and the fact that $|\mathcal{L}(T_G,r)| \leq n$ imply that the family $\{G^u_{\mathcal{P}_u}\}_{u \in \mathcal{L}(T_G,r)}$ can be computed in $O(n^2+nm)$ time. By Lemma~\ref{lemma-detect-ring}, rings can be recognized in $O(n^2)$ time, and clearly, one can check in $O(1)$ time whether a graph is trivial (i.e.\ whether it has just one vertex) or is a 7-antihole. Since $|\mathcal{L}(T_G,r)| \leq n$, it follows that it can be checked in $O(n^3)$ time whether the family $\{G^u_{\mathcal{P}_u}\}_{u \in \mathcal{L}(T_G,r)}$ satisfies condition (b) of Lemma~\ref{GT-decomp-cor}; if so, then the algorithm returns the answer that $G \in \mathcal{G}_{\text{T}}$, and otherwise, it returns the answer that $G \notin \mathcal{G}_{\text{T}}$. Clearly, the algorithm is correct, and its running time is $O(n^3)$. 
\end{proof} 

\begin{lemma} \label{no-univ-wheel-neighborhood-chordal} Let $G$ be a graph. Then the following are equivalent: 
\begin{itemize} 
\item[(a)] $G$ contains no universal wheels; 
\item[(b)] for all $u \in V(G)$, $G[N_G[u]]$ is chordal. 
\end{itemize} 
\end{lemma} 
\begin{proof} 
Suppose first that (a) holds. Fix $u \in V(G)$. First, note that if $G[N_G(u)]$ contains a hole $H$, then $(H,u)$ is a universal wheel in $G$, contrary to (a). Thus, $G[N_G(u)]$ is chordal. Since $u$ is complete to $N_G(u)$, we deduce that $G[N_G[u]]$ is also chordal. Thus, (b) holds. 

Suppose now that (b) holds. Suppose that $G$ contains a universal wheel, say $(H,u)$. Then $H$ is a hole in $G[N_G[u]]$, contrary to the fact that $G[N_G[u]]$ is chordal. 
\end{proof} 

\begin{theorem} There exists an algorithm with the following specifications: 
\begin{itemize} 
\item Input: A weighted graph $(G,w)$; 
\item Output: Either a maximum weight clique $C$ of $(G,w)$, or the true statement that $G$ contains a universal wheel (and therefore $G \notin \mathcal{G}_{\text{T}}$); 
\item Running time: $O(n^2+nm)$. 
\end{itemize} 
\end{theorem} 
\begin{proof} 
For each $u \in V(G)$, we form the graph $G_u = G[N_G[u]]$, we check whether $G_u$ is chordal, and if so, we compute a maximum weight clique $C_u$ of $G_u$; in view of Lemma~\ref{lemma-chordal}, for each $u \in V(G)$ individually, we can perform these computations in $O(n+m)$ time, and so for all $u \in V(G)$ together, we can perform them in $O(n^2+nm)$ time. Now, if for some $u \in V(G)$, we determined that $G_u$ is not chordal, then the algorithm returns the answer that $G$ contains a universal wheel (this is correct by Lemma~\ref{no-univ-wheel-neighborhood-chordal}) and stops. So assume that the algorithm computed a maximum weight clique $C_u$ for each $G_u$. Among all cliques in the family $\{C_u\}_{u \in V(G)}$, the algorithm finds one of maximum weight, and it returns that clique and stops. It is clear that the algorithm is correct, and that its running time is $O(n^2+nm)$. 
\end{proof} 

\begin{lemma} \label{GT-nonneighborhood-chordal} Let $G \in \mathcal{G}_{\text{T}}$. Then at least one of the following holds: 
\begin{itemize} 
\item for all $u \in V(G)$, $G \setminus N_G(u)$ is chordal; 
\item $G$ admits a clique-cutset. 
\end{itemize} 
\end{lemma} 
\begin{proof} 
Assume that $G$ does not admit a clique-cutset. Fix $u \in V(G)$; we must show that $G \setminus N_G(u)$ is chordal. By Theorem~\ref{decomp-thm-GT}, $G$ is either a ring, a complete graph, or a 7-hyperantihole. If $G$ is a ring, then the result follows from Lemma~\ref{ring-in-GT}, and if $G$ is a complete graph or a 7-hyperantihole, then the result is immediate. 
\end{proof} 

\begin{theorem} There exists an algorithm with the following specifications: 
\begin{itemize} 
\item Input: A weighted graph $(G,w)$; 
\item Output: Either a maximum weight stable set $S$ of $(G,w)$, or the true statement that $G \notin \mathcal{G}_{\text{T}}$; 
\item Running time: $O(n^3+n^2m)$. 
\end{itemize} 
\end{theorem} 
\begin{proof} 
Let $\mathcal{B}$ be the class of all graphs $G$ such that for every vertex $u \in V(G)$, we have that $G \setminus N_G(u)$ is chordal. Clearly, $\mathcal{B}$ is a hereditary class, and by Lemma~\ref{GT-nonneighborhood-chordal}, every graph in $\mathcal{G}_{\text{T}}$ either belongs to $\mathcal{B}$ or admits a clique-cutset. In view of Lemma~\ref{clique-cutset decomposition tree-stable}, it now suffices to show that there exists an algorithm with the following specifications: 
\begin{itemize} 
\item Input: A weighted graph $(G,w)$; 
\item Output: Either a maximum weight stable set $S$ of $(G,w)$, or the true statement that $G \notin \mathcal{B}$; 
\item Running time: $O(n^2+nm)$. 
\end{itemize} 
Let $(G,w)$ be an input weighted graph. For each $u \in V(G)$, we form the graph $G_u = G \setminus N_G(u)$, we check whether $G_u$ is chordal, and if so, we compute a maximum weight stable set $S_u$ of $(G_u,w)$; by Lemma~\ref{lemma-chordal}, for each $u \in V(G)$ individually, these computations can be performed in $O(n+m)$ time, and for all $u \in V(G)$ together, they can be performed in $O(n^2+nm)$ time. If the algorithm determined that for some $u \in V(G)$, $G_u$ is not chordal, then we return the answer that $G \notin \mathcal{B}$ and stop. So assume that for each $u \in V(G)$, the algorithm found a maximum weight stable set $S_u$ of $(G_u,w)$. Clearly, $\alpha(G,w) = \max\{w(S_u) \mid u \in V(G)\}$. We now find a vertex $x \in V(G)$ such that $w(S_x) = \max\{w(S_u) \mid u \in V(G)\}$; this takes $O(n^2)$ time. We return $S_x$ and stop. Clearly, the algorithm is correct, and its running time is $O(n^2+nm)$. 
\end{proof}

\subsection{Class $\mathcal{G}_{\text{UT}}^{\text{cap-free}}$} 

In this subsection, we show that the recognition, optimal coloring, maximum weight clique, and maximum weight stable set problems can be solved in polynomial time for the class $\mathcal{G}_{\text{UT}}^{\text{cap-free}}$. 

Let $\mathcal{B}^{\text{H}}_{\text{C}}$ be the class of all graphs $G$ such that every anticomponent of $G$ is either a long hyperhole or a chordal graph. 

\begin{lemma} \label{decomp-GUTcap-BHC} The class $\mathcal{B}^{\text{H}}_{\text{C}}$ is hereditary. Furthermore, every graph in $\mathcal{G}_{\text{UT}}^{\text{cap-free}}$ either belongs to $\mathcal{B}^{\text{H}}_{\text{C}}$ or admits a clique-cutset. 
\end{lemma} 
\begin{proof} 
Clearly, the class of chordal graphs is hereditary, and every induced subgraph of a long hyperhole is either a long hyperhole or a chordal graph; this implies that $\mathcal{B}^{\text{H}}_{\text{C}}$ is hereditary. Next, it is clear that $\mathcal{B}_{\text{UT}}^{\text{cap-free}} \subseteq \mathcal{B}^{\text{H}}_{\text{C}}$. This, together with Theorem~\ref{decomp-thm-GUTcap}, implies that every graph in $\mathcal{G}_{\text{UT}}^{\text{cap-free}}$ either belongs to $\mathcal{B}^{\text{H}}_{\text{C}}$ or admits a clique-cutset. 
\end{proof} 

\begin{lemma} \label{GUTcap-decomp-cor} Let $G$ be a graph, let $(T_G,r)$ be a clique-cutset decomposition tree of $G$, and let $\{G^u\}_{u \in V(T_G)}$ be the associated family of induced subgraphs of $G$. Then the following are equivalent: 
\begin{itemize} 
\item[(a)] $G \in \mathcal{G}_{\text{UT}}^{\text{cap-free}}$; 
\item[(b)] $G$ is ($K_{2,3}$, cap)-free, and all graphs in $\{G^u\}_{u \in \mathcal{L}(T_G,r)}$ belong to $\mathcal{B}^{\text{H}}_{\text{C}}$. 
\end{itemize} 
\end{lemma} 
\begin{proof} 
Clearly, every graph in $\mathcal{G}_{\text{UT}}^{\text{cap-free}}$ is ($K_{2,3}$, cap)-free. The fact that (a) implies (b) now follows from Lemma~\ref{decomp-GUTcap-BHC}. 

For the converse, we assume (b), and we prove (a). By (b), $G$ is cap-free, and so it suffices to show that $G$ is (3PC, proper wheel)-free. No 3PC and no proper wheel admits a clique-cutset, and so by Lemma~\ref{lemma-no-clique-cut-ind-sg}, it suffices to show that all graphs in the family $\{G^u\}_{u \in \mathcal{L}(T_G,r)}$ are (3PC, proper wheel)-free. Fix $u \in \mathcal{L}(T_G,r)$. Note that every 3PC other than $K_{2,3}$ is anticonnected, as is every proper wheel; since $G^u$ is $K_{2,3}$-free (because $G$ is), it now suffices to show that every anticomponent of $G^u$ is (3PC, proper wheel)-free. Let $H$ be an anticomponent of $G^u$. By (b), we have that $G^u \in \mathcal{B}_{\text{C}}^{\text{H}}$, and so by the definition of $\mathcal{B}_{\text{C}}^{\text{H}}$, $H$ is either a chordal graph or a long hyperhole. In the former case, it is clear that $H$ is (3PC, proper wheel)-free (this is because every 3PC and every wheel contains a hole, and by definition, chordal graphs contain no holes). So assume that $H$ is a hyperhole. Then $H$ is a ring, and so by Lemma~\ref{ring-in-GT}, $H$ is (3PC, proper wheel)-free. This proves (a). 
\end{proof} 

\begin{theorem} There exists an algorithm with the following specifications: 
\begin{itemize} 
\item Input: A graph $G$; 
\item Output: Either the true statement that $G \in \mathcal{G}_{\text{UT}}^{\text{cap-free}}$, or the true statement that $G \notin \mathcal{G}_{\text{UT}}^{\text{cap-free}}$; 
\item Running time: $O(n^5)$. 
\end{itemize} 
\end{theorem} 
\begin{proof} 
We test for (b) from Lemma~\ref{GUTcap-decomp-cor}. We first check in $O(n^5)$ time whether $G$ is ($K_{2,3}$, cap)-free (to test whether $G$ is $K_{2,3}$-free, we simply examine all five-tuples of vertices of $G$, and to check whether $G$ is cap-free, we use the $O(n^5)$ time algorithm from~\cite{CapEvenHoleFree}). If $G$ is not ($K_{2,3}$, cap)-free, then the algorithm returns the answer that $G \notin \mathcal{G}_{\text{UT}}^{\text{cap-free}}$ and stops. So assume that $G$ is ($K_{2,3}$, cap)-free. We now compute a clique-cutset decomposition tree $(T_G,r)$ of $G$, together with the associated family $\{G^u\}_{u \in V(T_G)}$ of induced subgraphs of $G$; this takes $O(n^2+nm)$ time. For each $u \in \mathcal{G}^u$, we proceed as follows. First, we compute the anticomponents $G_1,\dots,G_t$ of $G^u$ in $O(n^2)$ time; for each $i \in \{1,\dots,t\}$, set $n_i = |V(G_i)|$. For each $i \in \{1,\dots,t\}$, we test in $O(n_i^2)$ time whether $G_i$ is either a chordal graph or a long hyperhole (we use Lemmas~\ref{lemma-chordal} and~\ref{lemma-hyperhole-rec}); testing this for all anticomponents of $G^u$ together takes $O(\sum_{i=1}^t n_i^2)$ time, which is $O(n^2)$ time. Since $|\mathcal{L}(T_G,r)| \leq n$, performing this computation for all graphs in the family $\{G^u\}_{u \in \mathcal{L}(T_G,r)}$ takes $O(n^3)$ time. If for each $u \in \mathcal{L}(T_G,r)$, we determined that every anticomponent of $G^u$ is either a chordal graph or a long hyperhole, then (by the definition of $\mathcal{B}_{\text{C}}^{\text{H}}$) we have that every graph in the family $\{G^u\}_{u \in \mathcal{L}(T_G,r)}$ belongs to $\mathcal{B}_{\text{C}}^{\text{H}}$, and so by Lemma~\ref{GUTcap-decomp-cor}, we have that $G \in \mathcal{G}_{\text{UT}}^{\text{cap-free}}$, and we return this answer and stop. Otherwise, we return the answer that $G \notin \mathcal{G}_{\text{UT}}^{\text{cap-free}}$ and stop. Clearly, the algorithm is correct, and its running time is $O(n^5)$. 
\end{proof} 

\begin{theorem} There exists an algorithm with the following specifications: 
\begin{itemize}
\item Input: A graph $G$; 
\item Output: Either an optimal coloring of $G$, or the true statement that $G \notin \mathcal{G}_{\text{UT}}^{\text{cap-free}}$; 
\item Running time: $O(n^3)$. 
\end{itemize} 
\end{theorem} 
\begin{proof} 
In view of Lemmas~\ref{clique-cutset decomposition tree-coloring} and~\ref{decomp-GUTcap-BHC}, it suffices to show that there exists an algorithm with the following specifications: 
\begin{itemize}
\item Input: A graph $G$; 
\item Output: Either an optimal coloring of $G$, or the true statement that $G \notin \mathcal{B}^{\text{H}}_{\text{C}}$; 
\item Running time: $O(n^2)$. 
\end{itemize} 
Let $G$ be an input graph. We begin by computing the anticomponents $G_1,\dots,G_t$ of $G$ in $O(n^2)$ time. For each $i \in \{1,\dots,t\}$, we set $n_i = |V(G_i)|$, and we proceed as follows. We first check whether $G_i$ is chordal, and if so, we compute an optimal coloring $c_i$ of $G_i$; by Lemma~\ref{lemma-chordal}, this can be done in $O(n_i^2)$ time. If $G_i$ is not chordal, then we call the algorithm from Lemma~\ref{lemma-hyperhole-col}, and we obtain either an optimal coloring $c_i$ of $G_i$, or the true statement that $G_i$ is not a hyperhole; this takes $O(n_i^2)$ time. If for some $i \in \{1,\dots,t\}$, we determined that $G_i$ is neither a chordal graph nor a hyperhole, then the algorithm returns the answer that $G \notin \mathcal{B}^{\text{H}}_{\text{C}}$ and stops. So assume that for each $i \in \{1,\dots,t\}$, the algorithm found an optimal coloring $c_i$ of $G_i$. We then rename the colors used by the colorings $c_1,\dots,c_t$ so that the color sets used by these colorings are pairwise disjoint, and then we let $c$ be the union of the resulting $t$ colorings. The algorithm now returns the coloring $c$ and stops. Clearly, the algorithm is correct, and its running time is $O(n^2+\sum_{i=1}^t n_i^2)$, which is $O(n^2)$. 
\end{proof} 

\begin{theorem} \label{GU-GUTcap-clique-stable-alg} There exists an algorithm with the following specifications: 
\begin{itemize} 
\item Input: A weighted graph $(G,w)$; 
\item Output: Either a maximum weight clique $C$ and a maximum weight stable set $S$ of $(G,w)$, or the true statement that $G \notin \mathcal{G}_{\text{UT}}^{\text{cap-free}}$; 
\item Running time: $O(n^3)$. 
\end{itemize} 
\end{theorem} 
\begin{proof} 
In view of Lemmas~\ref{clique-cutset decomposition tree-clique},~\ref{clique-cutset decomposition tree-stable}, and~\ref{decomp-GUTcap-BHC}, it suffices to show that there exists an algorithm with the following specifications: 
\begin{itemize}
\item Input: A weighted graph $(G,w)$; 
\item Output: Either a maximum weight clique $C$ and a maximum weight stable set $S$ of $(G,w)$, or the true statement that $G \notin \mathcal{B}^{\text{H}}_{\text{C}}$; 
\item Running time: $O(n^2)$. 
\end{itemize} 
Let $(G,w)$ be an input weighted graph. We begin by computing the anticomponents $G_1,\dots,G_t$ of $G$ in $O(n^2)$ time. For each $i \in \{1,\dots,t\}$, we set $n_i = |V(G_i)|$. For each $i \in \{1,\dots,t\}$, we proceed as follows. We first check whether $G_i$ is chordal, and if so, we find a maximum weight clique $C_i$ and a maximum weight stable set $S_i$ of $(G_i,w)$; by Lemma~\ref{lemma-chordal}, this can be done in $O(n_i^2)$ time. If $G_i$ is not chordal, then we call the algorithm from Lemma~\ref{lemma-hyperhole-clique-stable}, and we obtain either a maximum weight clique $C_i$ and a maximum weight stable set $S_i$ of $(G_i,w)$, or the true statement that $G_i$ is not a hyperhole; this takes $O(n_i^2)$ time. If for some $i \in \{1,\dots,t\}$, we determined that $G_i$ is neither a chordal graph nor a hyperhole, then the algorithm returns the answer that $G \notin \mathcal{B}^{\text{H}}_{\text{C}}$ and stops. So assume that for each $i \in \{1,\dots,t\}$, the algorithm found a maximum weight clique $C_i$ and a maximum weight stable set $S_i$ of $(G_i,w)$. We then form the clique $C = C_1 \cup \dots \cup C_t$, and we find an index $j \in \{1,\dots,t\}$ such that $w(S_j) = \max\{w(S_i) \mid 1 \leq i \leq t\}$; clearly, this can be done in $O(n^2)$ time. The algorithm now returns the clique $C$ and the stable set $S_j$ and stops. Clearly, the algorithm is correct, and its running time is $O(n^2+\sum_{i=1}^t n_i^2)$, which is $O(n^2)$. 
\end{proof}

\section*{Acknowledgments} 
We would like to thank Haiko M\"uller for a number of helpful discussions.

\end{document}